\documentclass[a4paper,11pt,leqno]{article}

\usepackage[T1]{fontenc}
\usepackage{amsmath}
\usepackage{amsfonts}
\usepackage{amssymb}
\usepackage{amsthm}
\usepackage{mathabx}
\usepackage{svg}
\usepackage{tikz-cd}
\tikzcdset{arrow style=math font} 
\usepackage{textcomp}
\usetikzlibrary{shapes.geometric}
\usetikzlibrary{positioning}

\usepackage{ifpdf}
\usepackage{ifthen}
\usepackage{enumerate}
\usepackage{aliascnt}
\usepackage{color}
\usepackage{hyperref} 
\hypersetup{
	colorlinks=true, breaklinks, linkcolor=blue, filecolor=magenta, 
	urlcolor=blue,
	citecolor=magenta, linktoc=all, }

\usepackage[nameinlink]{cleveref}
\crefformat{enumi}{#2\textup{#1}#3}

\usepackage[all]{xy}
\usepackage{url}
\usepackage{pifont}
\usepackage{enumitem}
\usepackage{tikz}
\usepackage{footnote}
\usepackage{multirow}
\usepackage{geometry}

\reversemarginpar

\usepackage{footnote}

\usepackage{multirow}

\graphicspath{{./figures/}}

\numberwithin{equation}{subsection}
\numberwithin{figure}{subsection}
\numberwithin{table}{subsection}

\newaliascnt{thm}{equation}
\newtheorem{thm}[thm]{Theorem}
\aliascntresetthe{thm}
\Crefname{thm}{Theorem}{Theorems}

\Crefname{thmA}{Theorem}{Theorems}

\newaliascnt{lemma}{equation}
\newtheorem{lemma}[lemma]{Lemma}
\aliascntresetthe{lemma}
\Crefname{lemma}{Lemma}{Lemmas}

\newaliascnt{prop}{equation}
\newtheorem{prop}[prop]{Proposition}
\aliascntresetthe{prop}
\Crefname{prop}{Proposition}{Propositions}

\newaliascnt{cor}{equation}

\aliascntresetthe{cor}
\Crefname{cor}{Corollary}{Corollaries}

\theoremstyle{definition}

\newaliascnt{rem}{equation}

\aliascntresetthe{rem}
\Crefname{rem}{Remark}{Remarks}

\newaliascnt{example}{equation}

\aliascntresetthe{example}
\Crefname{example}{Example}{Examples}

\newaliascnt{definition}{equation}
\newtheorem{definition}[definition]{Definition}
\aliascntresetthe{definition}
\Crefname{definition}{Definition}{Definitions}

\newaliascnt{question}{equation}

\aliascntresetthe{question}
\Crefname{question}{Question}{Questions}

\newaliascnt{notation}{equation}
\newtheorem{notation}[notation]{Notation}
\aliascntresetthe{notation}
\Crefname{notation}{Notation}{Notations}

\newaliascnt{block}{equation}
\newtheorem{block}[block]{}
\aliascntresetthe{block}
\Crefname{block}{Block}{Blocks}

\newaliascnt{remark}{equation}
\newtheorem{remark}[remark]{Remark}
\aliascntresetthe{remark}
\Crefname{remark}{Remark}{Remarks}

\crefname{equation}{eq.}{eqs.}
\Crefname{figure}{Fig.}{Fig.}
\crefname{table}{Table}{Tables}
\crefname{subsection}{subsection}{subsections}
\Crefname{subsection}{Subsection}{Subsections}

\crefformat{block}{#2#1#3}
\Crefformat{block}{#2#1#3}

\crefformat{enumi}{#2#1#3}
\crefrangeformat{enumi}{(#3#1#4-#5#2#6)}

\newcommand{\C}{\mathbb{C}}

\newcommand{\Z}{\mathbb{Z}}
\newcommand{\R}{\mathbb{R}}

\newcommand{\W}{\mathcal{W}}

\newcommand{\V}{\mathcal{V}}
\newcommand{\A}{\mathcal{A}}
\newcommand{\E}{\mathcal{E}}

\renewcommand{\O}{\mathcal{O}}
\newcommand{\m}{\mathfrak{m}}

\renewcommand{\Im}{\mathop{\mathrm{Im}}}

\newcommand{\sign}{\mathop{\mathrm{sign}}}

\newcommand{\Hom}{\mathop{\rm Hom}\nolimits}

\renewcommand{\mod}{\mathop{\rm mod}\nolimits}

\newcommand{\Yro}{Y^{\mathrm{ro}}}

\newcommand{\Sinv}{S^{\mathrm{inv}}}

\newcommand{\Ainv}{A^{\mathrm{inv}}}
\newcommand{\Siinv}{\Sigma^{\mathrm{inv}}}
\newcommand{\xiinv}{\xi^{\mathrm{inv}}}

\newcommand{\Uro}{U^{\mathrm{ro}}}

\newcommand{\tU}{\tilde{U}}
\newcommand{\tu}{\tilde{u}}
\newcommand{\tv}{\tilde{v}}
\newcommand{\talpha}{\tilde{\alpha}}
\newcommand{\tbeta}{\tilde{\beta}}
\newcommand{\tr}{\tilde{r}}
\newcommand{\ts}{\tilde{s}}

\newcommand{\xiro}{\xi^{\mathrm{ro}}}

\newcommand{\fro}{f^{\mathrm{ro}}}

\newcommand{\Dro}{D^{\mathrm{ro}}}

\newcommand{\Droc}{D^{\mathrm{ro}, \circ}}

\newcommand{\Tu}{{\mathrm{Tub}}}

\newcommand{\hot}{{\mathrm{h.o.t.}}}
\newcommand{\lot}{{\mathrm{l.o.t.}}}

\newcommand{\Zmax}{Z_{\mathrm{max}}}

\newcommand{\Var}{\mathop{\rm Var}\nolimits}

\newcommand{\ord}{\mathrm{ord}}

\newcommand{\set}[2]{\left\{ #1 \,\middle\vert\, #2 \right\}}
\newcommand{\gen}[2]{\left\langle #1 \,\middle|\, #2 \right\rangle}

\renewcommand{\epsilon}{\varepsilon}

\newcommand{\Rep}{\mathcal{R}}
\newcommand{\Blro}{\textrm{Bl}^{\mathrm{ro}}}

\usepackage{geometry}

\newcounter{dummy}
\renewcommand{\thedummy}{\roman{dummy}}
\crefformat{dummy}{#2(#1)#3}
\crefrangeformat{dummy}{(#3#1#4-#5#2#6)}

\title{On the integral variation map of isolated plane curve singularities}
	
\author{Pablo Portilla Cuadrado\thanks{The first author is supported by 
RYC2022-035158-I, funded by MCIN/AEI/10.13039/501100011033 and by the FSE+}	 
\and Baldur 
		Sigur{\dh}sson\thanks{The second
	 	author was partly supported by Contratos María Zambrano para la 
	 	atracción de talento
	 	internacional at Universidad Complutense de  Madrid and
	 	Universidad Politécnica de Madrid and the
Spanish grant of MCIN
(PID2020-114750GB-C32/AEI/10.13039/501100011033).}}
\date{\today}

\begin{document}

	\maketitle

\begin{abstract}
	The integral variation map and algebraic monodromy of isolated plane curve 
	singularities are important homological invariants of the singularity which 
	are still far from being completely understood. This work provides 
	effective ways of computing them with respect to an explicit geometric 
	basis of the homology. 
	
	For any given topological type of plane curve singularity, we construct
	an analytic model of it, along with a vector field on 
	our version of its A'Campo space. This 
	vector field is tangent to the Milnor fibers at radius zero and the union 
	of the stable manifolds of their singularities yields a spine of
    each fiber, which can be described explicitly.
    This is very much inspired by a recent work of the authors.
	
	Our first main contribution is the algorithmic computation of the algebraic 
	monodromy and integral variation map as matrices with explicit bases for 
	any Milnor fiber in the Milnor fibration, not 
	merely congruence classes. 
	 For our second contribution, we introduce gyrographs which are graphs 
	 equipped with angular data and 
	rational weights.
    We prove that the invariant spine naturally carries 
	a gyrograph structure were the weights are given by the Hironaka 
	numbers, and that this structure recovers
    the geometric monodromy as a homotopy class,
    as well as the integral variation map.
	This provides a combinatorial framework for computation by hand.
    Our methods are further implemented in a  publicly available
    computer program written in Python.
\end{abstract}

	\tableofcontents

	\section{Introduction}

The topological type of an isolated plane curve
singularity $(C,0) \subset (\C^2,0)$,
defined by a germ $f \in \C\{x,y\}$ can be encoded in the minimal
resolution graph. The topological information can be also fully recovered from 
the Milnor fibration, that is, from the locally trivial fibration
\[f^{-1}(\partial D_\delta) \cap B_\epsilon \to \partial D_\delta\]
given by the restriction of $f$ to the preimage of a small circle intersected 
with a small enough sphere, with $0 < \delta <<\epsilon$. Its fiber $F$ 
is called the Milnor fiber, and its characteristic mapping class 
\[
  G:F \to F
\] 
is 
called the {\em geometric monodromy}.
Its induced action at the homological level is the algebraic monodromy 
associated with 
the Milnor fibration
\[
  G_1: H_1(F; \Z) \to H_1(F; \Z).
\]
The singularity being isolated allows one to consider a well defined geometric 
monodromy which acts by the identity on the boundary $\partial F$ on the Milnor 
fiber which, in turn, allows for the definition of the variation map with 
integral coefficients
\begin{align*}
	\Var:H_1(F, \partial F; \Z) &\to H_1(F;\Z) \\
	[a] &\mapsto \left[G(a)-a\right]
\end{align*}
where $a$ is any relative cycle representing $[a]$.

Both the algebraic monodromy and the variation map are homological algebraic 
invariants of the embedded topological type of a
plane curve.
The variation map is equivalent to the Seifert form, and it
determines the algebraic monodromy.
There does not exist a satisfactory classification for these
algebraic objects. 
Indeed, there is no good generalization of the (generalized) Jordan
canonical form of a linear endomorphism on a finitely generated
free $\Z$-module,
nor is there any classification of integral bilinear (and non-symmetric)
forms which encode the variation map.
This work builds towards the objective of better understanding these objects.

Results in this direction include
Selling-reduction \cite{selling_binary} which was used by
Kaenders to prove that the Seifert form determines the intersection
numbers between branches of a curve singularity \cite{kaenders_seifert}.
Also, explicit description of the Seifert form is obtained in 
\cite{EisNeu} for a class of links including algebraic ones,
as well as by A'Campo
\cite{acampo_groupe_i,acampo_groupe_ii,ACampo_def,Acam_gord}
and Gusein-Zade \cite{GZ_int} using divides.

It is known that the algebraic monodromy determines the number of branches,
and that in the case of one branch, the algebraic monodromy, in fact,
its characteristic polynomial, determines its topological type fully.
This is no longer the case for two or more branches.

Examples of pairs of plane curve singularities having the same algebraic 
monodromy but distinct topological
types, are given in \cite{Grima_mon_rat}, using rational coefficients, and
in \cite{Mich_Web_sur} using integral coefficients.
In \cite{Bois_Mich_Seif}, du Bois and Michel describe pairs of plane curves
which are topologically distinct, but have the same Seifert form.
In particular, the integral variation map does not distinguish them.
Deep results in the direction of classifying these homological
invariants include
\cite{Mich_Web_sur, Mich_Web_Jor,Bois_Robin,Bois_Mich_filt}.

In order to investigate these homological
invariants of a plane curve singularity, the algebraic monodromy and
the variation map, we construct explicitly a suitable
analytic model of the topological type, as well as
a spine in its \emph{invariant A'Campo space}. In  \cite{PSspine} we already 
constructed a spine on the Milnor fibers at radius $0$ as the union of the 
stable manifolds associated with a suitable vector field that arises  as a 
limit-rescaled version of the complex gradient of $f$. The spine in this work 
is a combination of
A'Campo's classical construction of the Milnor fibration at radius
zero by real oriented blow up, and interpolating pieces described in
\cite{ACampo}, and our work \cite{PSspine}. Here \emph{invariant}
means collapsing certain disks in each invariant Milnor fiber at radius zero.
The \emph{invariant spine} in this realization of the Milnor fibration
consists of an embedded graph in each Milnor fiber.
Vertices and edges of this graph are singularities and trajectories of
a specific vector field on the invariant Milnor fiber which is similar to that 
obtained in \cite{PSspine} but slightly modified for computational convenience.
This way we get to one of the {\em main contributions} of the present work: for 
an arbitrary
topological type, we compute explicitly the action of the monodromy of
the chain complex associated with a finite graph,
as well as a map from the dual chain complex
which represents the variation map, with integral coefficients. Our methods are 
effective and produce, not just congruence classes of matrices but an explicit 
description of the basis in which these matrices are computed.

Our methods differ from the existing literature in that the calculation
of integral matrices representing the algebraic monodromy and variation
map are done at the level of the chain complex, obtained by representing
the Milnor fiber as a handle body, with $\mu + m_0 - 1$ handles of
index $1$, and $m_0$ handles of index zero, where $m_0$ is the multiplicity.
Handle body decompositions of this type, and reductions of
them by handle cancellations, have appeared in the algebraic
and topological study of hypersurface singularities
in \cite{le_topological_use,le_perron,Tei,bko}.

This handle body decomposition is strongly related to
\emph{L\^e polyhedra} \cite{le_poly,le_menegon_poly},
which inspired our construction of the spine.
They cannot, however, be
defined over the whole base space of the Milnor fibration, as this
would imply periodicity for the monodromy.
The spine, in contrast, is well defined for every angle, at the cost
of not giving a locally trivial family. As a result, however, this
allows us to study the monodromy acting directly on the cells in the
spine.

These results are used in two ways.
Firstly, they are implemented in computer code, written in
Python, and relying on Singular for calculating resolutions, which allows for 
fast computation of the integral monodromy and variation map. This code is 
publicly available in a GitHub repository in \cite{PS_code}.
Secondly, we introduce the notion of a \emph{gyrograph}, which is an extra
structure on a graph that can be understood as a version of 
t\^ete-\`a-t\^ete graphs (see \cite{AC_tat} for the original definition).
This structure recovers completely the algebraic monodromy and
variation map, with integral coefficients.
Rather than introducing a metric on each edge, this structure requires
angles between adjacent edges at certain vertices of the graph, as well
as a filtration, by putting rational weights on each vertex.
The invariant spine introduced 
here is 
naturally
endowed with the structure of a gyrograph,
the weights being given essentially by the \emph{Hironaka numbers},
introduced in \cite{LMW_hir}, which are
ratios of divisorial valuations of the defining equation of the plane curve,
and a generic linear function. This result is materialized in 
\Cref{thm:gyro}. Therefore, this approach further underlines
the importance of the role of Hironaka numbers in the investigation
of the monodromy of a plane curve.
The gyrograph construction makes feasible the calculation of the
algebraic monodromy and variation map, with integral coefficients, by hand.
	
\subsubsection*{Outline of the paper}

The paper is structured to go from the purely combinatorial data of the 
resolution graph to the analytic geometry of the singularity, and finally to 
the algebraic computation of its invariants. We begin in \Cref{s:preparation} 
by establishing the requisite combinatorial framework. We define the resolution 
graph $\Gamma$ not merely as a topological object, but as a directed graph 
rooted at the blow-up of the origin. We introduce the notion of a ``nice'' 
total ordering on the branches of the curve, which allows us to induce a unique 
total order on the vertices of the graph. This ordering is crucial for the 
subsequent consistent labeling of charts and gluing maps.

In \Cref{s:analytic_model}, we translate this combinatorial data into a 
concrete geometric object. We construct a complex manifold $Y$ by 
gluing local analytic charts $(U_i, \tilde{U}_i)$ according to the directed 
edges and labeling functions defined previously. Within this analytic model, we 
define the function $f$ and a regular sequence of parameters, 
ensuring that the divisor structure matches the resolution graph. This 
construction provides the local coordinates necessary for our later 
dynamical analysis.

The topological realization of the Milnor fibration is developed in 
\Cref{ss:invariant_milnor}. We perform the real oriented blow-up of our 
analytic model to obtain a manifold with corners. A central object of our 
study, the \emph{invariant A'Campo space} $\Ainv$, is constructed here as a 
union of Seifert-fibered spaces over the real blow-up of the exceptional 
divisors, connected by interpolating thickened tori. This space serves as the 
domain for the geometric monodromy and the support for the vector field that 
drives our homological calculations.

The dynamical core of the paper is  in \Cref{ss:flow_description} and 
\Cref{s:dynamics_vector_field}. We define a horizontal vector field $\xiinv$ on 
$\Ainv$ by lifting  gradient fields defined on the exceptional 
divisors. We provide a detailed local analysis of this vector field, 
characterizing its singular locus, which consists of non-degenerate saddle 
points and specific multipronged singularities arising from ``dead branches'' 
(bamboos ending at valency $1$ vertices) of 
the resolution graph. This analysis follows closely that of \cite{PSspine}. A 
key analytic result is established in 
\Cref{thm:finite_nongeneric}, where we prove that the set of ``non-generic'' 
angles (those values in the circle $\R/2\pi\Z$ for which saddle connections 
occur) is finite, thereby ensuring the stability of our topological models for 
generic angles. 

In \Cref{s:mon_var}, we transition from dynamics to algebra. We construct a 
$1$-dimensional chain complex $(C_\theta, d_\theta)$ generated by the stable 
manifolds of the vector field, and a dual complex $C^\vee_\theta$ generated by 
the unstable manifolds. We prove that these are quasi-isomorphic to the 
absolute and relative singular chain complexes of the Milnor fiber, 
respectively. This identification allows us to define the algebraic 
monodromy matrix $B_\theta$ (at the level of the chain complexes) explicitly in 
terms of the intersection numbers of 
the flow trajectories on the invariant A'Campo space.

Building on this algebraic framework, \Cref{s:variation_operator} focuses on 
the variation operator. We define a linear map $V$ at the chain level, acting 
from the dual complex to the primal complex. We prove that this operator 
induces the classical variation map on homology, providing a bridge between our 
Morse-theoretic construction and classical singularity theory.

On \Cref{s:gyrographs} introduces the theory of ``gyrographs'' as a 
computational application of our results.  We show that the invariant spine of 
the fiber, when equipped with weights derived from the Hironaka numbers of the 
resolution, forms a gyrograph (\Cref{thm:gyro}).

Finally, on \Cref{s:software} we explain briefly how to use the implemented 
software \cite{PS_code} that realizes the computation of algebraic monodromy 
and variation 
operator.
	
	\section{Preparation} \label{s:preparation}
	
	Let $(C,0)\subset (\C^2,0)$ be a plane curve defined by a germ $f \in
	\O_{\C^2,0}$.

	\subsection{Initial data} \label{ss:intial_data}
	
	We take as initial data the resolution graph $\Gamma$ of a plane curve
	 $(C,0) \subset (\C^2, 0)$ together with a total order of the branches 
	 $C_1, \ldots, C_r$.
	
	Let $\pi_0:(Y_0, D_0) \to (\C^2, 0)$ be the blow up of  $\C^2$ at the 
	origin. And let $\tilde{C}$ be the strict transform of $C$ in $Y_0$.  Let 
	$t$ be the number of distinct tangents of 
	$(C,0)$. In 
	this way, $\tilde{C}$ intersects $D_0$ in $t$ points $\{p_1, \ldots, 
	p_t\}$ and, since $Y_0$ is smooth, each germ $(\tilde{C},p_\ell)$ with
	$\ell \in \{1,\ldots,t\}$ is a plane curve whose branches correspond to
	a subset of the branches of $(C,0)$.

	\begin{notation}\label{not:dual_graph_rewritten}
		Let $Y \to \C^2$ be a resolution of the plane curve defined by $f$. We 
		define the 
		associated dual graph $\Gamma$.
		
		\paragraph{Vertices of $\Gamma$}
		The set of vertices of $\Gamma$, denoted $\W$, is the disjoint union of 
		two sets: $\W = \V \cup \A$.
		\begin{itemize}
			\item The first set is $\V = \{0, 1, \ldots, s\}$.
			\item The second set, $\A$, is in bijection with the set of 
			irreducible branches of the curve $(C,0)$. We write this curve as 
			the union of its branches: $(C,0) = \bigcup_{a\in \A} (C_a,0)$.
		\end{itemize}
		
		\paragraph{Divisors and Edges of $\Gamma$}
		Each vertex $w \in \W$ corresponds to a divisor $D_w$ in $Y$.
		For the vertices $a \in \A$, the divisor $D_a$ is the 
		strict transform of the corresponding branch $C_a$ in $Y$.
		
		The edges of $\Gamma$ are defined by the intersection of these 
		divisors. An edge $ij$ connects two distinct vertices $i, j \in \W$ if 
		and only if their associated divisors intersect: $D_i \cap D_j \neq 
		\emptyset$.
		In the case of such an intersection, we denote the intersection point 
		as $p_{ij} = D_i \cap D_j$.
		
		We denote by $\W_i$ the set of {\em neighbor vertices} of the vertex 
		$i$, that is, the vertices that are adjacent to $i$.
		\paragraph{Associated Divisor Sets}
		We define several collections of divisors.
		\begin{itemize}
			\item $D_\V = \bigcup_{i \in \V} D_i$, which we also denote as $D$.
			\item $D_\A = \bigcup_{a \in \A} D_a$, the union of all strict 
			transforms.
			\item $D_\W = \bigcup_{w \in \W} D_w = D_\V \cup D_\A$, the total 
			divisor associated with $\Gamma$.
		\end{itemize}
		For a single divisor $D_i$ (where $i \in \W$), we also define the 
		subset $D_i^\circ$ which consists of the points in $D_i$ that do not 
		lie on any other divisor $D_j$ (for $j \neq i$):
		\[
		D_i^\circ = D_i \setminus \bigcup_{j\in \W\setminus \{i\}} D_j.
		\]

		\paragraph{Graph Properties}
		The graph $\Gamma$ is a tree, as 
		established in \cite[Section 3.6]{Wall_ctc}.
		A direct consequence of $\Gamma$ being a tree is that 
		for any two vertices $j, k$ in its vertex set, there 
		exists a unique shortest path connecting them. This unique path is also 
		the {\em 
		geodesic}, meaning it is the path with the smallest possible number of 
		edges.
	\end{notation}
	
	\subsubsection*{The dual graph as a directed graph}
	
	We endow $\Gamma$ with the structure of a directed graph and recall (see 
	\cite[Section 3]{PSspine}) the 
	definition of the maximal cycle.
	
	\begin{definition} \label{def:max_cycle}
		Denote by $\Zmax$ the
		\index{maximal cycle}
		\index{$c_{0,i}$}
		\emph{maximal cycle}
		in $Y$, and denote by
		$c_{0,i}$ its coefficients. Thus,
		\[
		\Zmax = \sum_{i\in \V} c_{0,i} D_i.
		\]
		and we have $c_{0,i} = \ord_{D_i}(\ell)$, where $\ell:\C^2\to\C$ is a 
		generic linear function. Here {\em generic} means that the strict 
		transform of $\{\ell=0\}$ intersects $D_0^\circ$, or equivalently, that 
		$\{\ell=0\}$ is not a tangent of $C$ at the origin.
		
		Since $c_{0,i} = \ord_{D_i}(\ell)$ we can similarly define these 
		numbers for arrowheads, that is for $i \in \A$. In this case $c_{0,i}= 
		\ord_{D_i}(\ell) = 0$.
	\end{definition}

	\begin{definition} \label{def:directed}
		The graph $\Gamma$ is seen as a directed graph as follows.
		Let $i,j$ be neighbors in $\Gamma$. The edge $ji$ is directed from
		$j$ to $i$ if and only if $i$ is further 
		from $0$ than $j$.
		
		A vertex $u \in \W_\Gamma$ is called an {\em ancestor} of a vertex 
		$v \in \W_\Gamma$ if there exists a directed path from $u$ to $v$.
		
		We denote by $\V_i^+$ the neighbors $k \in \V_i$ with $i \to k$.
	\end{definition}
	
\begin{definition} \label{def:Ga_branch}
	Let $i$ be a vertex in the directed graph $\Gamma$, as above. If $i\neq0$, 
	let $j$ be the unique neighbor of $i$ such that $j \to i$.
	The \emph{branch} (\cref{fig:branch}) 
	of $\Gamma$ at $i$ is 
	\begin{enumerate}
		\item $\Gamma$, if $i = 0$ and,
		\item otherwise, the connected component of $\Gamma$ with the edge $ji$ 
		removed, containing $i$.
	\end{enumerate}
The \emph{branch} of $\Gamma$ at $i$ is denoted by $\Gamma[i]$. We denote by 
$\A[i]$ the set of arrowheads in $\Gamma[i]$. We say that a branch is {\em 
dead} if it contains no arrowheads in it.
\end{definition}

\begin{figure}
	\centering
	\footnotesize
	\resizebox{0.4\textwidth}{!}{
\begingroup%
  \makeatletter%
  \providecommand\color[2][]{%
    \errmessage{(Inkscape) Color is used for the text in Inkscape, but the package 'color.sty' is not loaded}%
    \renewcommand\color[2][]{}%
  }%
  \providecommand\transparent[1]{%
    \errmessage{(Inkscape) Transparency is used (non-zero) for the text in Inkscape, but the package 'transparent.sty' is not loaded}%
    \renewcommand\transparent[1]{}%
  }%
  \providecommand\rotatebox[2]{#2}%
  \newcommand*\fsize{\dimexpr\f@size pt\relax}%
  \newcommand*\lineheight[1]{\fontsize{\fsize}{#1\fsize}\selectfont}%
  \ifx\svgwidth\undefined%
    \setlength{\unitlength}{84.95487748bp}%
    \ifx\svgscale\undefined%
      \relax%
    \else%
      \setlength{\unitlength}{\unitlength * \real{\svgscale}}%
    \fi%
  \else%
    \setlength{\unitlength}{\svgwidth}%
  \fi%
  \global\let\svgwidth\undefined%
  \global\let\svgscale\undefined%
  \makeatother%
  \begin{picture}(1,0.568551)%
    \lineheight{1}%
    \setlength\tabcolsep{0pt}%
    \put(0,0){\includegraphics[width=\unitlength,page=1]{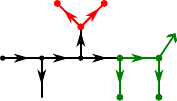}}%
    \put(0.4932798,0.35985375){\color[rgb]{1,0,0}\makebox(0,0)[lt]{\lineheight{1.25}\smash{\begin{tabular}[t]{l}$i$\end{tabular}}}}%
    \put(0.64316869,0.27830198){\color[rgb]{0,0.50196078,0}\makebox(0,0)[lt]{\lineheight{1.25}\smash{\begin{tabular}[t]{l}$j$\end{tabular}}}}%
    \put(0,0){\includegraphics[width=\unitlength,page=2]{branch.pdf}}%
  \end{picture}%
\endgroup%
}
	\caption{The resolution graph of a plane curve with its directed structure. 
	In green the branch $\Gamma[j]$ of $\Gamma$ at $j$ and in red, $\Gamma[i]$.}
	\label{fig:branch}
\end{figure}

\begin{definition}\label{def:nice_total_order}
	A total order $>$ of the set of branches $\{C_1, \ldots, C_r\}$ of a 
	plane curve is said to be {\em nice} if the following is satisfied
	
	\begin{enumerate}
	\item \label{it:nice_total_order_i} When $r=1$ the unique total order is nice
	\item \label{it:nice_total_order_ii}If $C_a$ and $C_b$ have the same 
	tangent and it is different from the tangent of $C_c$, then $C_a > C_c$ 
	if and only if $C_b > C_c$.
	\item \label{it:nice_total_order_iii} The total order induced on the plane curves $(\tilde{C},p_\ell)$ 
	is nice for all $\ell \in \{1, \ldots, t\}$.  
	\end{enumerate}
\end{definition}

\begin{remark}\label{rem:nice_total_order}
	\begin{enumerate}
	\item By \cref{it:nice_total_order_ii}, a nice total order induces a total 
	order on the set of tangents. 
	
	\item \label{it:rem_nice_total_order_converse} Conversely, if we equip each 
	of the plane curves $(\tilde{C},p_\ell)$  
	with a nice total order, then, any total order of the tangents of 
	$(C,0)$ induces a nice total order on the set of branches of $(C,0)$.
	\end{enumerate}
\end{remark}

\begin{lemma}\label{lem:nice_total_order}
	Any plane curve $(C,0)$ admits a nice total order.
\end{lemma}
\begin{proof}
	We use induction on the number $\sigma$ of blow ups needed for the 
	minimal embedded resolution.
	If $\sigma =0$, then $(C,0)$ is smooth and, in particular, it has one 
	single branch and 
	\Cref{def:nice_total_order} \cref{it:nice_total_order_i} yields the 
	result.
	
	If $\sigma > 0$, the induction hypothesis applies to the plane curves 
	$(\tilde{C}, p_\ell)$, $\ell=1, \ldots, t$. Thus, by 
	\Cref{rem:nice_total_order} \cref{it:rem_nice_total_order_converse}, it 
	suffices to choose a total order of the tangents of $(C,0)$. 
\end{proof}

\begin{notation}
	Let $<$ be a nice total order on the set of branches $\{C_1, \ldots, C_r\}$ 
	of a plane curve $(C,0)$. We write $\A[\ell] < \A[k]$ if $C_i < C_j$ for 
	every branch $C_i$ corresponding to an arrowhead in $\A[\ell]$ and every 
	branch $C_j$ corresponding to an arrowhead in $\A[k]$.
\end{notation}

\begin{lemma}\label{lem:total_order_vertices}
	Let $(C,0)$ be a plane curve with a nice total order $<$ on the set of its 
	branches $\{C_1, \ldots, C_r\}$.
	Then there exists a unique total order (that we also denote by $<$) on the 
	set of vertices and arrowheads
	$\W_\Gamma$ of the dual graph $\Gamma$ of the minimal resolution of $(C,0)$ 
	that satisfies the
	following two properties.
	\begin{enumerate}
		\item If $j\to i$, then $j < i$.
		\item Assume that $i\to k$ and $i \to \ell$ with $k \neq \ell$. 
		\begin{itemize}
			\item If 
			$\A[\ell] \neq \emptyset$ and  $\A[k] \neq \emptyset$ with
			$\A[\ell] < \A[k]$, then $u < v$ for every pair of vertices with $u 
			\in 
			\Gamma[\ell$] and $v \in \Gamma[k]$.
			
			\item If 
			$\A[\ell] = \emptyset$  then $u < v$ for every pair of vertices 
			with $u \in 
			\Gamma[\ell]$ and $v \in \Gamma[k]$.
		\end{itemize}
	\end{enumerate}
\end{lemma}
\begin{proof}
	We first show the existence by constructing the total order $<$. Let $u, v 
	\in 
	\W_\Gamma$ be two distinct vertices. Since $\Gamma$ is a directed tree 
	rooted at $0$, we have two mutually exclusive cases:
	\begin{enumerate}
		\item  If $u$ is an ancestor of $v$ (that is, if there
		is a directed path $u \to \ldots \to v$), we set $u < v$. 
		Analogously, if $v$ is an ancestor of $u$, we define $v < u$.
		
		\item  If neither $u$ nor $v$ is an 
		ancestor of the other, let $i = \text{LCA}(u,v)$ be their lowest common 
		ancestor. Then there exist distinct {\em children} $k, \ell$ of $i$ 
		such that 
		$i \to k$, $i \to \ell$, $u \in \Gamma[k]$, and $v \in \Gamma[\ell]$. 
		Let 
		$\A[k] 
		= 
		 \Gamma[k] \cap \A$ and $\A[\ell] = \Gamma[\ell] \cap \A$ be the 
		 non-empty sets of 
		branch descendants (identified with the corresponding arrowheads). The 
		given nice total order on $\A$ (denoted $<_\A$) 
		is in particular a total order, so it completely orders these 
		 sets.
		\begin{itemize}
			\item If $\A[k] <_\A \A[\ell]$ (that is, if $\forall a \in \A[k], 
			\forall 
			b 
			\in \A[\ell], a <_\A b$), we define $u < v$.
			\item If $\A[\ell] <_\A \A[k]$, we define $v < u$.
		\end{itemize}
	\end{enumerate}
	This relation $<$ is total and anti-symmetric by construction. Transitivity 
	follows from a case by case analysis of the LCA's of pairs $(u,v)$, 
	$(v,w)$, and 
	$(u,w)$. This construction satisfies the two properties:
	\begin{enumerate}
		\item If $j \to i$, then $j$ is an ancestor of $i$, so $j < i$.
		\item If $i \to k$, $i \to \ell$, and $\A[\ell] <_\A \A[k]$, then for 
		any 
		$u \in \Gamma[\ell]$ and $v \in \Gamma[k]$, their LCA is $i$. And we 
		have already seen that 
		$\A[\ell] <_\A \A[k]$ implies $u < v$.
	\end{enumerate}
	
	To prove uniqueness, let $<'$ be any total order satisfying properties (i) 
	and (ii). Let $u, v \in \W_\Gamma$ be distinct.
	\begin{enumerate}
		\item If $u$ is an ancestor of $v$, there is a path $u = x_0 \to x_1 
		\to \ldots \to x_m = v$. By property (i) for $<'$, $x_j <' x_{j+1}$ for 
		all $j$. By transitivity of $<'$, $u <' v$. This matches $<$.
		\item If $u, v$ are not ancestors, let $i = \text{LCA}(u,v)$ with $u 
		\in \Gamma[k]$ and $v \in \Gamma[\ell]$. Assume $\A[\ell] <_\A \A[k]$. 
		Property 
		(ii) states that $u' <' v'$ for every $u' \in \Gamma[\ell]$ and $v' \in 
		\Gamma[k]$. Thus, $v <' u$. This also matches $<$. The case $\A[k] <_\A 
		\A[\ell]$ is analogous.
	\end{enumerate}
	Since $<'$ must agree with $<$ in all cases, the total order is unique.
\end{proof}

\begin{definition}\label{def:label_function}
	For each vertex $i\neq 0$ with $j \to i$. We define a {\em labeling 
	function} $h_i$ on the set of neighbors of $i$ other than $j$. We consider 
	two cases and in each of the cases the labeling function preserves the 
	total order of $\W$. 
	
	If $i$ has a nearby dead branch, then
	\[
	h_i: \V_i \setminus \{j\} \to \{0, \ldots, |\V^+_i| - 2\}
	\]
	and, in particular, $h_i(k)=0$ if $k$ is on the dead branch. Otherwise,
	\[
	h_i: \V_i \setminus \{j\} \to \{1, \ldots, |\V^+_i| - 1\}.
	\]
\end{definition}

\begin{figure}
	\centering
	\small
	\resizebox{0.7\textwidth}{!}{
\begingroup%
  \makeatletter%
  \providecommand\color[2][]{%
    \errmessage{(Inkscape) Color is used for the text in Inkscape, but the package 'color.sty' is not loaded}%
    \renewcommand\color[2][]{}%
  }%
  \providecommand\transparent[1]{%
    \errmessage{(Inkscape) Transparency is used (non-zero) for the text in Inkscape, but the package 'transparent.sty' is not loaded}%
    \renewcommand\transparent[1]{}%
  }%
  \providecommand\rotatebox[2]{#2}%
  \newcommand*\fsize{\dimexpr\f@size pt\relax}%
  \newcommand*\lineheight[1]{\fontsize{\fsize}{#1\fsize}\selectfont}%
  \ifx\svgwidth\undefined%
    \setlength{\unitlength}{376.35747474bp}%
    \ifx\svgscale\undefined%
      \relax%
    \else%
      \setlength{\unitlength}{\unitlength * \real{\svgscale}}%
    \fi%
  \else%
    \setlength{\unitlength}{\svgwidth}%
  \fi%
  \global\let\svgwidth\undefined%
  \global\let\svgscale\undefined%
  \makeatother%
  \begin{picture}(1,0.67526043)%
    \lineheight{1}%
    \setlength\tabcolsep{0pt}%
    \put(0,0){\includegraphics[width=\unitlength,page=1]{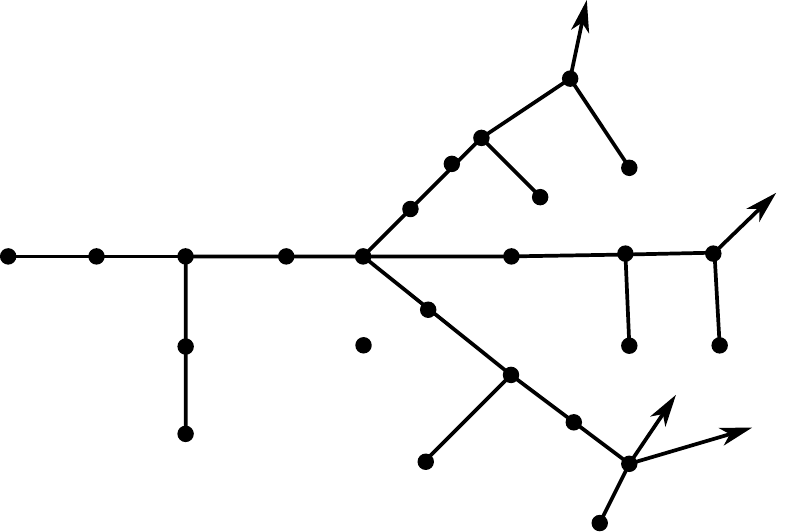}}%
    \put(0.94207016,0.09945551){\color[rgb]{0,0,0}\makebox(0,0)[lt]{\lineheight{1.25}\smash{\begin{tabular}[t]{l}$1$\end{tabular}}}}%
    \put(0.86627996,0.15291248){\color[rgb]{0,0,0}\makebox(0,0)[lt]{\lineheight{1.25}\smash{\begin{tabular}[t]{l}$2$\end{tabular}}}}%
    \put(0.94912961,0.43158872){\color[rgb]{0,0,0}\makebox(0,0)[lt]{\lineheight{1.25}\smash{\begin{tabular}[t]{l}$3$\end{tabular}}}}%
    \put(0.75330305,0.65001065){\color[rgb]{0,0,0}\makebox(0,0)[lt]{\lineheight{1.25}\smash{\begin{tabular}[t]{l}$4$\end{tabular}}}}%
    \put(0,0){\includegraphics[width=\unitlength,page=2]{nice_total.pdf}}%
  \end{picture}%
\endgroup%
}
	\caption{A nice total order gives a natural embedding of the resolution 
	graph on the real plane $\mathbb{R}^2$. The numbers on the arrows indicate 
	the nice total order given on the set of branches as input data for a
	total order on the set of vertices and arrowheads as defined in 	
	\Cref{lem:total_order_vertices}.}
	\label{fig:nice_total}
\end{figure}

\subsubsection*{A special subgraph}

In \cite[Section 6]{PSspine} we defined a special subgraph of the dual graph 
of the minimal embedded resolution. The motivation underlying the definition of 
this graph was the characterization of the set of divisors to which we could 
extend (a rescaling of) the vector field $- \nabla \log |f|$. For more on this, 
we refer the reader to the cited paper. Here we just give a direct definition.

\begin{definition}\label{def:upsilon}
	Let
	\index{$\Upsilon$}
	$\Upsilon \subset \Gamma$ be the smallest connected subgraph of
	$\Gamma$ 
	containing the vertex corresponding to the first blow-up $0 \in \V$, 
	as well as any vertex in
	$\V$ adjacent to an arrow-head $a \in \A$.
	Let $\V_\Upsilon \subset \V$ be the vertex set of $\Upsilon$. We call the 
	vertices of $\Upsilon$ {\em invariant vertices}.
\end{definition}

\section{Construction of an analytic model}
\label{s:analytic_model}

In the previous section, we established the purely combinatorial
data of the singularity, consisting of the dual graph $\Gamma$
equipped with a directed structure and a total order on its
vertices $\W_\Gamma$.

We now proceed to construct a geometric object that realizes
this combinatorial data. We will build a complex manifold $Y$
by gluing together simple analytic charts ($U_i, \tU_i$). The
gluing process itself will be dictated by the directed edges
and the labeling function $h_i$ from \Cref{def:label_function}.
This manifold $Y$ will serve as an explicit model for the
resolution space of a singularity.

\subsection{Analytic charts}
\label{ss:analytic_chart}
For each $i \in \V$, define charts
\[
  U_i = \C \times D^2_\delta,\qquad
  \tilde U_i = \C\times D^2_\delta,
\]
where $D^2_\delta \subset \C$ is a disk of a small radius $\delta$.
Let $\sim$ be the equivalence
relation on their disjoint union generated by
\begin{equation} \label{eq:Y_sim_i}
  U_i \ni (u,v)
  \sim
  (u v^{b_i},v^{-1}) \in \tilde U_i, \quad
  u \in D^2_\delta,\: v \in \C^*,
\end{equation}
for any $i \in \V$
\begin{equation} \label{eq:Y_sim_ij}
  \tU_i \ni (u,v)
  \sim
  (v, u+h_j(i)) \in U_j, \qquad
  u,v \in D^2_\delta,
\end{equation}
for any edge $j \to i$. We then set
\[
  Y = \coprod_{i \in \V} \left( U_i \amalg \tU_i \right) / \sim.
\]
Then $Y$ is a complex manifold with charts $U_i$ and $\tU_i$
with coordinates $u_i,v_i$ and $\tu_i, \tv_i$.

Furthermore, we have a compact rational curve $D_i \subset Y$
for each $i \in \V$ defined by $u_i = 0$ in $U_i$
and $\tu_i = 0$ in $\tU_i$.
It follows from \cref{eq:Y_sim_i} that
the Euler number of the normal bundle of $D_i$ in $Y$ is
$-b_i$. Also, $D_i$ and $D_j$ intersect precisely when $j i$ is an edge\
in $\Gamma$, by \cref{eq:Y_sim_ij}.
Therefore, the dual graph to this configuration of curves
is $\Gamma$.
As a result, using the Castelnuovo criterion,
$Y$ blows down to a smooth surface, which we call $Y_{-1}$.
We get a map
\[
  \pi: Y \to Y_{-1}
\]
which is a modification of the surface $Y_{-1}$. Let $o \in Y_{-1}$ be the 
image of $D_\V$ by $\pi$.

If $j\to i$ is a directed edge, then,
in the coordinates $u_j,v_j$, the set $U_j\cap D_i$ is given by
$v_j = h_j(i)$.
We also have a noncompact curve $D_a \subset Y$ given by
$v_i = h_i(a)$ for each $a \in \A$ adjacent to $i \in\V$.
Denote by $C_a$ the image of $D_a$ in $Y_{-1}$ under $\pi$.
Thus $C_a$ is the image of a good parametrisation and therefore, a  complex 
curve in $Y_{-1}$ by 
\cite[Lemma 2.3.1]{Wall_ctc} . Then there exists an 
\[
  f_a \in \O_{Y_{-1},o}
\]
such that the germ $(C_a,o)$ with its reduced analytic structure is the 
hypersurface germ defined by $f_a$ at $o\in 
Y_{-1}$, in other words
\[
(Z(f_a), o) = (C_a,o).
\]
We define $f \in \O_{Y_{-1},o}$ by 
\begin{equation} \label{eq:f}
  f = \prod_{a \in \A} f_a^{m_a}.
\end{equation}

\subsection{A regular sequence}
\label{ss:regular_sequence}
The sets $\tU_0 \setminus U_0$ and $U_0 \setminus \tU_0$ are curvettes
to the divisor $D_0$. By the same argument as above, there exist
functions $x,y\in \O_{Y_{-1},o}$ such that the germs given by
$x=0$ and $y=0$ are reduced, and their strict transforms in $Y$ are these
curvettes which intersect  $D^\circ_0$ transversely. It follows that $x,y$ 
represents a basis of
$\m_{Y_{-1},o}^{\phantom{2}} / \m^2_{Y_{-1},o}$,
 i.e. $x,y$ is a regular sequence
of parameters in $\O_{Y_{-1},o}$. As a result, we have ab isomorphism
\begin{equation} \label{eq:xy}
  \O_{Y_{-1},o} \simeq \C \{x,y\}.
\end{equation}
By construction, the coordinate axis $x=0$ and $y=0$ are not tangent
to the curve $(C,0)$. Therefore, if $i \in \V$, then
$\ord_{D_i}(x) = \ord_{D_i}(y) = c_{0,i}$ (recall \Cref{def:max_cycle}). In 
particular, the pullback $\pi^*x$ vanishes with order $1$ along
$D_0$, and the restriction of
\[
  \frac{ \pi^* x|_{U_0}}{u_0}
\]
to $\{u_0 = 0 \} = U_0 \cap D_0$
is a polynomial in $v$ which vanishes with order $c_{0,\ell}$ at $h_0(\ell)$.
Therefore, there exists an $a \in \C^*$ such that
\[
  \left.
  \frac{ \pi^* x|_{U_0}}{u_0}
  \right|_{U_0 \cap D_0}
  =
  a \prod_{\ell \in \V_0} (v_0 - h_0(\ell))^{c_{0,\ell}}.
\]
Similarly, since $\pi^*y|_{U_0}$ also vanishes with order one along $v_0 = 0$
there is a $b \in \C^*$ so that
\[
  \left.
  \frac{ \pi^* y|_{U_0}}{u_0}
  \right|_{U_0 \cap D_0}
  =
  b v_0 \prod_{\ell \in \V_0} (v_0 - h_0(\ell))^{c_{0,\ell}}.
\]
We will assume that $x,y$ are chosen in such a way that $a = b = 1$.
In particular,
\[
  \left.
  \frac{\pi^* y}{\pi^* x} 
  \right|_{U_0 \cap D_0}
  =
  v_0.
\]
In fact, since the fraction $y/x$ defines a meromorphic function
on $Y_{-1}$ whose only pole is along $\{x=0\}$, 
its pullback restricts to a holomorphic function on any $D_i$ for
$i \in \V \setminus \{0\}$. As a result, if $i \in \V\setminus \{0\}$,
then $i$ is on one of the branches $\Gamma[\ell]$ where $\ell$ is
a neighbor of $0$, and
\[
  \left.
  \frac{\pi^*y}{\pi^*x}
  \right|_{D_i}
  \equiv
  h_0(\ell).
\]

\subsection{Local expressions for $f$}
\label{ss:local_expressions}

Fix a vertex $i \in \V$ and let $\ell \in \V^+_i$, that is, an adjacent vertex 
with $i \to \ell$. By 
construction $\pi^*f$ has 
a zero 
of order $m_\ell$ (resp. $m_i$) along $D_\ell$ (resp. $D_i$). Therefore, the 
function 
$\pi^*f|_{U_i \cap \tU_\ell}$ expanded near the point $D_i\cap D_\ell = 
(0,h_i(\ell))$ 
on coordinates $u_i, v_i$, is of the 
form
\begin{equation} \label{eq:ali_hot}
	\pi^* f(u_i,v_i)|_{U_i \cap \tU_\ell}
	=
	u_i^{m_i} (v_i - h_i(\ell))^{m_\ell} (a_{i\ell} + \hot) 
\end{equation}
where $a_{i\ell} \in \C^*$. In coordinates
$\tu_\ell, \tv_\ell$, we have a similar expansion
\[
	\pi^* f(\tu_\ell,\tv_\ell)|_{U_i \cap \tU_\ell}
	=
	\tu_\ell^{m_\ell} \tv_\ell^{m_i} (a_{\ell i} + \hot).
\]
with $a_{\ell i} \in \C^*$. By \cref{eq:Y_sim_ij}, we find
\[
  a_{i\ell} = a_{\ell i}.
\]
Since $\pi^*f$ has a zero of order $m_i$ along $D_i$, the fraction
$\pi^*f|_{U_i} / u^{m_i}$ defines a holomorphic function in $U_i$.
Its restriction to $U_i\cap D_i = \{u_i = 0\}$ is a polynomial
in $v_i$ 
\begin{equation}\label{eq:fi_poly}
	\frac{\pi^*(f)|_{U_i}}{u^{m_i}}= a_{ij}\prod_{\ell \in \V_i^+} (v_i - 
	h_i(\ell))^{m_\ell}
\end{equation}
which has a zero of order $m_\ell$ at $v_i = h_i(\ell)$
for each $\ell \in \V_i^+$.
Let $a_0 \neq 0$ be the leading term of this polynomial, in the case
when $i = 0 \in \V$.
By replacing $f$ by $f/a_0$, we will, from now on, assume that this leading
term is $1$. As a result, in the case $i = 0$, we have
\begin{equation} \label{eq:a_zero_k}
  a_{0k} = \prod_{\substack{\ell \in\V_0 \\ \ell \neq k}}
           (h_0(k) - h_0(\ell))^{m_\ell},
  \quad
  k \in \V_0.
\end{equation}
If $i$ is any other vertex in $\V$ other that $0$, there exists a
$j \in \V$ so that $j \to i$. By the same argument, we find
\begin{equation} \label{eq:a_i_k}
  a_{ik} = a_{ij}\prod_{\substack{\ell \in\V_i^+ \\ \ell \neq k}}
           (h_i(k) - h_i(\ell))^{m_\ell},
  \quad
  k \in \V_i^+.
\end{equation}

\begin{definition}\label{def:Siell}
If $i\to k$ is an edge in $\Gamma$, let
\[
  S_{ik} = \set{ \ell \in \W_i } { \ell > k }
         = \set{ \ell \in \W_i } { h_i(\ell) > h_i(k) }.
\]
See left hand side of \cref{fig:s_ik_m_ik}.
\end{definition}

\begin{definition}\label{def:Miell}
Let $i \to k$ be an edge in $\Gamma$, and let
$0 = \ell_0 \to \ell_1 \to \cdots \to \ell_{s+1} = k$ be the geodesic
connecting $0$ and $k$ in $\Gamma$. In particularm $\ell_{s} = i$. We define 
\[
  M_{ik} = \sum_{ \substack{ 0 \leq r < s +1 \\ c \,\in 
		S_{\ell_{r\vphantom{+1}}\ell_{r+1}} }} m_c.
\]
Now let, $0 = \ell_0 \to \ell_1 \to \cdots \to \ell_s = i$ be the geodesic
connecting $0$ and $i$, we define
\[
M_{i} =\sum_{k \in \V_i^+} m_k +\sum_{ \substack{ 0 \leq r < s \\ c \,\in 
		S_{\ell_{r\vphantom{+1}}\ell_{r+1}} }} m_c.
\]
See right hand side of \cref{fig:s_ik_m_ik}.
\end{definition}

\begin{figure}[!ht]
	\centering
	\resizebox{0.9\textwidth}{!}{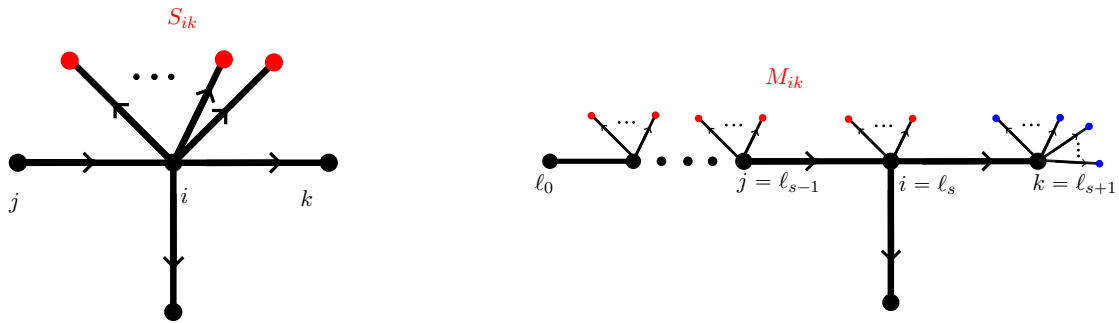}
	\caption{On the left hand side we see, in red, the vertices that are in 
	$S_{ik}$. On the right hand side we see: in red the vertices that 
	contribute to 
	$M_{ik}$; which together with those in blue are that the vertices 
	that contribute to $M_{k}$.}
	\label{fig:s_ik_m_ik}
\end{figure}

\begin{lemma} \label{lem:M_ik}
Let $i \to k$ be an oriented edge in $\Gamma$. Then,
 $a_{ik} \in \Z \setminus \{0\}$ and
\begin{equation} \label{eq:Mik}
  \sign(a_{ik})
  =
  (-1)^{M_{ik}}
\end{equation}

\end{lemma}

\begin{proof}
 Take signs on both sides of \cref{eq:a_i_k} 
which gives us
\[
\sign(a_{ik}) = \sign(a_{ij}) (-1)^{\sum_{h_i(\ell) > h_i(k)} m_\ell}.
\]
The result follows from applying the same formula to $\sign(a_{ij})$ and 
repeating the process iteratively along all the vertices in the geodesic $0 = 
\ell_0 \to \ell_1 \to \cdots \to \ell_{s+1} = k$ connecting $0$ and $k$ (with 
$\ell_{s}=i$).
\end{proof}

\begin{lemma}\label{lem:mi_equations}
	Let $q \in \Droc_{i,\theta}$ be such that its image 
	by $\Pi_i$ is real.
	Then, $\alpha_i$ satisfies the following equations depending on where 
	$\Pi_i(q)$ 
	lies:
	\[
	\begin{array}{llll}
		m_i \alpha_i(q) = \theta - M_i \pi & \text{ if } &- \infty <\Pi_{i}(q) 
		< 0 
		&\\
		m_i \alpha_i(q) = \theta - M_i \pi & \text{ if } &- \infty <\Pi_{i}(q) 
		< 1 
		& 
		\text{ 
			and } i \text{ does not have a neighbor in } \Gamma \setminus 
			\Upsilon\\
		m_i \alpha_i(q) = \theta - M_{i k} \pi  & \text{ if } & h_i(k) 
		<\Pi_{i}(q) < 
		h_i(k)+1 &  \text{ for } k \in \V_i^+\\
		m_i \alpha_i(q) = \theta - M_{i k} \pi  & \text{ if } & p_i +1 < 
		\Pi_{i}(q) 
		< 
		\infty &
	\end{array}
	\]
\end{lemma}

\begin{proof}
	By construction, $\arg(q) =\theta$, because $q 
	\in \Droc_{i,\theta}$. By 
	taking argument of both sides of \cref{eq:fi_poly}, we get
	\[
	\theta = m_i \alpha_i(q)  + \arg\left( a_{ij}\prod_{\ell \in \V_i^+} (v_i - 
	h_i(\ell))^{m_\ell} \right)
	\]
	which, by part \Cref{lem:M_ik} and \cref{eq:a_i_k}, is equivalent to
	\[
	\theta = m_i \alpha_i(q)  - \arg\left( (-1)^{M_{ik}} \right) = m_i 
	\alpha_i(q)  
	- M_{ik} \pi
	\]
	for points $q$ such that $\Pi_{i}(q)$ lies in between polar points. This 
	proves 
	the third line of the statement. The other $3$ lines follow from very 
	similar considerations.
\end{proof}

\subsection{Description of the invariant Milnor fibration}
\label{ss:invariant_milnor}
We now use this 
analytic model to describe the \emph{invariant Milnor fibration}.
This involves a topological construction: the real oriented
blow-up. In this section we construct a new space, the \emph{invariant A'Campo 
space} 
$\Ainv$, which is a union of Seifert-fibered
pieces over the real-blown-up divisors $\hat{D}_i$ with some interpolating 
thickened tori. This corresponds to the classical A'Campo space \cite{ACampo} 
after 
contracting some disks to points. The monodromy on these contracted disks is 
purely periodic.

\subsubsection*{The real oriented blow-up and the Milnor fibration at radius 
$0$}

We start by defining the real oriented blow-up of the resolution space along 
the exceptional divisors. Here we follow \cite[Section 4]{PSspine}.
Let $(Y,D) \to (\C^2,0) \simeq (Y_{-1},0)$ be the embedded resolution of our 
curve singularity. 
Denote by $\sigma_i:\Yro_i \to Y$
the
\index{real oriented blow-up}
\emph{real oriented blow-up}
of $Y$ along the submanifold $D_i \subset Y$
for $i \in \W$ constructed as follows.
If $U \subset Y$ is a chart with coordinates $u,v$ such that $u = 0$ is
an equation for $D_i \cap U$,
then we take a coordinate chart
$\Uro  \subset \Yro_i$ with coordinates $r,\alpha,v \in \R_{\geq 0} \times 
\R / 2 \pi\Z \times \C$, where $r$ and $\alpha$ are polar coordinates for 
$u$, that is:
\begin{equation} \label{eq:polar_coords}
	r = |u|:\Uro \to \R,
	\qquad
	\alpha = \arg(u):\Uro \to \R / 2\pi \Z.
\end{equation}
We can cover $D_i$ by such charts in order to define an atlas for $\Yro_i$. 
In each of these charts, the map $\sigma_i$ is given by 
$\sigma_i(r,\theta,v)= (re^{i\theta}, v)$.
Denote the fiber product of these maps by
\begin{equation}\label{eq:fiber_product}
	\sigma = \bigtimes_{i\in\W} \sigma_i:(\Yro,\Dro_\W) \to (Y,D_\W).
\end{equation}
The spaces $\Yro_i$ are manifolds with boundary, and the space $\Yro$ is
a manifold with corners and, as a topological manifold, its boundary 
$\partial \Yro$ coincides with $\Dro_\W$.

\begin{notation}\label{not:strat}
	The corners of $\Yro$ induce a stratification indexed by the graph
	$\Gamma$ as follows.
	Set $\Dro_\emptyset = \Yro \setminus \Dro_\W$.
	For $i,j \in \W$, we define the following subspaces of $\partial \Yro$:
	\index{$\Dro_i$}
	\[
	\Droc_i = \sigma^{-1}(D_i^\circ), \qquad
	\Dro_i  = \sigma^{-1}(D_i),\qquad
	\Dro_{i,j} = \sigma^{-1}(D_i \cap D_j).
	\]
	Note that $\sigma^{-1}(D_i)= \overline{\Droc_i}$.
\end{notation}

The {\em Milnor fibration at radius zero}  is the locally trivial topological 
fibration given by the map
\[
\arg(\fro)|_{\partial \Yro} : \partial \Yro \to \R / 2 \pi \Z.
\]

\begin{definition}\label{def:angled_ray}
	For $\theta \in \R/2\pi \Z$, we define the {\em Milnor ray} at angle 
	$\theta$, as  
	\[
	\Tu^*_\theta = \arg(f)^{-1}(\theta) \subset \Tu^*,
	\]
	as well as the subsets of $\Yro$:
	\begin{equation}
		\begin{aligned}
			\Yro_\theta &= \arg(\fro)^{-1}(\theta), \qquad 
			&\Dro_{i,\theta} &= \Yro_\theta \cap \Dro_i, \\
			\Droc_{i,\theta} &= \Yro_\theta \cap \Droc_i, 
			&\Dro_{i,j,\theta} &=  \Yro_\theta \cap \Dro_{i,j}.
		\end{aligned}
	\end{equation}
\end{definition}

\subsubsection*{The invariant A'Campo space}
\begin{definition}
For $i \in \V_\Upsilon$, let 
\[
\hat{\sigma}_i: \widehat D_i \to D_i
\]
 be the
real oriented blow-up of $D_i$ at intersection points
$D_i \cap D_\ell$ corresponding to adjacent edges $i\to \ell$
or $\ell \to i$ in $\Upsilon$ (see \cref{fig:hatD_i}). 
\end{definition}

\begin{notation}
	If $i$ has a neighbor in $\Gamma \setminus \Upsilon$ (equivalently if the 
	minimum of $h_i$ is $0$), denote by $q_{i,0} \in \hat{D}_i$ the unique 
	preimage point 
	$\hat{\sigma}_i^{-1}(\{v_i 
	=0\})$.
\end{notation}

 \begin{figure}
 	\centering
 	\Large
	\resizebox{0.7\textwidth}{!}{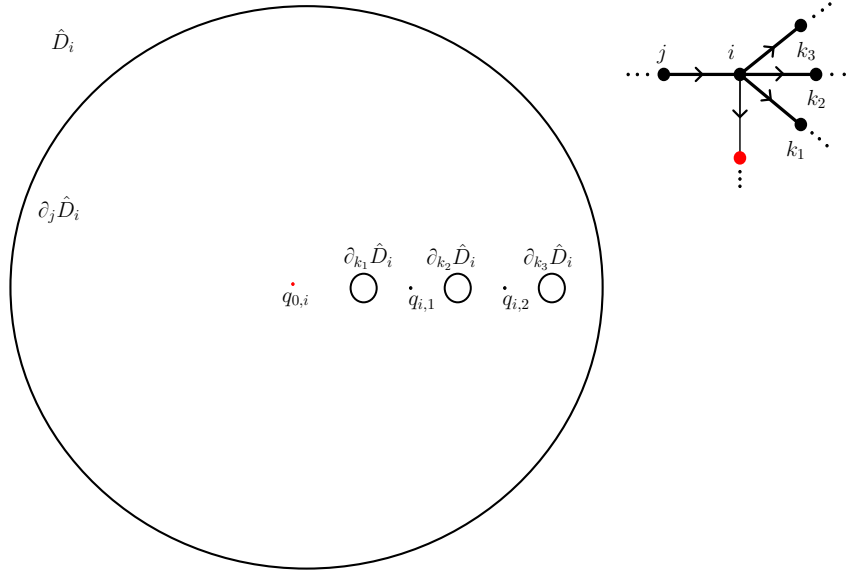}
	\caption{A representation of $\hat{D}_i$ corresponding to a vertex $i$ with 
	four neighbor invariant vertices $j,k_1,k_2$ and $k_3$ (with $j \to i$); 
	and one  vertex in $\Gamma \setminus \Upsilon$ (in red). We can also see 
	the two points of 
	intersection, $p_1$ and $p_2$, of $\hat{D}_i$ with the strict transform of 
	the generic polar curve. In this case $h_i(k_s) = s$ for $s=1,2,3$.}
	\label{fig:hatD_i}
\end{figure}
\begin{block}
In order to construct the invariant A'Campo space, we start by defining a
possibly singular Seifert fibration over $D_i$ for each $i \in \V_\Upsilon$.
Let
\[
  \Ainv_i
\]
be the quotient space of $\Dro_i$
obtained by collapsing each connected component of
$\Dro_{ik,\theta}$ to a point, for each $\theta \in \R/2\pi\Z$, and
for each edge $ik$ in $\Gamma \setminus \Upsilon$.
Note that each invariant $i \in \V$ (recall \Cref{def:upsilon} can have at most 
one
adjacent edge that is not in $\Upsilon$.
As a result, we have a Seifert fibration
\[
  \Pi_i: \Ainv_i \to \widehat D_i
\]
which has at most one singular fiber 
\[
O_i = \Pi_i^{-1}(q_{i,0}).
\]
The function
$\arg(f)$ induces a horizontal fibration
\[
  H_i : \Ainv_i \to \R / 2\pi\Z.
\]
For each $\theta \in \R/2\pi\Z$, denote by
$\Ainv_{i,\theta} = H_i^{-1}(\theta)$ the
fiber.
Then, the restriction of $\Pi_i$
\[
  \Pi_{i,\theta}: \Ainv_{i,\theta} \to \widehat D_i
\]
is a branched covering of degree $m_i$ with branching locus $O_{i,\theta} = O_i 
\cap \Ainv_{i,\theta}$. 
Denote by
\[
  G_{i,\theta}: \Ainv_{i,\theta} \to \Ainv_{i,\theta}
\]
the generator of the Galois group of this cover obtained by following
Seifert fibers in the direction where $H_i$ increases.
Then, $\Ainv_i$ is the mapping torus (using an interval of length $2\pi$)
of $G_{i,\theta}$ for any $\theta$, 
and $H_i$ is the usual projection to $\R/2\pi\Z = [0,2\pi]/0\sim 2\pi$.

For any neighbor $k$ of $i$ in $\Upsilon$ we define $\partial_k \Ainv_i$ as the 
connected component of $\partial \Ainv_i$ corresponding to the edge $ik$. 
Similarly $\partial_k \Ainv_{i,\theta}$ is the part of $\partial 
\Ainv_{i,\theta}$ formed by the collection of $\gcd(m_i,m_k)$ circles that lie 
in 
$\partial_k \Ainv_i$.
\end{block}

\begin{block}\label{blc:coords_uro}
 Let $j\to i$ be an edge in 
$\Upsilon$.
 We use the coordinates $(u_j, v_{ji})$ on $U_{ji} = U_j \cap \tU_i$ where 
 $v_{ji} = v_j - h_j(i) = 
 s_{ji} e^{i\beta{ji}}$ and  $u_j = r_j e^{i\alpha_j}$. We use coordinates 
 $(r_j,\alpha_j, s_{ji}, \beta_{ji})$ for 
\[
\Uro_{ji} = \sigma^{-1}(U_{ji}).
\]  
In these coordinates, the boundary component $\partial_i \Ainv_j$ 
(corresponding to the edge $ij$) has 
coordinates $(\alpha_j,\beta_{ji})$. We also use coordinates $(\tu_i, \tv_{i}$ 
on $U_{ji} = U_j \cap \tU_i$, in this case the induced coordinates on 
$\Uro_{ji}$ are $(\tr_i,\talpha_i, \ts_i, \tbeta_i)$.

\end{block}
Now, define the space
\[
  \Ainv_{ji} = [0,1] \times \R/2\pi\Z \times \R/2\pi\Z.
\]
With coordinates $t, \alpha, \beta$. The space $\Ainv_{ji}$ is the 
interpolating piece of the A'Campo space between 
the pieces $\Ainv_j $ and $\Ainv_i$. We also have a horizontal fibration on 
this 
piece
\[
\begin{split}
	H_{ji} : \Ainv_{ji} &\to \R / 2 \pi \Z \\
	(t, \alpha, \beta) & \mapsto \pi M_{ji} + m_j \alpha + m_i \beta
\end{split}
\]
Recall that by \cref{eq:Mik}, $M_{ji} \in \Z$. Observe that 
\[
H_{ji}^{-1} (\theta) = \set{(t, \alpha, \beta) \in \Ainv_{ji}}{ m_j \alpha + 
m_i \beta = \theta 
- \pi M_{ji}}
\]
is a collection of $m_{ij}=\gcd(m_i,m_j)$ cylinders. We set $\Ainv_{ji,\theta} 
= H_{ji}^{-1} (\theta)$.
\begin{figure}[ht]
\begin{center}
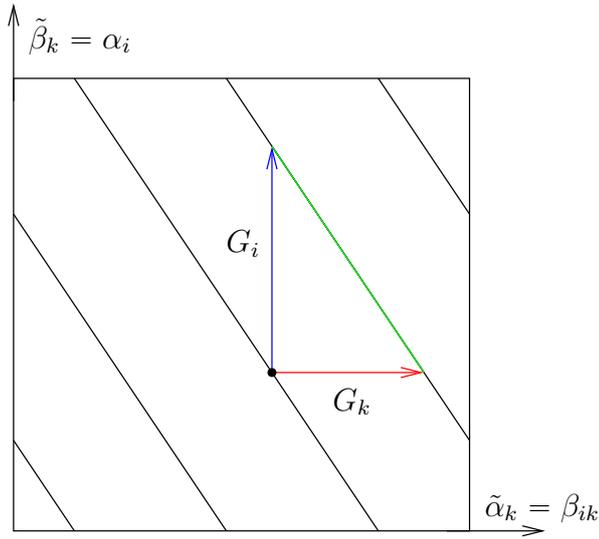
\caption{Suppose that $i\to k$ is an edge, and that
$m_i = 3$ and $m_k = 2$. We see the projection of $\Ainv_{ik}$
to $(\R/2\pi\Z)^2$,
with coordinates
$(\tilde\alpha_k, \tilde\beta_k) = (\beta_{ik}, \alpha_i)$.
The black subspace is the intersection with a Milnor fiber, given as a level
set of $m_i \tilde \alpha_i + m_k \tilde\beta_k$.
The geometric monodromy $G_k$ maps
a point with coordinates $(\tilde \alpha_k, \tilde\beta_k)$ to a point
with coordinates $(\tilde \alpha_k + 1/3, \tilde\beta_k)$, as indicated by
the red arrow. Simiarly, $G_i$ adds $1/2$ to the $\alpha_i$
coordinate. If we fix some point $x \in (\R/2\pi\Z)^2$, we have a segment
in $\{x\}\times [0,1] \subset \Ainv_{ik}$. Projecting the image of this
segment by $G_{ik} \to (\R/2\pi\Z)^2$, we get the green segment in the
picture, interpolating between $G_i$ and $G_k$.}
\label{fig:int_torus}
\end{center}
\end{figure}
We define also the map 
\[
\begin{split}
G_{ji} : \Ainv_{ji} & \to \Ainv_{ji} \\
(t,\alpha,\beta) & \mapsto \left(t, \alpha + 2\pi \frac{1-t}{m_j}, \beta + 2 
\pi 
\frac{t}{m_i} \right).
\end{split}
\]
It follows from construction that $ G_{ji} |_{\Ainv_{ji,\theta} }$ has 
$\Ainv_{ji,\theta} $ as its image and so it defines a map
\[
G_{ji, \theta} = G_{ji} |_{\Ainv_{ji,\theta} }: \Ainv_{ji,\theta}  \to 
\Ainv_{ji,\theta} . 
\]
We define the gluing maps
\begin{equation}\label{eq:glue_ji}
\begin{split}
	\partial_i \Ainv_j  & \to \{0\} \times \left( \R / 2 \pi 
	\Z \right)^2  \subset \Ainv_{ji}\\
	(\alpha_j,\beta_{ji}) & \mapsto (0, \alpha_j, \beta_{ji})
\end{split}
\end{equation}
and
\begin{equation}\label{eq:glue_ij}
\begin{split}
	\partial_j \Ainv_i  & \to \{1\} \times \left( \R / 2 \pi 
	\Z \right)^2 \subset \Ainv_{ji} \\
	(\talpha_i,\tbeta_{i}) & \mapsto (0, \tbeta_i, \talpha_{i})
\end{split}
\end{equation}
where $\partial_i \Ainv_j$ is the connected component of $\partial \Ainv_j$ 
that corresponds to the edge that connects $i$ with $j$ in $\Gamma$.
\begin{definition}\label{def:Ainv}
	We define the {\em invariant A'Campo space} as the quotient
	\[
	\Ainv = \left( \coprod_{i \in \V} \Ainv_i \amalg \coprod_{ji \in e(\Gamma)} 
	\Ainv_{ji}\right) / \sim
	\]
	where $\sim$ is the equivalence relation established by 
	\cref{eq:glue_ji,eq:glue_ij}. It follows from construction that the maps 
	$H_i$ and $H_{ji}$ glue together to form a map
	\[
	H: \Ainv \to \R / 2 \pi \Z.
	\]
	We denote $\Ainv_\theta = H^{-1}(\theta)$.
	Similarly, the maps $G_{i}$ and $G_{ji}$ glue together to a map
	\[
	G: \Ainv \to \Ainv. 
	\]
	This map leaves $\Ainv_\theta$ invariant for each $\theta \in \R / 2 \pi 
	\Z$, inducing a monodromy map
	\[
	G_\theta : \Ainv_\theta \to \Ainv_\theta.
	\]
\end{definition}

 \begin{figure}
	\centering
	\resizebox{0.8\textwidth}{!}{
\begingroup%
  \makeatletter%
  \providecommand\color[2][]{%
    \errmessage{(Inkscape) Color is used for the text in Inkscape, but the package 'color.sty' is not loaded}%
    \renewcommand\color[2][]{}%
  }%
  \providecommand\transparent[1]{%
    \errmessage{(Inkscape) Transparency is used (non-zero) for the text in Inkscape, but the package 'transparent.sty' is not loaded}%
    \renewcommand\transparent[1]{}%
  }%
  \providecommand\rotatebox[2]{#2}%
  \newcommand*\fsize{\dimexpr\f@size pt\relax}%
  \newcommand*\lineheight[1]{\fontsize{\fsize}{#1\fsize}\selectfont}%
  \ifx\svgwidth\undefined%
    \setlength{\unitlength}{427.86276509bp}%
    \ifx\svgscale\undefined%
      \relax%
    \else%
      \setlength{\unitlength}{\unitlength * \real{\svgscale}}%
    \fi%
  \else%
    \setlength{\unitlength}{\svgwidth}%
  \fi%
  \global\let\svgwidth\undefined%
  \global\let\svgscale\undefined%
  \makeatother%
  \begin{picture}(1,0.36695923)%
    \lineheight{1}%
    \setlength\tabcolsep{0pt}%
    \put(0,0){\includegraphics[width=\unitlength,page=1]{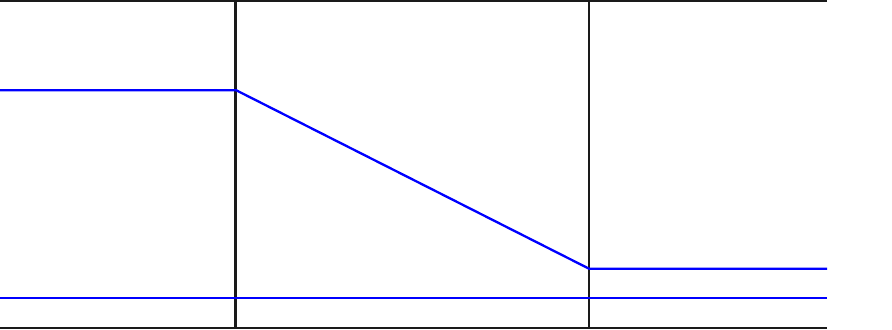}}%
    \put(0.39309182,0.3188148){\color[rgb]{0,0,0}\makebox(0,0)[lt]{\lineheight{1.25}\smash{\begin{tabular}[t]{l}$\Ainv_{ij,\theta}$\end{tabular}}}}%
    \put(0.93562954,0.03029298){\color[rgb]{0,0,0}\makebox(0,0)[lt]{\lineheight{1.25}\smash{\begin{tabular}[t]{l}$S_\theta$\end{tabular}}}}%
    \put(0.9332124,0.06324419){\color[rgb]{0,0,0}\makebox(0,0)[lt]{\lineheight{1.25}\smash{\begin{tabular}[t]{l}$G_\theta(S_\theta)$\end{tabular}}}}%
  \end{picture}%
\endgroup%
}
	\caption{We see a prong $S_\theta$ traversing an interpolating piece 
	$\Ainv_{ij,\theta}$ and its image $G_\theta(S_\theta)$ by the monodromy.}
	\label{fig:acampo_monodromy}
\end{figure}
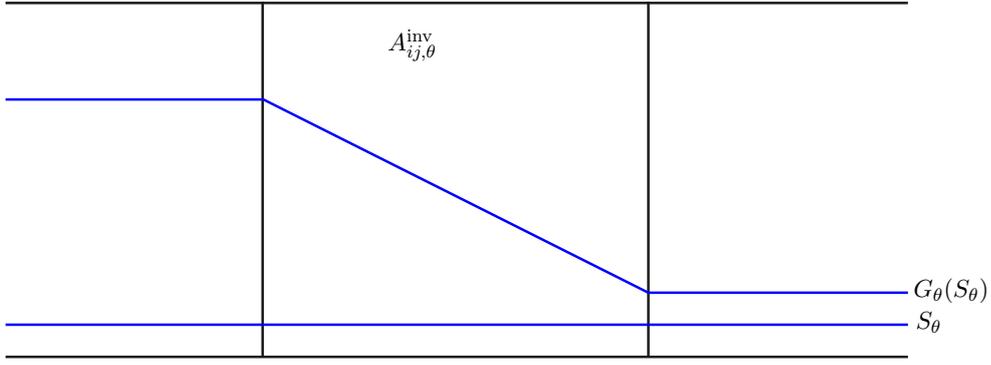

\subsection{Description of the flow} 
\label{ss:flow_description}

Having constructed the invariant A'Campo space $\Ainv$ as a
fibration $H: \Ainv \to \R/2\pi\Z$, we now introduce a
dynamical system on this space. We will define a 
vector 
field
$\xiinv$ that is \emph{horizontal} with respect to this
fibration (i.e., tangent to the fibers $\Ainv_\theta$). The
trajectories of this flow will encode the topology of the
Milnor fiber.  This vector field coincides with the one defined in 
\cite{PSspine} except over the exceptional divisor $D_0$ corresponding to the 
first 
blow up.

The flow is first defined on the base spaces $\hat{D}_i$ of the
Seifert fibrations using a meromorphic function $f_i$, and
then lifted horizontally to the pieces $\Ainv_i$. Finally,
we will show how these local flows glue together to form a
global, well-defined vector field $\xiinv$ on the entire space $\Ainv$.

Let $i \in \V$. Then $\pi^*f$ vanishes with order $m_i$ along $D_i$, and
$\pi^*x$ vanishes with order $c_{0,i}$ along $D_i$. As a result, we have
a meromorphic function on $D_i$, defined by
\begin{equation} \label{eq:f_i_def}
  f_i(v_i) = \frac{\pi^*f|_{U_i}^{c_{0,i}}}{\pi^*x|_{U_i}^{m_i}} (0,v_i).
\end{equation}
This function takes finite nonzero values on $D_i^\circ$, and has
valuation $m_k c_{0,i} - m_i c_{0,k}$ at the intersection point
$D_i \cap D_k$, for any $k \in \W$.
In fact, if $k \in \W^+_i$, then $m_k c_{0,i} - m_i c_{0,k} = n_{ik} \geq 0$
by \cite[Lemma 3.4.7]{PSspine}, and if we have a vertex $j\in \V$ such that
$j \to i$, then $f_i$ has a pole of order $n_{ij}$ at $D_i \cap D_j$.
Denote by $d_i$ the leading term of this polynomial in the variable
$v_i$. Since $\pi^*f$ is real by \cref{eq:a_i_k} and $\pi^*x$ is also real in 
these coordinates, then $d_i \in \R$. Then
\begin{equation} \label{eq:f_i_a_i}
  f_i(v_i)
  = d_i\prod_{\ell \in \V^+_i}(v_i - h_i(\ell))^{n_{i\ell}}
  = d_i v_i^{n_{ji}} + \lot
\end{equation}
where $\lot$ stands for lower order terms in $v_i$. In the case where $i=0$, 
there is no $j \to 0$ but still we have a similar expression
\begin{equation} \label{eq:f_0_a_0}
	f_0(v_0)
	= d_0\prod_{\ell \in \V^+_0}(v_0 - h_0(\ell))^{n_{i\ell}}
	= d_0 v_0^{m_0} + \lot
\end{equation}

\subsubsection*{A vector field on $\hat{D}_i$}
Let $$\xi_i = - \nabla \log 
|f_i|$$
be the logarithmic gradient vector field of $|f_i|$ defined on $D^\circ_i$. We 
claim that by 
\cref{eq:f_i_a_i} (see also \cref{fig:hatD_i}),
the vector field has exactly one zero $q_{i, a}$ lying on the real segment 
between the 
zeroes $a$ and $a+1$ for $ 1 \leq a \leq \max (h_i -1) = p_i$ (recall 
\Cref{def:label_function}). Indeed, $f_i'$ has $n_{ij}-1$ zeroes (one less 
zero than $f_i$) counted with multiplicity; it has a zero of order 
$n_{i\ell}-1$ at $h_i(\ell)$ for each $a$ and it has at least 
one zero between each $h_i(\ell)$ on the real line because it is a real 
polynomial, so the claim follows. We give the 
same name $q_{i,a}$ to the 
zeros of pullback vector field on $\hat{D}_i 
\setminus \partial \hat{D}_i$. 

If $i$ has a dead branch attached to it, we denote by $q_{i,0}$ the point that 
has coordinate $v_i = 0$ in this chart.

If $i=0$, we denote by $q_{i,-1}$ the point that has coordinate $\tilde{v}_0=0$ 
in $\tilde{U}_0 \cap D_0$.
\begin{notation}
 	We denote by $\Sigma_i \subset D_i$ the union all the $q_{i,a}$ points
 	defined above. Where we recall that $q_{i,0}$ is only there if $i$ has a 
 	dead branched attached to it. 
 	
 	We also  define the set $\hat{\Sigma}_i = \hat{\sigma}_i^{-1}(\Sigma_i)$.
\end{notation}
\begin{remark}\label{rem:well_defined}
	The functions $v_i$ and $f_i$ are well defined on $D_i \cap U_i$. From now 
	on, whenever it is convenient we denote by the same symbols the pullback 
	functions $\hat{\sigma}_i^*v_i$ and $\hat{\sigma}_i^*f_i$ defined on $\hat
	D_i$.  We consider the vector field 
\[
\hat{\sigma}_i^{*} \xi_i = \hat{\sigma}_i^{*} \left(- \nabla \log 
|f_i|\right)
\]
which is well defined on $\hat D_i \setminus \partial \hat D_i $. We 
are going to rescale this vector field now so that it has a non-zero extension 
to the boundary components of $\hat{D}_i$. In order to do so, we consider a 
function
\[
\Psi_i: \hat D_i \to \R_+
\]
which takes only positive values and,

\begin{enumerate}
\item in a collar neighborhood of $\partial_k 
\hat{D}_i$, takes the value $|f_i|^{1/n_{ik}}$ for each neighbor vertex $k$ 
with 
$i \to k$
\item in a collar neighborhood of $\partial_j 
\hat{D}_i$, it takes the value $|f_i|^{-1/n_{ij}}$ near $\partial_j \hat{D}_i$ 
where  
$j$ 
is the only neighbor with $j \to i$, if it exists. 

\item when $i=0$ we add the condition that $\Psi_0 = |\tilde{v}_0|^2$ in a 
neighborhood of $0 \in \tilde{U}_0 \cap D_0$. 

\item \label{it:noninv_point} equals the function $|v_i|^{2\left( 1 - 
\frac{m_{ik}}{m_i} \right)}$ in 
a small neighborhood of 
$q_{i,0} 
\in \hat{D}_i$ where $k$ is the only adjacent neighbor with $h_i(k)=0$ in case 
it exists.
\end{enumerate} 
\end{remark}
\begin{prop} \label{prop:hat_xi}
	The vector field
	\[
	 \Psi_i \hat{\sigma}_i^{*} \xi_i
	\]
	extends to all of $ \hat{D}_i$ as a vector field $\hat \xi_i$, and it 
	points inwards along 
	$\partial_j \hat{D}_i$ when $j \to i$ and outwards along $\partial_k 
	\hat{D}_i$ for the neighbors $k$ with $i \to k$. Furthermore,
	\begin{enumerate}
	\item \label{it:zeroset_sigma_i}the zero set of 
	$\hat{\xi}_i$ is $\hat{\Sigma}_i$,
	\item \label{it:saddle} for $1 \leq a \leq p_i$, the singularity $q_{i,a}$ 
	is a non 
	degenerate saddle point of the vector field, 
	\item \label{it:repller} at $\tilde{v}_0 = 0$ on $\tilde{U}_0 \cap D_0$, 
	the vector field 
	$\hat{\xi}_0$ has a repeller, 
	and
	\item \label{it:traj_tangent} the trajectories of 
	$\hat{\xi}_i$ are tangent to the level sets of the function 
	$\arg(\hat\sigma_i^*f_i)$ or, 
	equivalently, the level sets of $\arg(\hat\sigma_i^*f_i)$ are disjoint 
	union of 
	trajectories of $\hat{\xi}_i$.
\end{enumerate}
\end{prop}

\begin{proof}
 Let $k$ be a neighbor with $i \to k$. Near the intersection point $D_i \cap 
 D_k$ we consider the coordinate $v_{ik} = v_i - h_i(k)$. By \cref{eq:f_i_a_i}, 
 the function $f_i$ is of the form  $v_{ik}^{n_{ik}} g(v_{ik})$ near $D_i \cap 
 D_k$ where $g(v_{ik})$ is a  unit. Hence vector field 
 $ \xi_i$ is of the form
\begin{equation}\label{eq:xi_i}
\xi_i
=
-\nabla \log |f_i|
=
-\frac{n_{ik}}{\widebar{v_{ik}}} - \frac{\bar{g}'}{g} 
\end{equation}
So, in polar coordinates $(s_{ik}, 
\beta_{ik})$ with $v_{ik} = s_{ik} e^{i  \beta_{ik}}$, the vector field 
$\hat{\sigma}_i^{*} \xi_i$  looks like
\[
-
\left(
\begin{matrix}
	n_{ik} s_{ik}^{-1} \\
	0
\end{matrix}
\right)
- \sigma^* \left(\frac{\bar{g}'}{g}\right).
\]
That is, its only non-zero coordinate has a pole of order $1$ at $\partial_k 
\hat{D}_i$. Since the function $|f_i|^{1/n_{ik}}$ extends to and vanishes with 
order $1$ at
$\partial_k \hat{D}_i$, we find that  $\xi_i = \Psi_i \hat{\sigma}_i^{*} \xi_i$ 
extends as a non-zero vector field to $\partial_k \hat{D}_i$. Moreover, since 
$n_{ik} >0$, this vector field points outwards.

If $j \to i$, the argument applies verbatim using the coordinate $\tv_{ij}$ 
instead of $v_{ik}$. The only difference is that $f_i$ has a pole of order 
$n_{ij}$ instead of a zero and so, in the end, the vector field points inwards. 
If $i=0$, we recall the expression \cref{eq:f_0_a_0} and observe that $f_0$ has 
a pole of order $m_0$ at infinity and so $\xi_0$ is transverse and points 
inwards to every sufficiently large circle centered at $0$ in the chart $U_0$.

By construction, the vector field restricted to $\hat{D}_i \setminus 
\left(\partial 
\hat{D}_i \cup \{q_{i,0}\} \right)$ vanishes on $\sigma_i^{-1}(\Sigma_i)$. That 
it also vanishes on $q_{i,0}$ follows 
from 
\Cref{rem:well_defined}, \cref{it:noninv_point}. This proves 
\cref{it:zeroset_sigma_i}.

Item \cref{it:saddle},  follows from the fact that 
$f_i'(v_i)$ has a simple zero at $q_{i,a}$ and $f_i(q_{i,a}) \neq 0$, so 
$\overline{\frac{f_i'}{f_i}}$ has a simple saddle point at $q_{i,a}$.

In order to show \cref{it:repller}, we recall \cref{eq:f_0_a_0} and so 
\cref{eq:xi_i} becomes
\[
\xi_0
=
-\nabla \log |f_0|
=  \frac{m_0}{\widebar{\widetilde{v_0}}} - \frac{\bar{g}'}{g}
\]
and so $\hat{\xi}_0$ has the same behavior as the vector field $\tilde{v}_0$ 
near 
$q_{0,-1}$.

The last statement \cref{it:traj_tangent} follows from the expression in 
\cref{eq:xi_i} which 
coincides with $\hat{\xi}_i$ in the interior of $\hat{D}_i$ up to a positive 
scalar.
\end{proof}

\begin{definition}\label{def:hatxi_i}
	We denote by $\hat{\xi}_i$ the extension to all $\hat{D}_i$ of the vector 
	field whose existence is  given by the previous lemma.
\end{definition}

Let $\ell \in \V_i^+$ be an invariant vertex in $\Upsilon$ and let $j$ be the 
only vertex 
with $j \to i$. Then, by construction, the trajectories of the vector field 
$\hat{\xi}_i$  that end at $\partial_\ell \hat{D}_i$, start 
 either at $\partial_j 
\hat{D}_i$ or at a point in $\hat{\Sigma}_i $. Let
\begin{equation}\label{eq:Deltail}
	\hat\Delta_{i\ell}:\partial_\ell\hat{D}_i \dashrightarrow 
	\partial_j\hat{D}_i
\end{equation}
be the rational map that takes a point $\hat{q} \in \partial_\ell\hat{D}_i$ to 
the begining  in  $\partial_j \hat{D}_i$ of the trajectory
of $\hat{\xi}_i$ that ends at $\hat{q}$. Since $\hat{\Sigma}_i$ is finite, this 
is well defined on an open dense set of $\partial_\ell \hat{D}_i$. The 
following lemma 
shows that this map is linear and gives an explicit formula for it.

\begin{lemma}\label{lem:formula_Deltaiell}
In the situation above, the map $\hat\Delta_{i\ell}$ is a rational linear map 
defined on $\partial_\ell \hat{D}_i \setminus \{0,\pi\}$ and, when it is 
defined, it is given by
\[
\hat\Delta_{ik}(\hat{q}) = \beta_{ik}(\hat{q})\frac{n_{ik}}{n_{ij}} + 2 \pi 
\sum_{\ell 
\in 
	S_{ik}} \frac{n_{i\ell}}{2 n_{ij}}
\]
if $0 < \arg(v_{ik}(\hat q)) < \pi$, and
\[
\hat\Delta_{ik}(\hat{q}) = \beta_{ik}(\hat{q})\frac{n_{ik}}{n_{ij}} - 2 \pi 
\sum_{\ell 
\in 
	S_{ik}} \frac{n_{i\ell}}{2 n_{ij}}
\]
if $-\pi < \arg(v_{ik}(\hat q)) < 0$.
\end{lemma}

\begin{proof}
	The fact that it is rational, is explained in the discussion before the 
	lemma. That it is well defined on $\partial_\ell \hat{D}_i \setminus 
	\{0,\pi\}$ follows from the fact that $\Sigma_i$ is on the real axis $
	\{\Im(v_i) = 0\}$ and this real axis is tangent to the flow of $\xi_i$
	(see \cref{fig:stable_D_i}).
	Because of \Cref{prop:hat_xi} \cref{it:traj_tangent}, 
	\[
	\arg(\hat{\sigma}_i^*f_i)(\hat{\Delta}_{i k}(\hat{q})) =	
	\arg(\hat{\sigma}_i^*f_i)(\hat{q}).
	\]
	 Then, the above formula together with \cref{eq:f_i_a_i}, gives an 
	 expression of the form
	\[
	n_{ij} \beta_i(\hat{\Delta}_{i k}(\hat{q})) + a \pi = n_{ik} 
	\beta_{i k}(\hat{q}) +b \pi.
	\]
	where, $a,b \in \{0,1\}$ depending on the signs of the leading coefficients 
	of $f_i$. In particular, it gives that the rational function is of the form
	
	\[
	\hat\Delta_{i\ell}(\hat{q}) = \beta_{ik}(\hat{q})\frac{n_{ik}}{n_{ij}} + 
	\text{independent term}
	\]
	This proves that the rational function is actually linear. 
	Let's find the independent term. Assume, for the moment, that $\ell > k$ 
	for all $k \in \V_i^+$, $k \neq 
	\ell$ and let $\hat{q}$ be such that $\beta_{i\ell}(\hat{q})=0$. Since 
	$\hat{\xi}_i$ is tangent to the real axis, we have 
	$\beta_{i}(\hat\Delta_{i\ell}(\hat{q}))=0$ as well. Substituting in the 
	above formula, gives that the independent term, in this case is $0$. 
	
	Denote by $I^+_a$ the points in the circle $\hat{\sigma_i}^{-1}(a)$ where 
	$0 < \beta_{ih_i^{-1}(a)}<\pi$. Equivalently, define $I^-_a$ corresponds 
	with the arc where	$\pi < \beta_{ih_i^{-1}(a)}<2\pi$. By 
	construction, the intervals 
	$\hat{\Delta}_{i\ell}(I^+_{h_i(k)})$ 
	appear in the order
	\[
	\hat{\Delta}_{i\ell_{p_i+1}}\left(I^+_{p_i+1}\right),
	\hat{\Delta}_{i\ell_{p_i}}\left(I^+_{p_i}\right),
	\ldots, 
	\hat{\Delta}_{i\ell_1}\left(I^+_{1}\right),\hat{\Delta}_{i\ell_1}\left(I^-_{1}
	 \right), 
	\ldots, \hat{\Delta}_{i\ell_{p_i+1}}\left(I^-_{p_i+1}\right)
	\]
   The formula follows, since this gives
   \[
     \lim_{\arg(v_{ik}(q)) \to 0^\pm}
     =
     \pm\pi\sum_{\ell \in S_{ik}} \frac{n_{i\ell}}{n_ij}. \qedhere
   \]
\end{proof}
\begin{lemma}\label{lem:xiinv}
There exists a smooth  vector field $\xiinv_i$ on $\Ainv_i$ which satisfies
\begin{enumerate}
\item \label{it:lift} it is a lift of the vector field $\hat{\xi}_i$, that is, 
\[
D \Pi_i (\xiinv_i (p)) = \hat{\xi}_i(\Pi_i(p))
\]
for all $p \in \Ainv_i$

\item \label{it:tangent} It is tangent to the manifolds $\Ainv_{i,\theta}$, 
that is, it is tangent 
to the horizontal fibration fibration defined by $H_i$, equivalently 
\[
D H_i(\xiinv_i (p)) = 0
\]
for all $p \in \Ainv_i$.
\end{enumerate}
\end{lemma}

\begin{proof}
The horizontal fibration defines a connection on the locally trivial fibration
\[
\Pi_i|_{\Ainv_i \setminus O_i} : \Ainv_i \setminus O_i 
\to
\hat{D}_i \setminus \{q_{i,0}\}
\]
so $\xiinv_i$ is determined by  \cref{it:lift,it:tangent} on $\Ainv_i \setminus 
O_i$. We define $\xiinv_i$ to be $0$ on $O_i$. In order to finish the proof we 
need to verify that $\xiinv_i$ is smooth at $O_i$.

Let $k$ be such that $i \to k$ and $k$ lies on the only dead branch of $i$. 
Near a point in $O_{i,\theta}$ the map $\Pi_{i,\theta}$ is of the form $w 
\mapsto w^{\frac{m_i}{m_{ik}}}$. We identify tangent vectors to 
$\Ainv_{i,\theta}$ near $O_{i}$ with complex numbers via the coordinate $w$. On 
another hand, near $q_{i,0}$ the vector 
field 
$\hat{\xi}_i$ is of the form
\[
|v_i|^{2 \left( 1 - \frac{m_{ik}}{m_i} \right)} 
\]
So, when pulled back by $\Pi_{i,\theta}$, near a point in $O_{i,\theta}$ takes 
the form
\begin{equation} \label{eq:bar_w}
	\left(\frac{m_i}{m_{ik}} w^{\frac{m_i}{m_{ik}} - 1}\right)^{-1} \cdot 
	\left(|v_i|^{2 
	\left( 1 - \frac{m_{ik}}{m_i} \right) } \right)_{v_i 
	=w^{\frac{m_i}{m_{ik}}}} 
=
	\frac{m_{ik}}{m_i}\bar{w}^{ \frac{m_i}{m_{ik}} - 1 }.
\end{equation}
See the left hand side of \cref{fig:multiprong}.
\end{proof}
The previous lemma defines a vector field $\xiinv_i$ on each piece $\Ainv_i$ by 
pulling back the vector field $\hat{\xi}_i$ by the functions $\Ainv_{i,\theta} 
\to 
\hat{D}_i$. We denote by $\xiinv_{i,\theta}$ the restriction 
$\xiinv_i|_{\Ainv_{i,\theta}}$.   We define vector fields 
$\xiinv_{ji}$ on $\Ainv_{ji}$ by
\begin{equation}\label{eq:xi_ji}
\hat{\xi}_{ji} = 
\left(
\begin{matrix}
1   \\
0\\
0
\end{matrix}
\right) = \partial_t
\end{equation}

\begin{definition}\label{def:xiinv}
	Using \cite[Theorem 1.4]{hMilnor} together with \Cref{prop:hat_xi}  and 
	\cref{eq:xi_ji},
	we can endow  the topological manifold $\Ainv$ with a $C^\infty$ structure 
	(see also \cite[Section 12.3]{PSspine}). This $C^\infty$ structure is the 
	unique one such that the vector 
	fields $\xiinv_i$ and $\xiinv_{ji}$ glue to a smooth a vector field defined 
	on all $\Ainv$. We denote this vector field by $$\xiinv.$$
	Moreover, since this vector field is tangent to $\Ainv_{\theta}$, we define
	\[\xiinv_\theta = \xiinv|_{\Ainv_{\theta}}.\]
	We define the sets $\Siinv_i = \Pi^{-1}(\hat{\Sigma}_i)$ and   
	$\Siinv_{i,\theta} 
	= \Pi^{-1}_\theta(\hat{\Sigma}_{i})$.
\end{definition}
The following lemma follows from the construction of the vector field $\xiinv$, 
the definition of the $\Sigma_i$ sets and \Cref{lem:xiinv}
\begin{lemma}\label{lem:critical_points_xiinv}
 The zero set of $\xiinv$ is $\Siinv = \bigcup_{i \in \Upsilon} 
 \Siinv_i$. The zero set of $\xiinv_{\theta}$ is $\Siinv_{\theta} = \bigcup_{i 
 \in \Upsilon} 
 \Siinv_{i, \theta}$. Furthermore
 
 \begin{enumerate}
\item \label{it:lift_saddle}
over each $q_{i,a}$ with $1 \leq a \leq p_i$, the vector field $\xiinv$ 
has $m_i$ singularities which are saddle points.

\item \label{it:lift_petri}
over each $q_{i,0}$ (when it exists) the vector field $\xiinv$ has 
$m_{ik}$ singularities which are $m_i/m_{ik}$- pronged singularities. Where $i 
\to k$ and $k$ is the only child of $i$ which lies on a dead branch.
\end{enumerate} 
\end{lemma}

\begin{proof}
Item \cref{it:lift_saddle} follows from the fact that,
over $q_{i,a}$, the vector field $\xiinv$ is a lift via a local diffeomorphism
of $\hat\xi_i$, and \Cref{prop:hat_xi}\cref{it:saddle}.
Item \cref{it:lift_petri} follows from \cref{eq:bar_w}.
\end{proof}

\section{The spine}
\label{s:spine}
We now define the central object for our combinatorial analysis:
the \emph{invariant spine} $\Sinv_\theta$. This spine is very similar to the 
one defined in \cite[Section 12]{PSspine}: it differs from that one on its 
definition on the points that lie over $\Dro_{0,\theta}$. This spine is the
$1$-dimensional CW-complex formed by the union of all stable
manifolds of the singularities (critical points) of the flow
$\xiinv_\theta$.

\begin{notation}
Let $q \in \Ainv_\theta$. We denote by 
\[
\gamma_q: J_q \to \Ainv_\theta
\]
the only trajectory of $\xiinv_{\theta}$ that passes through $q$ with 
$\gamma_q(0) = q$ and $J_q = (J_q^-, J_q^+) \subset \R$ the maximal interval of 
definition. We define the points $\omega^+(q)$ and $\omega^-(q)$.

\[
\omega^\pm(q) = \lim_{t \to I_q^\pm} \gamma_q(t).
\]

\end{notation}

\begin{definition}
Let $q \in \Siinv_\theta$ we denote by $\Sinv_{q, \theta}$ the union of all the 
stable manifolds
of the singularity $q$. Since there are no sinks, we find that the set
\[
\Sinv_\theta = \bigcup_{q  \in \Siinv_\theta} \overline{\Sinv_{q, \theta}}
\]
is a finite $1$-dimensional CW-complex. We call it the {\em relative invariant 
spine}. 
\end{definition}

\section{Dynamics of the vector field}
\label{s:dynamics_vector_field}

Recall that by \Cref{prop:hat_xi} \cref{it:saddle} at the points 
$\{q_{i,a}\}_{1 \leq a \leq p_i}$ the vector field $\hat\xi_i$ has saddle 
points and that $q_{i,a} \in (a, a+1)$. On $\Droc_{i} \subset \Ainv_{i}$ 
in 
the real oriented 
blow up, we have coordinates $(\alpha_i, s_i, \beta_i)$. For each $\theta \in 
\R / 2 \pi \Z$, over 
each 
$q_{i,a}$ there lie $m_i$ points on $\Droc_{i,\theta} \subset 
\Ainv_{i,\theta}$.

 \begin{figure}
	\centering
	\large
	\resizebox{0.6\textwidth}{!}{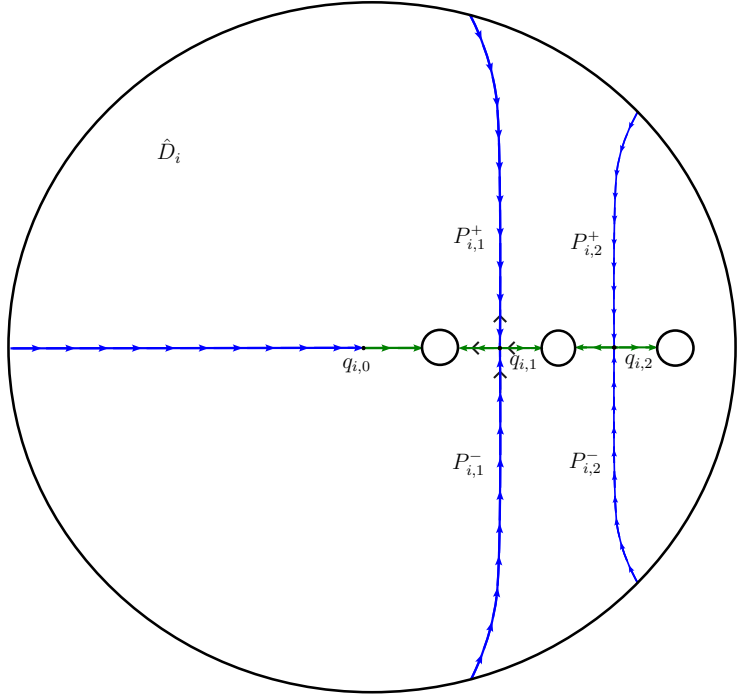}
	\caption{We see $\hat{D}_i$ with the stable manifolds (in blue) associated 
	with the polar points $q_{i,1}$ and $q_{i,2}$ and the point $q_{i,0}$. In 
	green 
	we see the corresponding unstable manifolds. The black arrows indicate the 
	orientations that prongs take to induce the correct orientation on the 
	chains of $C_{1,\theta}$.}
	\label{fig:stable_D_i}
\end{figure}

\subsubsection*{Coordinates of the prongs}

Similarly as we did in \cref{eq:Deltail} we now define a map
between boundary components of exceptional divisors in the real oriented blow 
up.

\begin{definition}\label{def:Delta_function}
Consider an invariant edge $i \to k$, that is with $i,k \in \V_{\Upsilon}$ and 
$k \in \V_i^+$. Let $q \in 
\Dro_{ik}$  with coordinates $(\alpha_i(q), \beta_{ik}(q)) \in \left(\R / 2\pi 
\Z \right)^2$ and with $\beta_{ik} \notin \{0, \pi\}$.  Take the unique 
trajectory 
of the vector field $\xiinv$ that passes through 
$q$. If $q$ lies in $\Ainv_{i,\theta}$, then so does this 
trajectory (because $\xiinv_i$ restricts to a tangent vector field 
$\xiinv_{i,\theta}$ on $\Ainv_{i,\theta}$).
If we follow this trajectory backwards, the next separating torus it passes 
through is precisely $\Dro_{ij}$ with $j \to i$. Let $\Delta(q)$ be this 
intersection point. This defines a rational map
\begin{equation}
	\Delta_{ik}:\Dro_{ik} \dashrightarrow 
	\Dro_{ij}
\end{equation}
that furthermore preserves the angle $\theta$ by construction. This map 
actually lifts the map $\hat{\Delta}_{ik}$ from \cref{eq:Deltail}.
\end{definition}

Using the formula from \Cref{lem:formula_Deltaiell} we get 
the following lemma.

\begin{lemma} \label{lem:formula_alpha_q}
	Let $q \in \Dro_{ik}$ with $i\to k$. Then, the trajectory of $\xiinv $ that 
	passes through $q$, intersects $\Dro_{ij}$ with $j\to i$ in the point 
	$\Delta(q)$ satisfying the following:
	
	If $0<\beta_{ik}<\pi$, then
	\[
	\alpha_i(\Delta(q))= \alpha_i(q) + 
	2\pi
	\left(
	\sum_{\ell \in S_{ik}}
	\left(
	\frac{m_\ell}{2 m_i}
	+ \left( \frac{m _j}{m_i} - b_i \right) \frac{n_{i\ell}}{2 n_{ij}}
	\right)
	+ \frac{\beta_{ik}(q)}{\pi}
	\left(
	\frac{m_k}{2 m_i}
	+
	\left(
	\frac{m_j}{m_i}
	-
	b_i
	\right) 
	\frac{n_{ik}}{2 n_{ij}}
	\right)
	\right) 
	\]
	
	If $\pi<\beta_{ik}<2\pi$, then
	\[
	\alpha_i(\Delta(q))= \alpha_i(q) + 
	2\pi
	\left(
	-\sum_{\ell \in S_{ik}}
	\left(
	\frac{m_\ell}{2 m_i}
	+ \left( \frac{m _j}{m_i} - b_i \right) \frac{n_{i\ell}}{2 n_{ij}}
	\right)
	+ \frac{\beta_{ik}(q)-\pi}{\pi}
	\left(
	\frac{m_k}{2 m_i}
	+
	\left(
	\frac{m_j}{m_i}
	-
	b_i
	\right) 
	\frac{n_{ik}}{2 n_{ij}}
	\right)
	\right) 
	\]
	
	If $i = 0$ we have no vertex $j$ with $j \to i$ but we can 
	define 
	$\alpha_0(\Delta(q))$ by the same formulae, substituting $n_{ij}$ by $m_0$.
\end{lemma}

\begin{proof}
	Let $q \in \Dro_{ik}$ be a point on the boundary component of $\Ainv_i$ 
	corresponding to the intersection $D_i \cap D_k$. Let $\Delta(q) \in 
	\Dro_{ij}$ be the point on the boundary component $D_i \cap D_j$ (where $j 
	\to i$) reached by following the trajectory of $\xiinv$ from $q$.
	
	The trajectory connecting $q$ and $\Delta(a)$ is a lift of a trajectory 
	$\hat{\gamma}$ of $\hat{\xi}_i$ on 
	$\hat{D}_i$, which connects $\hat{q} = \Pi_i(q)$ to
	$\hat{\Delta}(\hat{q}) 
	= \Pi_i(\Delta(q))$. This trajectory $\hat{\gamma}$ must lie on a level set 
	of $\arg(f_i)$, where $f_i$ is the meromorphic function on $D_i$ from 
	\cref{eq:f_i_def}.
	The lifted trajectory $\gamma$ on $\Ainv_i$ lies on a fiber of the 
	horizontal 
	fibration $H_i$, defined by $\theta = \arg(\pi^*f)$.
	
	Let $g_i(v_i) = \pi^*f|_{U_i} / u_i^{m_i}$ be the holomorphic function on 
	$D_i \cap U_i$ defined in \cref{eq:fi_poly}:
	\begin{equation} \label{eq:givi}
	g_i(v_i) = a_{ij} \prod_{\ell \in \V_i^+} (v_i - 
	h_i(\ell))^{m_\ell}.
	\end{equation}
	In the $U_i$ chart we have $u_i = r_i e^{i\alpha_i}$, so the level set is 
	defined by
	\[
	\theta = \arg(\pi^*f) = m_i \alpha_i + \arg(g_i(v_i)).
	\]
	This means that as we flow along the trajectory, the change in the 
	coordinate $\alpha_i$ is determined by the change of $\arg(g_i(v_i))$:
	\[
	m_i \alpha_i(\Delta(q)) - m_i \alpha_i(q)
      = \arg(g_i(v_i(q))) - \arg(g_i(v_i(\Delta(q)))).
	\]
	The change we want to compute is $\alpha_i(\Delta(q)) - 
	\alpha_i(q)$. From the coordinate change of \cref{eq:Y_sim_i}, we have 
	$\tu_i 
	= u_i 
	v_i^{b_i}$, which implies $\talpha_i = \alpha_i + b_i \beta_i$, where 
	$\beta_i = \arg(v_i)$.
	The change in $\alpha_i$ (in the coordinates of the $U_i$ chart) from $q$ 
	to $\Delta(q)$ is
	\begin{equation}\label{eq:alpha_shift}
		\alpha_i(\Delta(q)) - \alpha_i(q) = \frac{1}{m_i} \left( 
		\arg(g_i(\hat{q})) - \arg(g_i(\hat{\Delta}(\hat{q}))) \right).
	\end{equation}
	Let $\beta_{ik}(\hat{q}) = \arg(v_i - h_i(k))$ and 
	$\beta_i(\hat{\Delta}(\hat{q})) = \arg(v_i)$.
	
	We treat the rest of the proof by cases:
	\paragraph{Case 1: $q$ is in the ``upper half'' ($0 < \beta_{ik}(\hat{q}) < 
		\pi$).}
	\begin{itemize}
		\item At $\hat{q}$: as
        $v_i \to h_i(k)$ along the trajectory converging to $\hat q$,
        \cref{eq:givi} gives
		$$\arg(g_i(v_i)) \to \arg(a_{ij}) + m_k \beta_{ik}(\hat{q}) + \pi 
		\sum_{\ell \in S_{ik}} m_\ell \pmod{2\pi}.$$
		
		\item At $\hat{\Delta}(\hat{q})$, that is, $v \to \infty$ along
        the same trajectory, the 
		constants $h_i(\ell)$ are negligible and we find
        \[
        g_i(v_i) 
		\approx
        a_{ij} v_i^{\sum m_\ell} = a_{ij} v_i^{m_i b_i - m_j},
        \]
        where the sum in the exponent runs through $\ell \in \V_i^+$,
        and we use the relation
        \[
          -b_i m_i + \sum_{\ell \in \V_i} m_i = 0,\qquad
          \V_i = \V_i^+ \amalg \{j\}.
        \]
        As a result,
		\[
          \arg(g_i(v_i))
          \to
          \arg(a_{ij}) + (m_i b_i - m_j) \beta_i(\hat{\Delta}(\hat{q}))
          \pmod{2\pi}.
        \]
	\end{itemize}
	Substituting these into the formula for $\alpha_i(\Delta(q)) 
	- \alpha_i(q)$, we find:
	\begin{align*}
		\alpha_i(\Delta(q)) - \alpha_i(q) &= \frac{1}{m_i} 
		\left(
        m_k \beta_{ik}(q)
        + \pi \sum_{\ell \in S_{ik}} m_\ell
        - (m_i b_i - m_j) \beta_i(\Delta(q))
        \right) \\
		&=
        \frac{m_k \beta_{ik}(q)}{m_i}
        + \frac{\pi \sum_{\ell \in S_{ik}} m_\ell}{m_i}
        + \left(\frac{m_j}{m_i} - b_i \right) \beta_i(\Delta(q))
	\end{align*}
	From \Cref{lem:formula_Deltaiell}, we have
	\[
	\beta_i(\Delta(q)) = \frac{n_{ik}}{n_{ij}} 
	\beta_{ik}(q) + \frac{\pi}{n_{ij}} \sum_{\ell \in S_{ik}} n_{i\ell}.
	\]
	Substituting this expression for $\beta_i$ in the expression obtained 
	before (with $\beta_{ik} = \beta_{ik}(\Delta(q))$:
	\begin{align*}
		\alpha_i(\Delta(q)) - \alpha_i(q) &= \frac{m_k 
			\beta_{ik}}{m_i} + \frac{\pi \sum_{S_{ik}} m_\ell}{m_i} + 
		\left(\frac{m_j}{m_i} - b_i \right) \left( \frac{n_{ik}}{n_{ij}} 
		\beta_{ik} + \frac{\pi}{n_{ij}} \sum_{S_{ik}} n_{i\ell} \right) \\
		&= \beta_{ik} \left( \frac{m_k}{m_i} + \left(\frac{m_j}{m_i} - 
		b_i\right) \frac{n_{ik}}{n_{ij}} \right) + \pi \left( 
		\frac{\sum_{S_{ik}} m_\ell}{m_i} + \left(\frac{m_j}{m_i} - b_i\right) 
		\frac{\sum_{S_{ik}} n_{i\ell}}{n_{ij}} \right)
	\end{align*}
	Rewriting this by factoring out $2\pi$:
	\[
	= 2\pi \left[ \sum_{\ell \in S_{ik}} \left( \frac{m_\ell}{2m_i} + 
	\left(\frac{m_j}{m_i} - b_i\right) \frac{n_{i\ell}}{2n_{ij}} \right) + 
	\frac{\beta_{ik}}{\pi} \left( \frac{m_k}{2m_i} + \left(\frac{m_j}{m_i} - 
	b_i\right) \frac{n_{ik}}{2n_{ij}} \right) \right]
	\]
	This matches the first formula in the lemma.
	
	\paragraph{Case 2: $q$ is in the ``lower half'' ($\pi < \beta_{ik}(\hat{q}) 
		< 
		2\pi$).}
	This case follows similarly, using the case $-\pi<\arg(v_{ik})<0$
    in \Cref{lem:formula_Deltaiell}.

	The case $i=0$ follows by setting $j=-1$ (the virtual {\em ancestor vertex} 
	corresponding to $\C^2$), $m_j = m_{-1} = 0$, and noting 
	that the pole order $n_{ij} = n_{0,-1}$ of $f_0$ at $\infty$ is $m_0$.
\end{proof}

\subsection{Generic angles}
\label{ss:generic_angles}

\begin{definition}\label{def:non_generic_angle}
	We say that $\theta \in \R/2\pi\Z$ is a \emph{generic angle} for $f$
	if there is no trajectory in $\Ainv_\theta$ connecting
	two singularities of $\xiinv_\theta$ whose stable set has dimension
	one. Otherwise, we say that $\theta$ is nongeneric.
	
	We denote the set of non-generic angles by $\Theta_{\mathrm{ng}} \subset 
	\R/2\pi\Z$.
	
\end{definition}

\subsection{The set of non-generic angles}
\label{ss:non_generic_angles}

In this section, we prove that the set of angles $\theta$ for which the 
topology of the Milnor fiber varies is finite. The following number is a 
classical invariant and its name is due to 
\cite{LMW_hir}.
\begin{definition}
	\label{def:hironaka_number}
	For any vertex $i$ of the resolution graph $\Gamma$, its associated
	\emph{Hironaka number} is
	\[
	h_i = \frac{m_i}{c_{0,i}}.
	\]
	If $a$ is an arrowhead vertex in $\Gamma$, then we set $h_a = +\infty$, and
	$h_a^{-1} = 0$.
\end{definition}

\begin{thm} \label{thm:finite_nongeneric}
	The set of non-generic angles $\Theta_{\mathrm{ng}} \subset \R / 2\pi \Z$ 
	is finite.
\end{thm}

\begin{proof}
	By \Cref{def:non_generic_angle}, an angle $\theta$ is non-generic if and 
	only if there exists a trajectory $\gamma$ of the vector field 
	$\xiinv_\theta$ inside the fiber $\Ainv_\theta$ connecting two 
	singularities $p, q \in \Siinv_\theta$ such that:
	\begin{enumerate}
		\item $\lim_{t \to +\infty} \gamma(t) = q$, where $q$ is a singularity 
		with $\dim(W^s(q)) = 1$ (a saddle or a multipronged saddle point).
		\item $\lim_{t \to -\infty} \gamma(t) = p$, where $p$ is a singularity 
		with $\dim(W^u(p)) = 1$ (a saddle or a multipronged saddle point).
	\end{enumerate}
	Since $\Siinv_\theta$ is a finite set, it suffices to prove that for any 
	fixed pair of vertices $u, v \in \V$ (not necessarily distinct) and any 
	pair of specific separatrices associated with singularities in $\Ainv_{u}$ 
	and $\Ainv_{v}$, the set of angles $\theta$ allowing a connection is finite.
	
	Recall from \cref{eq:f_i_def} that on each divisor $D_i$, the geometry of 
	the vector field $\hat{\xi}_i$ is determined by the meromorphic function:
	\[
	f_i = \frac{(\pi^*f)^{c_{0,i}}}{(\pi^*x)^{m_i}}.
	\]
	By \Cref{prop:hat_xi}~\cref{it:traj_tangent}, the trajectories of 
	$\hat{\xi}_i$ are contained in the level sets of $\arg(f_i)$. Since 
	$\xiinv_\theta$ is a lift of $\hat{\xi}_i$ (via the local diffeomorphism 
	$\Pi_{i,\theta}$ away from singular fibers), the function $\arg(f_i)$ is 
	constant along any trajectory $\gamma$ lying within the piece 
	$\Ainv_{i,\theta}$.

	The fiber $\Ainv_\theta$ is globally defined by the condition $\arg(\pi^* 
	f) = \theta$. We substitute this into the expression for the argument of 
	$f_i$:
	\[
	\arg(f_i) \equiv c_{0,i} \arg(\pi^* f) - m_i \arg(\pi^* x) \equiv c_{0,i} 
	\theta - m_i \arg(\pi^* x) \pmod{2\pi}.
	\]
	We can express $\arg(\pi^* x)$ in terms 
	of $\theta$ and a local invariants of each vertex. Let $h_i = m_i / 
	c_{0,i}$ be the Hironaka number of the divisor $D_i$ (see 
	\Cref{def:hironaka_number}).
    We have:
	\begin{equation} \label{eq:global_phase_relation}
		\arg(\pi^* x) \equiv \frac{c_{0,i}}{m_i} \theta - \frac{1}{m_i} 
		\arg(f_i) \equiv \frac{1}{h_i} \theta - \frac{1}{m_i} \arg(f_i) 
		\pmod{2\pi}.
	\end{equation}

	Let $\gamma$ be a trajectory connecting $p \in \Ainv_{u,\theta}$ to $q \in 
	\Ainv_{v,\theta}$.
	\begin{itemize}
		\item Near $q$, the trajectory coincides with a specific stable 
		separatrix of $q$. Let $K_q \in [0, 2\pi)$ be the constant value of 
		$\arg(f_v)$ along this separatrix. This constant $K_q$ is is 
		the argument of $f_v$ evaluated at the saddle point or multipronged 
		point. Thus, along the stable part of $\gamma$:
		\[
		\arg(\pi^* x) \equiv \frac{1}{h_v} \theta - \frac{K_q}{m_v} \pmod{2\pi}.
		\]
		\item Near $p$, the trajectory coincides with an unstable 
		separatrix of $p$. Let $K_p \in [0, 2\pi)$ be the constant value of 
		$\arg(f_u)$ along this separatrix. Similarly, along the unstable part 
		of $\gamma$:
		\[
		\arg(\pi^* x) \equiv \frac{1}{h_u} \theta - \frac{K_p}{m_u} \pmod{2\pi}.
		\]
	\end{itemize}

	Since $\arg(\pi^* x)$ is a global function on the ambient space (and the 
	chart gluing maps respect this coordinate), for the unstable manifold of 
	$p$ to connect to the stable manifold of $q$, their values of $\arg(\pi^* 
	x)$ must coincide. This yields the linear equation:
	\[
	\frac{1}{h_v} \theta - \frac{K_q}{m_v} \equiv \frac{1}{h_u} \theta - 
	\frac{K_p}{m_u} \pmod{2\pi}.
	\]
	Rearranging terms:
	\begin{equation} \label{eq:theta_collision}
		\theta \left( \frac{1}{h_v} - \frac{1}{h_u} \right) \equiv 
		\frac{K_q}{m_v} - \frac{K_p}{m_u} \pmod{2\pi}.
	\end{equation}
	
	We claim that if $u \neq v$ are distinct vertices involved in a saddle 
	connection, the coefficient $\Lambda_{uv} = \frac{1}{h_v} - \frac{1}{h_u}$ 
	is generically non-zero.
	Using the definition of Hironaka numbers:
	\[
	\Lambda_{uv} = \frac{c_{0,v}}{m_v} - \frac{c_{0,u}}{m_u} = \frac{m_u 
	c_{0,v} - m_v c_{0,u}}{m_u m_v}.
	\]
	In the resolution graph of a plane curve singularity, the quantity $n_{uv} 
	= m_u c_{0,v} - m_v c_{0,u}$ is non-zero for any pair of vertices connected 
	by a path in the tree (this relates to the strict monotonicity of the 
	function $m/c_0$ 
	along geodesics from the root).	If $u = v$, then the equation becomes $0 
	\equiv K_q/m_v - K_p/m_v 
	\pmod{2\pi}$. In this case we just verify that our model does not yield any 
	saddle connections.
	
	Since $\Lambda_{uv} \neq 0$, \cref{eq:theta_collision} has a finite number 
	of solutions for $\theta$ in $[0, 2\pi)$. As the number of pairs of 
	singular points and separatrices is finite, the total set 
	$\Theta_{\mathrm{ng}}$ is finite.
\end{proof}

\section{Algebraic monodromy and variation} \label{s:mon_var}

In this section, we construct a finite-dimensional chain complex
$(C_\theta, d_\theta)$ that computes the homology of the Milnor fiber. The 
relative invariant spine has a natural structure of CW-complex and the chain 
complex
that we define in this section is a sub-complex of that one which carries the 
same homological information but yields computational advantages.
The generators of the first level of this complex $C_{1,\theta}$ will be unions 
of pairs of trajectories converging to the
singularities of $\xiinv_\theta$. The generators of $C_{0,\theta}$ are the 
repellers 
of 
the vector field in $\Ainv_{0,\theta}$.  We will also define a dual complex 
$C^\vee_\theta$
generated by the unstable manifolds.

We will show that these complexes are quasi-isomorphic to the
singular chain complexes of $\Ainv_\theta$ (absolute and relative,
respectively). This construction provides a concrete algebraic
framework on which the monodromy and variation operators can be
represented as matrices.

\subsection{A chain complex}\label{ss:chain_complex} 

Fix a vertex $i \in \V_{\Upsilon}$ and a non-generic angle $\theta$.

\subsubsection*{Chains for saddle points}

Choose a saddle point $q_{i,a} \in \hat{D}_i$. Let $q_{i,a,\theta}[1]$ be the 
point in $\Pi_{i,\theta}^{-1}(q_{i,a})$ whose $\alpha$ 
coordinate in $[0, 2 \pi)$ is closest to $0$. Let 
\[
q_{i,a,\theta}[2], \ldots, q_{i,a,\theta}[m_i]
\]
be the other $m_i -1$ points labeled according to the orientation of 
$\Pi_i^{-1}(q_{i,a})$. Equivalently, the points are labaled so that
\[
\alpha_i\left(q_{i,a,\theta}[b] \right) \in \left[ \frac{2 \pi (b-1)}{m_i}, 
\frac{2 \pi 
b}{m_i}\right) + 2 \pi \Z \quad  \text{ for all } a \in \{1, \ldots, p_i\} 
\text{ and all } b \in \{1,\ldots, m_i\}.
\]
By \Cref{lem:critical_points_xiinv}, the vector field $\xiinv$ has saddle 
points at 
$q_{i,a,\theta}[b]$. We denote by 
\[
  S_{i,a,\theta}[b]
  =
  \overline{ (\omega^+)^{-1}(q_{i,a,\theta}[b])}
\]
the closure of the stable manifold of $\xiinv$ at $q_{i,a,\theta}[b]$. This 
stable manifold is naturally partioned intro three sets
 \[
  S_{i,a,\theta}[b] = P^+_{i,a,\theta}[b] \cup \{q_{i,a,\theta}[b]\} \cup 
  P^-_{i,a,\theta}[b]
 \] 
where $P^+_{i,a,\theta}[b]$  is the preimage by 
$\Pi_{i,\theta}$ of the stable manifold  $P^+_{i,a}$ of $\hat{\xi}_i$ at 
$q_{i,a}$ that lies 
on the upper half plane; and  $P^-_{i,a,\theta}[b])$ the preimage of the stable 
manifold $P^-_{i,a}$ that lies on the lower half plane (See 
\cref{fig:stable_D_i}).  We orient $S_{i,a,\theta}[b]$ 
so that $P^-_{i,a,\theta}[b]$ keeps its orientation induced by the flow of the 
vector field and $P^+_{i,a,\theta}[b]$ inverts it. See \cref{fig:stable_D_i} 
for a drawing representing these orientations.
Similarly, we denote by 
\[
  U_{i,a,\theta}[b]
  =
  \overline{ (\omega^-)^{-1}(q_{i,a,\theta}[b])}
\]
the unstable manifold of $\xiinv$ at that same point.
We orient $U_{i,a,\theta}[b]$ in such a way that near the point
$q_{i,a,\theta}[b]$, the restriction of $\Pi_{i,\theta}$ to the
preimage of the real axis in $\hat D_i$ reverses orientation. Similary as we 
have done for the stable manifolds, we observe that the unstable manifold is 
partioned as

\[
U_{i,a,\theta}[b]
=
R^+_{i,a,\theta}[b]\cup \left\{q_{i,a,\theta}[b]\right\} \cup 
R^-_{i,a,\theta}[b]
\]
where $R^+_{i,a,\theta}[b]$ is the preimage by 
$\Pi_{i,\theta}$ of the stable manifold  $R^+_{i,a}$ of $\hat{\xi}_i$ at 
$q_{i,a}$ that lies to the right of $q_{i,a}$ on the real axis and 
$R^-_{i,a,\theta}[b]$ is the other one.

If each prong (stable or unstable) is oriented by the flow of $\hat{\xi}_i$, we 
orient the stable and unstable manifolds like
\[
\begin{split}
	S_{i,a,\theta}[b] & 
	= P^-_{i,a,\theta}[b] - P^+_{i,a,\theta}[b] \\
	U_{i,a,\theta}[b]
	&=
	R^-_{i,a,\theta}[b] - R^+_{i,a,\theta}[b]
\end{split}
\]
These orientations are chosen so that the signed intersection 
$\langle S_{i,a,\theta}[b], U_{i,a,\theta}[b]\rangle=+1$ is positive.

\subsubsection*{Chains for multipronged singularities}
Next, assume that $i$ has a neighbor
$k \in \V \setminus \V_\Upsilon$.
Note that there can be at most one such $k$.
We label the $m_{ik}$ preimages of $q_{i,0}$ on $\Ainv_{\theta}$, by
\[
  q_{i,0,\theta}[c], \quad
  \text{ with } c \in \Z / m_{ik}\Z.
\] 
At these points, the vector field $\xiinv$ has 
$m_i/m_{ik}$-pronged singularities where each of the prongs lies in a 
trajectory of $\xiinv$ converging to some $q_{i,0,\theta}[c]$. We label these 
prongs in a similar way as we labeled the points $q_{i,a,\theta}[b]$. More 
concretely, let $q^-_{i,0} \in \R_{-}$ be a point in $\hat{D}_i$ very close to 
$q_{i,0}$. Using the same criterion as above, we obtain points
\[
q^-_{i,0,\theta}[1], \ldots, q^-_{i,0,\theta}[m_i]
\]
each lying in a different prong. Let 
\[
P_{i,0,\theta}[b]
\]
be the prong where the point $q^-_{i,0,\theta}[b]$ lies. Observe that 
$P_{i,0,\theta}[b]$ is a prong of the singularity $q_{i,0,\theta}[b \mod 
m_{ik}]$.
Finally we define the chains
\begin{equation} \label{eq:Si0}
  S_{i,0,\theta}[b]
  = P_{i,0,\theta}[b]
    \cup P_{i,0,\theta}[b']
    \cup 
    \{q_{i,0,\theta}[b \mod m_{ik}]\},
  \qquad
  1 \leq b \leq m_i - m_{ik}
\end{equation}
where $b'$ is the largest element of $\{ 1, 2, \ldots, m_i \}$
which is congruent to $b$ modulo $m_{ik}$.
The prongs $P_{i,0,\theta}[b]$ are oriented, since they are trajectories.
We orient the cycle \cref{eq:Si0} in such a way that
$P_{i,0,\theta}[b]$ has this orientation, and
$P_{i,0,\theta}[b']$ has the opposite orientation.

 \begin{figure}
	\centering
	\resizebox{0.6\textwidth}{!}{
\begingroup%
  \makeatletter%
  \providecommand\color[2][]{%
    \errmessage{(Inkscape) Color is used for the text in Inkscape, but the package 'color.sty' is not loaded}%
    \renewcommand\color[2][]{}%
  }%
  \providecommand\transparent[1]{%
    \errmessage{(Inkscape) Transparency is used (non-zero) for the text in Inkscape, but the package 'transparent.sty' is not loaded}%
    \renewcommand\transparent[1]{}%
  }%
  \providecommand\rotatebox[2]{#2}%
  \newcommand*\fsize{\dimexpr\f@size pt\relax}%
  \newcommand*\lineheight[1]{\fontsize{\fsize}{#1\fsize}\selectfont}%
  \ifx\svgwidth\undefined%
    \setlength{\unitlength}{362.78365008bp}%
    \ifx\svgscale\undefined%
      \relax%
    \else%
      \setlength{\unitlength}{\unitlength * \real{\svgscale}}%
    \fi%
  \else%
    \setlength{\unitlength}{\svgwidth}%
  \fi%
  \global\let\svgwidth\undefined%
  \global\let\svgscale\undefined%
  \makeatother%
  \begin{picture}(1,0.63051097)%
    \lineheight{1}%
    \setlength\tabcolsep{0pt}%
    \put(0.76484712,0.27717898){\color[rgb]{0,0,0}\makebox(0,0)[lt]{\lineheight{1.25}\smash{\begin{tabular}[t]{l}$q_{i,0}$\end{tabular}}}}%
    \put(0.74083376,0.46223505){\color[rgb]{0,0,0}\makebox(0,0)[lt]{\lineheight{1.25}\smash{\begin{tabular}[t]{l}$\hat{D}_{i}$\end{tabular}}}}%
    \put(0.39630712,0.3294134){\color[rgb]{0,0,0}\makebox(0,0)[lt]{\lineheight{1.25}\smash{\begin{tabular}[t]{l}$\Pi_{i,\theta}$\end{tabular}}}}%
    \put(0,0){\includegraphics[width=\unitlength,page=1]{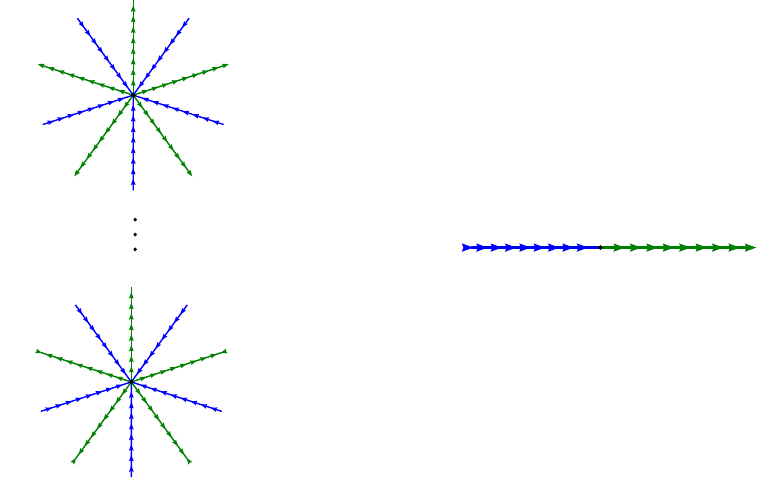}}%
    \put(-0.00033978,0.31821978){\color[rgb]{0,0,0}\makebox(0,0)[lt]{\lineheight{1.25}\smash{\begin{tabular}[t]{l}$A^{\mathrm{inv}}_{i,\theta}$\end{tabular}}}}%
    \put(0,0){\includegraphics[width=\unitlength,page=2]{multiprong.pdf}}%
  \end{picture}%
\endgroup%
}
	\caption{On the left we see the preimage on $\Ainv_{i,\theta}$ of a 
	neighborhood of the point $q_{i,0} \in \hat{D}_i$ by $\Pi_{i,\theta}$.}
	\label{fig:multiprong}
\end{figure}

For each $b$, with $1 \leq b \leq m_i - m_{ik}$, we also define a
\emph{dual cycle} to $S_{i,0,\theta}[b]$ as follows.
Let $R$ and $R'$ be the two outgoing prongs adjacent to the
incoming prong $P_{i,0,\theta}[b]$, as in \cref{fig:multiprong}.
Here, $R$ is on the right of $P_{i,0,\theta}[b]$, and $R'$ is on the left.
Orient $R$ as incoming, and $R'$ as outgoing. We then set
\[
  U_{i,0,\theta}[b]
  =
  R \cup R' \cup \{ q_{i,0,\theta}[b \mod m_{ik}] \}
\]

As in the saddle point case, these orientations are chosen so that the signed 
intersection 
$\langle S_{i,0,\theta}[b], U_{i,0,\theta}[b]\rangle=+1$ is positive.

Let $q_{0,-1} \in \hat D_0$ be the unique repeller of $\hat\xi_0$
as in \Cref{prop:hat_xi}, \cref{it:repller}. For any $\theta \in \R/2\pi\Z$, 
the point $q_{0,-1}$
has precisely $e = m_0$ preimages in $\Ainv_{0,\theta}$. Call them
$q_{0,-1,\theta}[1],\ldots,q_{0,-1,\theta}[e]$, in such a way that
for $b = 1,\ldots,e$, we have
\[
  \talpha(q_{0,-1,\theta}[b])
  \in
  \left[
  \frac{(b-1)2\pi}{e},
  \frac{b2\pi}{e}
  \right).
\]
For $1 \leq b \leq e$, we set
\[
\begin{split}
  S_{0,0,\theta}[b]
  &= \{ q_{0,-1,\theta}[b] \} = (\omega^+)^{-1}(q_{0,-1,\theta}[b]),\\
  U_{0,0,\theta}[b]
  &= (\omega^-)^{-1}(q_{0,-1,\theta}[b]).\\
\end{split}
\]
Note that $S_{0,0,\theta}[b]$ is just a point and $U_{0,0,\theta}[b]$ is a 
topological disk since $q_{0,-1,\theta}[b]$ is a fountain of $\xiinv_{\theta}$.

\begin{definition}
Let $i \in \V_\Upsilon$. If $1\leq a \leq p_i$, set
\[
\begin{split}
  C_{1,i,a,\theta}
  &=
  \Z
  \gen{S_{i,a,\theta}[b]}{1\leq b \leq m_i}, \\
  C^\vee_{1,i,a,\theta}
  &=
  \Z
  \gen{U_{i,a,\theta}[b]}{1\leq b \leq m_i},
\end{split}
\]
as well as
\[
\begin{split}
  C_{1,i,0,\theta}
  &=
  \Z
  \gen{S_{i,0,\theta}[b]}{1\leq b \leq m_i-m_{ik}} \\
  C_{1,i,0,\theta}
  &=
  \Z
  \gen{U_{i,0,\theta}[b]}{1\leq b \leq m_i-m_{ik}}
\end{split}
\]
in case $i$ has a  neighbor $k \in \V_\Gamma \setminus \V_\Upsilon$, otherwise 
we define
$C^{\phantom{\vee}}_{1,i,0,\theta}$ and $C^\vee_{1,i,0,\theta}$
to be trivial $\Z$-modules. We define
\[
\begin{aligned}
  &C_{1,i,\theta}
  =
  \bigoplus_{a=0}^{p_i}
  C_{1,i,a,\theta},
  &\qquad
  &C_{1,\theta}
  =
  \bigoplus_{i \in \V_{\Upsilon}}
  C_{1,i,\theta} \\
  &C^\vee_{1,i,\theta}
  =
  \bigoplus_{a=0}^{p_i}
  C^\vee_{1,i,a,\theta},
  &\qquad
  &C^\vee_{1,\theta}
  =
  \bigoplus_{i \in \V_{\Upsilon}}
  C^\vee_{1,i,\theta}.
\end{aligned}
\]

\begin{block}\label{blc:index_sets}
Let 
\[
I_{1} \subset \V_{\Upsilon} \times \Z \times \Z
\]
be defined by 
$(i,a,b) \in I_1$ if and only if $S_{1,i,a,\theta}[b] \in C_{1, \theta}$. By 
\Cref{lem:nice_total_order}, the set $\V_{\Upsilon}$ is totally ordered so the 
set 
$\V_{\Upsilon} \times \Z_{\leq 0} \times \Z_{\leq 1}$ has a lexicographic order 
which induces a lexicographic order on $I_{1}$. Similarly, we define 
\[
I_{1,i} = 
\{i\} \times \Z \times \Z \cap I_1 \subset I_1
\]
and
\[
I_{1,i,a} = 
\{i\} \times \{a\} \times \Z \cap I_1 \subset I_{1,i}
\]
\end{block}

We define also
\[
\begin{split}
  C_{0,\theta}
  &=
  \Z\left\langle S_{0,0,\theta}[1],\ldots,S_{0,0,\theta}[e] \right\rangle\\
  C^\vee_{2,\theta},
  &=
  \Z\left\langle U_{0,0,\theta}[1],\ldots,U_{0,0,\theta}[e] \right\rangle.
\end{split}
\]
Finally, if $j \neq 0,1$, then let $C_{j,\theta}$ be the trivial
$\Z$-module. Similarly, if $j \neq 1,2$, then let $C^\vee_{j,\theta}$
be the trivial $\Z$-module. We then have graded $\Z$-modules
\[
  C_\theta = \bigoplus_{j\in \Z} C_{j,\theta},\qquad
  C^\vee_\theta = \bigoplus_{j\in \Z} C^\vee_{j,\theta}.
\]
These graded modules are naturally seen as submodules of
the singular chain complex of $\Ainv_\theta$, and the
sinugular chain complex of $\Ainv_\theta$ relative to
$\partial\Ainv_\theta$, respectively, that is
\begin{equation} \label{eq:C_inclusions}
  C_\theta \subset C(\Ainv_\theta;\Z), \qquad
  C^\vee_\theta \subset C(\Ainv_\theta,\partial\Ainv_\theta;\Z).
\end{equation}
These are, in fact, subcomplexes, that is, they are invariant under
the differential.
Denote the restrictions by
\[
  d_{j,\theta}:C_{j,\theta} \to C_{j-1,\theta},\qquad
  d^\vee_{j,\theta}:C^\vee_{j,\theta} \to C^\vee_{j-1,\theta}.
\]
There is a natural pairing between the two complexes
$C_\theta$ and $C^\vee_\theta$, by requiring the bases described above
to be dual bases. This means that we consider as an identification
the isomorphism
\[
  C^\vee_{j,\theta} \to \Hom(C_{2-j,\theta}, \Z)
\]
whose image of $U_{i,a,\theta}[b]$ is dual to $S_{i,a,\theta}[b]$, and so on.
\end{definition}

\begin{block}
The inclusions \cref{eq:C_inclusions} are quasiisomorphisms. As a result,
the homology groups of $C_\theta$ and $C^\vee_\theta$ are naturally
isomorphic to the absolute and relative homology groups of $\Ainv_\theta$.
Furthermore,
the restriction of the pairing 
\[
\langle \cdot, \cdot \rangle: C_\theta \times C^\vee_\theta \to \Z 
\] 
between 
$C_\theta$ and $C^\vee_\theta$
to cycles coincides with the natural pairing between cycles
in $Z_j(\Ainv;\Z)$ and relative cycles in $Z_{2-j}(\Ainv,\partial\Ainv;\Z)$.
\end{block}

\subsection{The action on the chain complex}
\label{ss:action_complex}

Having built the chain complex $(C_\theta, d_\theta)$ in
the previous section, we are now in a position to
compute the algebraic monodromy. The geometric monodromy
$G_\theta: \Ainv_\theta \to \Ainv_\theta$ (defined in \Cref{def:Ainv})
is a diffeomorphism that permutes the trajectories of the
flow $\xiinv$ in the periodic pieces and interpolates between them in the 
connecting annuli.

In this section, we compute the matrix $B_\theta$ of this
action with respect to the basis of chains
$\{S_{i,a,\theta}[b]\}$ defined previously. This matrix $B_\theta$
is the algebraic representation of the monodromy operator
$(G_\theta)_*$ acting on the defined CW-complex of the fiber.

\subsubsection*{Action on $C_{0,\theta}$}

We define  $B_{0,\theta}: C_{0, \theta } \to C_{0, \theta}$ as the permutation 
matrix
\[
\left(
\begin{matrix}
	0 & 0 & 0 & \cdots & 1 \\
	1 & 0 & 0 & \cdots & 0 \\
	0 & 1 & 0 & \cdots & 0 \\
	\vdots & \vdots  &  \vdots  & \ddots & \vdots \\
	0 & 0 & 0 & \cdots & 0 \\	
\end{matrix}
\right)
\]
of size $e = m_0$.
\subsubsection*{Action on $C_{1, \theta}$}

We define the matrix $B_{1,\theta}$ acting on the chain complex as

\[
B_{1, \theta}
=
\left(
\langle
G_\theta(S_{i,a,\theta}[b]) ,  U_{i',a',\theta}[b]
\rangle
\right)_{(i,a,b),(i',a',b') \in I_1}.
\]
Here the intersection number is taken viewing $ U_{i',a',\theta}[b]$ as a 
relative 
cycle in $\Ainv_{\theta}$ relative to $\partial \Ainv_{\theta}$, and 
$G_\theta(S_{i,a,\theta}[b])$ as a cycle in $\Ainv_{\theta}$ relative to the 
set of points $q_{0, -1,\theta}[1], \ldots, q_{0, -1,\theta}[e]$.
We also recall that $I_1$ (\cref{blc:index_sets}) is a totally ordered index 
set.
Equivalently, $B_{1, \theta}$ is the linear operator $B_{1,\theta}: C_{1, 
\theta} \to C_{1,\theta}$ defined by
\begin{equation}\label{eq:def_B1}
B_{1,\theta}(S_{i,a,\theta} [b]) = \sum_{(i',a',b') 
	\in I_1} \langle 
G_{\theta}(S_{i,a,\theta} [b]) ,
U_{i',a',\theta}[b'] \rangle S_{i',a',\theta}[b'].
\end{equation}
We also give a name to some blocks of the matrix:
\[
B_{1,i,\theta} 
= 
\left(
\langle
G_\theta(S_{i,a,\theta}[b]) , U_{i',a',\theta}[b']
\rangle
\right)_{(i,a,b),(i',a',b') \in I_{1,i}}
\]
and	
\[
B_{1,i,a,\theta} 
= 
\left(
\langle
G_\theta(S_{i,a,\theta}[b]) , U_{i',a',\theta}[b']
\rangle
\right)_{(i,a,b),(i',a',b') \in I_{1,i,a}}
\]

Next, we show the form that some of these blocks take.
\begin{lemma}
The matrix $B_{1,\theta}$ is block upper triangular,
with block diagonal entries the submatrices $B_{1,i,a,\theta}$.

If $a \neq 0$, the matrix $B_{1,i,a,\theta}$ is the $m_i \times m_i$ 
permutation matrix
\[
B_{1,i,a,\theta}
=
\left(
\begin{matrix}
	0 & 0 & 0 & \cdots & 1 \\
	1 & 0 & 0 & \cdots & 0 \\
	0 & 1 & 0 & \cdots & 0 \\
	\vdots & \vdots  &  \vdots  & \ddots & \vdots \\
	0 & 0 & 0 & \cdots & 0 \\	
\end{matrix}
\right)
\]
which is the companion matrix to the polynomial $t^{m_i}-1$.

If $a=0$,  the matrix $B_{1,i,0,\theta}$ is the $(m_{i} - m_{ik}) \times (m_{i} 
- m_{ik})$ matrix of the form
\[
B_{1,i,0,\theta}
=
\left(
\begin{matrix}
	-c_{m_i - m_{ik}-1} & -c_{m_i - m_{ik}-2} & \cdots &-c_1& 
	-c_0 \\
	1 & 0 & \cdots &0 & 0 \\
	0 & 1 & \cdots &0& 0 \\
	\vdots &  \vdots  & \ddots&\vdots & \vdots \\
	0 &  0 & \cdots & 1 & 0 \\	
\end{matrix}
\right)
\]
which is the inverse transpose of the companion matrix to the polynomial
\[
\frac{t^{m_i} -1}{t^{m_{ik}} -1} 
= 
\sum_{\ell =0}^{m_i - m_{ik}} c_\ell t^\ell 
=
t^{m_i - m_{ik}} + t^{m_i - 2m_{ik}} + t^{m_i - 3m_{ik}} + \cdots + 1.
\]
Furthermore, the matrix $B_{1,i,\theta}$ is the diagonal block matrix
\[
B_{1,i,\theta} = \bigoplus_{a} B_{1,i,a,\theta}
\]
\end{lemma}
 
\begin{proof}
For the first statement on the block upper triangular structure of
$B_{1,\theta}$, first consider two distinct vertices $i,k$ in $\Gamma$.
The chain $S_{i,a,\theta}[b]$, as well as its image by $G_\theta$
is contained in the pieces
\[
  \Ainv_{\ell_0}, \Ainv_{\ell_0,\ell_1}, \Ainv_{\ell_1}
  \ldots,
  \Ainv_{\ell_{s-1},\ell_s}
  \Ainv_{\ell_s}
\]
where $0 = \ell_0 \to \ell_1 \to \ldots \to \ell_s = i$ is the geodesic
in $\Gamma$ between $0$ and $i$. Similarly, both prongs in $U_{k,c,\theta}[d]$
are contained in unions of pieces corresponding to directed geodesics
in $\Gamma$ starting at $k$. As a result, these cycles are disjoint,
unless there is a directed geodesic starting at $i$ and ending at $k$.
Thus, the matrix $B_{1,\theta}$ is block upper triangular with diagonal
blocks the matrices $B_{1,i,\theta}$. We will see below that these are
block diagonal matrices with diagonal blocks
$B_{1,i,a,\theta}$, proving the first statement.

	This proof proceeds by analyzing the three statements of the lemma. The 
	central element is the action of the monodromy map $G_\theta: \Ainv_\theta 
	\to \Ainv_\theta$ on the chain complex $C_{1,\theta}$.
	
	\paragraph{Proof of $B_{1,i,\theta} = \bigoplus_{a} B_{1,i,a,\theta}$}
	
	The intersection pairing $\langle \cdot, \cdot \rangle$ between a chain in 
	$\Ainv_i$ and a chain in $\Ainv_{i'}$ is zero unless $i=i'$. This implies 
	$B_{1,\theta}$ is block-diagonal with respect to the vertex index $i$, 
	i.e., $B_{1,\theta} = \bigoplus_i B_{1,i,\theta}$.
	
	Furthermore, within $\Ainv_i$, the map $G_\theta|_{\Ainv_i} = G_{i,\theta}$ 
	is the generator of the Galois group of the cover $\Pi_{i,\theta}: 
	\Ainv_{i,\theta} \to \widehat D_i$. The map $G_\theta$ permutes the points 
	of the finite  fiber over any point 
	$\hat{q} \in \widehat D_i$.
	The chains $S_{i,a,\theta}[b]$ are stable manifolds of the singularities 
	$q_{i,a,\theta}[b]$. The chain $U_{i,a',\theta}[b']$ is the unstable 
	manifold of a singularity $q_{i,a',\theta}[b']$ in the fiber over $q_{i,a'} 
	\in \widehat D_i$.  Thus, 
	$G_\theta(S_{i,a,\theta}[b])$ is a chain that contains a point in the fiber 
	over 
	$q_{i,a}$ and similarly for the (relative) chain $U_{i,a',\theta}[b']$. 
	Since $G_\theta$ sends trajectories to trajectories and singularities to 
	singularities within the pieces $\Ainv_{i,\theta}$, and since the sequences 
	of vertices by which 
	$U_{i,a',\theta}[b']$ passes and by which $S_{i,a,\theta}[b]$ passes only 
	concide at $i$, we deduce that the 
	intersection pairing
	\[
	\langle G_\theta(S_{i,a,\theta}[b]) , U_{i,a',\theta}[b'] \rangle
	\]
	can only be non-zero if their corresponding singularity lies over the same 
	singularity of $\hat{\xi}$ in $\hat{D}_i$, which requires $a = a'$.
	Therefore, the matrix $B_{1,i,\theta}$ is itself block-diagonal with 
	respect to the index $a$, proving the third statement of the lemma.
	
	\paragraph{Proof for $a \neq 0$ (Saddle points)}
	We now compute the block $B_{1,i,a,\theta}$ for $a \neq 0$. The relevant 
	elements for this block are
	$\{S_{i,a,\theta}[b]\}_{b=1}^{m_i}$, and their corresponding duals are 
	$\{U_{i,a,\theta}[b]\}_{b=1}^{m_i}$.
	The singularities $\{q_{i,a,\theta}[b]\}_{b=1}^{m_i}$ are the $m_i$ 
	saddle points which are preimages of the saddle point $q_{i,a} \in 
	\hat{D}_i$ of $\hat{\xi}$.
	They are explicitly labeled by their $\alpha$  coordinates such that:
	\[
	\alpha_i\left(q_{i,a,\theta}[b] \right) \in \left[ \frac{2 \pi (b-1)}{m_i}, 
	\frac{2 \pi b}{m_i}\right) + 2 \pi \Z.
	\]
	The monodromy map $G_\theta = G_{i,\theta}$ is the generator of this 
	$m_i$-sheeted cyclic cover. Its action on the fiber corresponds to adding 
	$\frac{2\pi}{m_i}$ to the $\alpha_i$-coordinate.
	This action permutes cyclically the labeled points.

	In particular, the action on the points is $G_\theta(q_{i,a,\theta}[b]) = 
	q_{i,a,\theta}[b+1]$ (for $b<m_i$) and $G_\theta(q_{i,a,\theta}[m_i]) = 
	q_{i,a,\theta}[1]$.
	By definition, $G_\theta$ maps the stable manifold of a 
	point to the stable manifold of its image within the piece 
	$\Ainv_{i,\theta}$:
	\begin{itemize}
		\item $G_\theta(S_{i,a,\theta}[b]) \cap \Ainv_{i,\theta} = 
		S_{i,a,\theta}[b+1]  \cap \Ainv_{i,\theta}$, for $b = 1, 
		\ldots, m_i-1$.
		\item $G_\theta(S_{i,a,\theta}[m_i]) \cap \Ainv_{i,\theta} = 
		S_{i,a,\theta}[1] \cap \Ainv_{i,\theta}$.
	\end{itemize}
	The entries of the matrix $B_{1,i,a,\theta}$ are $(B)_{b',b} = \langle 
	G_\theta(S_{i,a,\theta}[b]), U_{i,a,\theta}[b'] \rangle$.
	\begin{itemize}
		\item For a column $b < m_i$: The entry is $\langle 
		S_{i,a,\theta}[b+1], U_{i,a,\theta}[b'] \rangle = \delta_{b+1, b'}$. 
		This gives a $1$ in row $b+1$ and $0$s elsewhere.
		\item For the column $b = m_i$: The entry is $\langle 
		S_{i,a,\theta}[1], U_{i,a,\theta}[b'] \rangle = \delta_{1, b'}$. This 
		gives a $1$ in row $1$ and $0$s elsewhere.
	\end{itemize}
	The resulting matrix is:
	\[
	B_{1,i,a,\theta} =
	\begin{pmatrix}
		0 & 0 & 0 & \cdots & 0 & 1 \\
		1 & 0 & 0 & \cdots & 0 & 0 \\
		0 & 1 & 0 & \cdots & 0 & 0 \\
		\vdots & \vdots & \vdots & \ddots & \vdots & \vdots \\
		0 & 0 & 0 & \cdots & 1 & 0 \\	
	\end{pmatrix}
	\]
	This is the companion matrix for the polynomial $p(t) = t^{m_i} - 1$. For 
	this polynomial, $n=m_i$, $c_0 = -1$, and $c_1 = \ldots = c_{m_i-1} = 0$. 
	The last column of the companion matrix is $(-c_0, -c_1, \ldots, 
	-c_{m_i-1})^T = (1, 0, \ldots, 0)^T$. This matches our computed matrix (the 
	$1$ is in the first row, last column).
	
\paragraph{Proof for $a = 0$ (Multipronged singularities)}
For $a=0$, we consider the block $B_{1,i,0,\theta}$. Let $t$ be the variable 
representing the action of the monodromy $G_\theta$. We first establish the 
structure of $C_{1,i,0,\theta}$ as a $\mathbb{Z}[t]$-module.

The map $\Pi_{i,\theta}: \Ainv_{i,\theta} \to \widehat{D}_i$ is an 
$m_i$-sheeted cyclic cover. Let $\hat{q}$ be a point in $\hat{D}_i$ on the 
trajectory of $\hat{\xi}_i$ escaping from$q_{i.0}$  (on the positive axis) . 
Let 
$C_{\text{fiber}}(\hat{q})$ denote the free 
$\mathbb{Z}$-module on the fiber $\Pi_{i,\theta}^{-1}(\hat{q})$.
\begin{itemize}
	\item The fiber over $\hat{q}$ has $m_i$ 
	points. The module of unstable prongs $C_{\text{uprongs}} := 
	C_{\text{fiber}}(\hat{q})$ is 
	isomorphic to $\mathbb{Z}[t]/(t^{m_i}-1)$ as a $\mathbb{Z}[t]$-module.
	\item For the singular point $q_{i,0}$ itself, the fiber has $m_{ik}$ 
	points. The module $C_{\text{points}} := C_{\text{fiber}}(q_{i,0})$ is 
	isomorphic to $\mathbb{Z}[t]/(t^{m_{ik}}-1)$.
\end{itemize}
There is a natural  map $\phi: C_{\text{uprongs}} \to 
C_{\text{points}}$ 
that sends each prong to the singular point it converges 
to. This map is a surjective $\mathbb{Z}[t]$-module homomorphism. The module 
$C_{1,i,0,\theta}$ is, by its construction from the prongs in \cref{eq:Si0}, 
precisely the kernel of this map, $\ker(\phi)$.

We have a short exact sequence of $\mathbb{Z}[t]$-modules:
\[
0 \to C_{1,i,0,\theta}^\vee \to C_{\text{uprongs}} \xrightarrow{\phi} 
C_{\text{points}} \to 0
\]
Using the module isomorphisms, this becomes:
\[
0 \to \ker(\phi) \to \mathbb{Z}[t]/(t^{m_i}-1) \xrightarrow{\phi} 
\mathbb{Z}[t]/(t^{m_{ik}}-1) \to 0
\]
The kernel of the canonical projection $\phi$ is the ideal generated by 
$(t^{m_{ik}}-1)$ inside the ring $\mathbb{Z}[t]/(t^{m_i}-1)$. We thus have
\[
C_{1,i,0,\theta}^\vee \cong \ker(\phi) = \frac{(t^{m_{ik}}-1)}{(t^{m_i}-1)} 
\subset 
\frac{\mathbb{Z}[t]}{(t^{m_i}-1)}.
\]
This kernel is a cyclic module generated by the element $g = (t^{m_{ik}}-1)$. 
By the isomorphism theorem, this module is isomorphic to $\mathbb{Z}[t] / 
\text{Ann}(g)$, where $\text{Ann}(g)$ is the annihilator ideal of $g$ in 
$\mathbb{Z}[t]$. A polynomial $P(t) \in \mathbb{Z}[t]$ is in $\text{Ann}(g)$ if 
$P(t) \cdot g = 
0$ in the module $\mathbb{Z}[t]/(t^{m_i}-1)$. This means
\[
P(t) \cdot (t^{m_{ik}}-1) \text{ is a multiple of } (t^{m_i}-1) \text{ in } 
\mathbb{Z}[t].
\]
Let $Q(t) = (t^{m_i}-1) / (t^{m_{ik}}-1)$. Since $m_{ik}$ divides $m_i$, $Q(t)$ 
is a polynomial in $\mathbb{Z}[t]$. The condition becomes:
\[
P(t) \cdot (t^{m_{ik}}-1) = R(t) \cdot Q(t) \cdot (t^{m_{ik}}-1) \quad 
\text{for some } R(t) \in \mathbb{Z}[t].
\]
This implies $P(t)$ must be a multiple of $Q(t)$. The annihilator ideal is 
therefore $\text{Ann}(g) = (Q(t))$. We have thus shown:
\[
C_{1,i,0,\theta}^\vee \cong \mathbb{Z}[t] / (Q(t)), \quad \text{where } Q(t) = 
\frac{t^{m_i} - 1}{t^{m_{ik}} - 1} = \sum_{\ell=0}^{N} c_\ell t^\ell,
\]
and $N = m_i - m_{ik} = \deg(Q(t))$.

The basis $\{U_{i,0,\theta}[b]\}_{b=1}^{N}$ is, by construction (cf. 
\cref{eq:Si0}), chosen to be a standard $\mathbb{Z}$-basis for this cyclic 
module, corresponding to $\{U[1], t U[1], t^2 U[1], \ldots, t^{N-1} U[1]\}$.
Let $U[b] = U_{i,0,\theta}[b] \leftrightarrow t^{b-1} U[1]$. The action of 
$G_\theta$ (multiplication by $t$) is:
\begin{itemize}
	\item For $b \in \{1, \ldots, N-1\}$:
	$G_\theta(U[b]) = G_\theta(t^{b-1} U[1]) = t^b U[1] = U[b+1]$.
	
	\item For $b = N$:
	$G_\theta(U[N]) = G_\theta(t^{N-1} U[1]) = t^N U[1]$.
\end{itemize}
Since $U[1]$ is a generator and $Q(t) \cdot U[1] = 0$, we have:
\[
(t^N + c_{N-1}t^{N-1} + \ldots + c_1 t + c_0) U[1] = 0
\]
Solving for $t^N U[1]$, we get:
\[
t^N U[1] = -c_{N-1} t^{N-1} U[1] - \ldots - c_1 t U[1] - c_0 U[1]
\]
Substituting back $U[b] = t^{b-1} U[1]$:
\[
G_\theta(U[N]) = -c_{N-1} U[N] - \ldots - c_1 U[2] - c_0 U[1]
\]
Now we write the matrix $B_{1,i,0,\theta}^\vee$ whose entries are $(B)_{b',b} = 
\langle G_\theta(S[b]), U[b'] \rangle$.
\begin{itemize}
	\item For a column $b < N$: $G_\theta(U[b]) = U[b+1]$. The column has a $1$ 
	in row $b+1$ and $0$s elsewhere, since $\langle S[b+1], U[b'] \rangle = 
	\delta_{b+1, b'}$.
	\item For the column $b = N$: $G_\theta(U[N]) = \sum_{\ell=0}^{N-1} -c_\ell 
	U[\ell+1]$. The column vector is $(-c_0, -c_1, \ldots, -c_{N-1})^T$.
\end{itemize}
This gives the matrix:
\[
B_{1,i,0,\theta}^\vee =
\left(
\begin{matrix}
	0 & 0 & 0 & \cdots & 0 & -c_0 \\
	1 & 0 & 0 & \cdots & 0 & -c_1 \\
	0 & 1 & 0 & \cdots & 0 & -c_2 \\
	\vdots & \vdots & \vdots & \ddots & \vdots & \vdots \\
	0 & 0 & 0 & \cdots & 1 & -c_{N-1} \\	
\end{matrix}
\right)
\]
where $N = m_i - m_{ik}$. This is precisely the companion matrix for the 
polynomial $Q(t) = \sum_{\ell=0}^{N} c_\ell t^\ell$.
Finally, $B_{1,i,0,\theta}^\vee = \left(B_{1,i,0,\theta}^{-1}\right)^T$.
\end{proof}
 
 \begin{notation} \label{not:B_matrix}
 	We denote by $B_\theta = B_{0,\theta} \oplus B_{1, \theta}$ the induced 
 	operator on $C_\theta = C_{0, \theta} \oplus C_{1, \theta}$ by the 
 	previously defined two operators.
 \end{notation}

\subsection{Action on homology}

 \begin{lemma}\label{lem:action_on_homology}
	The operator $B_{\theta}$ acts on the complex $(C_{\theta}, d_\theta)$. And 
	the induced map on homology 
		coincides with the map $(G_\theta)_\ast$ induced on  homology by 
		the geometric monodromy $G_\theta$.
\end{lemma}

\begin{proof}
	First observe that the action of $B_{0, \theta}$ and $G_\theta$ on the 
	points that generate $C_{0,\theta}$ coincides. Denote the set of these
points by $S_0$ in this proof.
On one hand this implies 
	that $(G_\theta)_\ast$ acts on the relative homology group 
	$H_1(\Ainv_{\theta}, S_0; \Z)$. On the other hand, we have the 
	sequence of isomorphisms
$H_1(\Ainv_{\theta}, S_0; \Z) \simeq
\tilde{H}_1(\Ainv_{\theta} / S_0; \Z)  \simeq C_{1,\theta}$. 
	So we are 
	left to prove that the actions of $B_{1, \theta}$ on $C_{1,\theta}$ and 
	$(G_\theta)_\ast$ on  $H_1(\Ainv_{\theta}, S_0); \Z)$ coincide via 
	the above isomorphisms. But this 
	follows from the construction of $B_{1,\theta}$ in \cref{eq:def_B1} and the 
	definition of $(G_\theta)_\ast$.
\end{proof}

This lemma confirms that our algebraic construction $B_\theta$
faithfully represents the topological monodromy. We now turn
to the other principal invariant, the variation operator.

\section{Description of the variation operator}
\label{s:variation_operator}

In the previous section, we analyzed the action of the monodromy
$G_\theta$. In this section, we describe the second key topological
invariant: the variation operator. The classical variation operator,
$\text{Var}: H_1(\Ainv_\theta, \partial \Ainv_\theta; \Z) \to H_1(\Ainv_\theta; 
\Z)$,
measures the difference between a relative cycle and its image under
the monodromy.

We will define a chain-level operator $V: C^\vee_{1,\theta} \to C_{1,\theta}$
that acts between our dual complex and the primary complex.
As shown in \Cref{lem:var_hom}, this operator $V$ is the 
realization on the chain complex of the classical variation operator.

 \subsection{$V$ on the chain complex}
 
 We define the operator $V: C^\vee_{1,\theta} \to C_{1,\theta}$ by the formula
 
\begin{equation}\label{eq:variation_V}
V_{1,\theta}(U_{i,a,\theta} [b]) = \sum_{(i',a',b') 
	\in I_1} \langle 
G_{\theta}(U_{i,a,\theta} [b]) - U_{i,a,\theta} [b] ,
U_{i',a',\theta}[b'] \rangle \, S_{i',a',\theta}[b'].
\end{equation}
Note that since $G_{\theta}$ equals the identity map on $\partial 
\Ainv_{\theta}$, then  $G_{\theta}(U_{i,a,\theta} [b]) - U_{i,a,\theta} [b]$ is 
an absolute cycle and the intersection product of the formula above is well 
defined.

 \subsection{$\textrm{Var}$ on homology}

By a similar argument as the one in \Cref{lem:action_on_homology}, we get

\begin{lemma}\label{lem:var_hom}
	The linear operator $V$ defines a map $H_1(\Ainv_{\theta}, \partial 
	\Ainv_{\theta}; \Z) \to H_1(\Ainv_{\theta}; \Z)$ which coincides with the 
	classical variation operator.
\end{lemma}

\begin{proof}
	We first show $V$ induces a map on homology $V_*: H_1(C^\vee_\theta) \to 
	H_1(C_\theta)$.
	\begin{enumerate}
		\item \textbf{$V$ maps relative cycles to absolute cycles:}
		Let $U \in Z^\vee_{1,\theta}$ (so $\partial U = 0$). The chain $V(U) = 
		G_\theta(U) - U$ is an absolute cycle because $G_\theta$ acts as the 
		identity on $\partial \Ainv_\theta$, which contains the boundary of 
		$U$. So,
		\[
		\partial(V(U)) = \partial(G_\theta(U) - U) = \partial(G_\theta(U)) - 
		\partial U = G_\theta(\partial U) - \partial U = G_\theta(0) - 0 = 0.
		\]
		So $V(U) \in Z_{1,\theta}$.
		
		\item \textbf{$V$ maps relative boundaries to absolute boundaries:}
		Let $U \in B^\vee_{1,\theta}$, so $U = \partial^\vee W$ for some $W \in 
		C^\vee_{2,\theta}$. Since $G_\theta$ commutes with $\partial$:
		\[
		V(U) = G_\theta(\partial^\vee W) - \partial^\vee W = 
		\partial(G_\theta(W)) - \partial W = \partial(G_\theta(W) - W).
		\]
		Since $G_\theta(W) - W$ is an absolute 2-chain, $V(U)$ is an absolute 
		1-boundary.
	\end{enumerate}
	Therefore, $V$ descends to a well-defined homomorphism $V_*: 
	H_1(C^\vee_\theta) \to H_1(C_\theta)$ given by $V_*([U]) = [G_\theta(U) - 
	U]$.
	
	Next, we show this map coincides with the classical variation operator, 
	$\text{Var}_{\text{classic}}$. The inclusions 
	$C^\vee_\theta \hookrightarrow C(\Ainv_\theta, \partial \Ainv_\theta)$ and 
	$C_\theta \hookrightarrow C(\Ainv_\theta)$ are quasi-isomorphisms, inducing 
	isomorphisms on homology:
	\begin{align*}
		\Psi &: H_1(C^\vee_\theta) \xrightarrow{\cong} H_1(\Ainv_\theta, 
		\partial \Ainv_\theta) \\
		\Phi &: H_1(C_\theta) \xrightarrow{\cong} H_1(\Ainv_\theta)
	\end{align*}
	The operator $V$ from the lemma is the map $\text{Var}_{\text{lemma}} = 
	\Phi \circ V_* \circ \Psi^{-1}$. The classical operator is 
	$\text{Var}_{\text{classic}}([c]) = [G_\theta(c) - c]$.
	
	Let $[c] \in H_1(\Ainv_\theta, \partial \Ainv_\theta)$ be an arbitrary 
	class. Let $[U] = \Psi^{-1}([c])$ be its corresponding class in 
	$H_1(C^\vee_\theta)$. The chain $U$ is thus a representative cycle for 
	$[c]$.
	\begin{itemize}
		\item Applying $\text{Var}_{\text{lemma}}$:
		\[
		\text{Var}_{\text{lemma}}([c]) = (\Phi \circ V_* \circ \Psi^{-1})([c]) 
		= \Phi(V_*([U])) = \Phi([G_\theta(U) - U]) = [G_\theta(U) - U].
		\]
		\item Applying $\text{Var}_{\text{classic}}$:
		\[
		\text{Var}_{\text{classic}}([c]) = [G_\theta(U) - U].
		\]
	\end{itemize}
	Since both maps yield the same absolute homology class $[G_\theta(U) - U] 
	\in H_1(\Ainv_\theta)$, they are identical.
\end{proof}

\section{Gyrographs}
\label{s:gyrographs}

\begin{block}
	In this section, we provide a method which can simplify the direct
	calculation of the matrices for monodromy and variation map.
	If $\theta$ is a generic angle, then the invariant spine $\Sinv_\theta$
	naturally has the structure of a conformal ribbon graph, being embedded
	in a surface, where near the vertices, we have a natural conformal
	structure. This data, together with the Hironaka numbers, allow us to
	calculate by hand the variation map for plane curve singularities.
\end{block}

\subsection{A general construction}
\label{ss:general_construction}
\begin{block}
	We denote by $S$ a graph, with vertex set $\V$
	and edge set $\E$. We will assume that every edge has two distinct
	adjacent vertices, i.e. there are no loops, but we do allow for multiple
	edges joining the same pair of vertices.
	Denote by $\E_v$ the set of edges adjacent to a vertex $v \in \V$.
	A \emph{ribbon graph} structure on $S$ allows us to define safe
	walks in $S$. Recall that a safe walk is a path on the ribbon graph that 
	turns 
	right at each vertex, see \cite{AC_tat}.
	To allow for flexibility, we do not always make an explicit distinction
	between a graph, and its topological realization.
	Throughout, we will consider a special subset of vertices $\Rep \in \V$
	in $S$. We will assume  that any edge in $S$ does not have
	both its vertices in $\Rep$, a condition that holds for
	the invariant spine, with all singularities seen as vertices.
	In general, we can subdivide any such edge in two by adding a vertex
	not in $\Rep$, to avoid any problems.
\end{block}

\begin{definition}
	A \emph{conformal ribbon graph} is a graph $S$, along with
	angles $a_v(e,f) \in \R / 2\pi\Z$ associated to any pair of edges
	$e,f$ adjacent to any vertex $v$ in $S$, satisfying
	\[
	a_v(e,g) = a_v(e,f) + a_v(f,g),\qquad
	a_v(e,f) \neq 0 \quad \mathrm{if}\quad e\neq f
	\]
	for any edge $g$ also adjacent to $v$. At any vertex $v\in \V$ we have an 
	$\R/2\pi\Z$-torsor $\Theta_v$,
	given as 
	\[
	\coprod_{e\in \E_v} \{e\} \times (\R/2\pi\Z) / \sim, \qquad
	(e,\alpha) \sim (f,\alpha + a_v(e,f)),\; e,f\in\E_v.
	\]
	For any $\Rep \subset \V$, we set
	\[
	\Theta_\Rep = \coprod_{r\in \Rep} \Theta_r
	\]
\end{definition}

\begin{definition}
	Let $S$ be a conformal ribbon graph, and let $\Rep \subset \V$ be a
	set of vertices. The real oriented blow-up
	$\Blro_\Rep(S)$
	of $S$ at $\Rep$ is
	the topological space given as the disjoint union of the sets
	$\Theta_\Rep$ and $\V \setminus \Rep$, with topology
	gluing the end points of any edge to the corresponding
	angle in $\Theta_\Rep$, if the endpoint is in $\Rep$.
\end{definition}

\begin{definition}
	Let $S$ be a conformal ribbon graph. A \emph{gyration} $\delta$ of angle
	$\lambda \in \R$ around $\Rep$ in $S$ is a continuous path
	in $\Blro_\Rep(S)$ satisfying the following conditions:
	\begin{itemize}
		\item
		The starting and finishing points of $\delta$ are in $\Theta_\Rep$.
		\item
		The combined length of all the segments of $\delta$ which are
		contained in $\Theta_\Rep$ add up to $\lambda$, and the path
		is oriented according to he orientation of $\Theta_\Rep$ along
		these segments.
		\item
		Any segment of $\delta$ outside $\Theta_\Rep$ is a safe walk.
	\end{itemize}
	The \emph{induced path} in $S$ is the projection of this path
	to $S$, seen as a path in a graph.
\end{definition}

\begin{block}
Given a gyration $\delta$,
we say that time passes with speed 1 along segments of $\Theta_\Rep$, 
whereas
safe paths in its complement are taken instantaneously. Thus, time
passes from $0$ to $\lambda$ along a gyration of angle $\lambda$.
Assume that the gyration $\delta$ passes through the angle of an edge in $S$
adjacent to $r \in \Rep$ at time $t$, such that just before $t$,
the path is on $\Theta_\Rep$.
There are then two possibilities. After time $t$, the path proceeds
along a safe path in $\Blro_\Rep(S)$, until it comes back to $\Theta_\Rep$,
in which case we say that the gyration \emph{takes a turn} along
the corresponding edge,
or the path proceeds along $\Theta_\Rep$, in which case we say that
the path \emph{ignores the turn}.
Thus, a gyration can be constructed by specifying
a starting or a finishing point in $\Blro_\Rep(S)$,
along with instructions on which turns to take, and which ones
to ignore, along the way.
\end{block}

\begin{block} \label{block:gyrograph}
	Let $S$ be a conformal ribbon graph, with $\Rep \subset \V$.
	Let $P$ be the graph obtained from $\Blro_\R(S)$ by deleting edges
	contained in $\Theta_\Rep$.
	Consider data of the following form:
	\begin{enumerate}
		\item \label{it:varsigma_weights}
		Nonnegative real weights $\varsigma_C, \varsigma^\vee_C \in \R_{\geq 0}$
		associated with
		every componet $C$ of $P$.
		For a vertex or edge $v,e$ in $C$,
		we set $\varsigma_v = \varsigma_e = \varsigma_C$.
		and $\varsigma^\vee_v = \varsigma^\vee_e = \varsigma^\vee_C$.
		\item \label{it:map_k}
		An $S^1$-invariant bijective map $k_\Theta:\Theta_\Rep \to \Theta_\Rep$,
		inducing a permutation $k_\Rep:\Rep \to \Rep$.
		\item
		An automorphism $k_P$ of the ribbon graph $P$.
	\end{enumerate}
	For any edge $e$ in $S$ adjacent to $r\in \Rep$ and $v \in \V\setminus 
	\Rep$,
	we define gyrations as follows.
\begin{itemize}

\item
The \emph{absolute} gyration $\delta_e$
has length $\varsigma_e$, starts at
$k_P(e)$ in $\Theta_r$, and takes a turn along an edge $f$ at time
$t$ if and only if $t > \varsigma_e - \varsigma_f$.

\item
The \emph{relative} gyration $\delta^\vee_e$ has length $\varsigma$, starts at
$e$, and takes a turn along an edge $f$ at time $t$ if and only if
$t > \varsigma^\vee_f - \varsigma^\vee_e$.

\end{itemize}
\end{block}

\begin{definition} \label{def:gyrograph}
	A conformal ribbon graph $S$ and $\Rep \subset \V$ a set of vertices,
	along with weights $\varsigma_i$ on 
	its vertices, constant along connected components of $S \setminus \Rep$,
	is a \emph{gyrograph} if, for every edge $e$ in $S$,
	adjacent to $r\in \Rep$ and $v \in \V\setminus \Rep$,
	the following hold for any edge $e$ in $S$, adjacent to $r \in \Rep$
	\begin{itemize}
		
		\item
		{\bf The absolute gyrograph property:}
		the absolute gyration $\delta_e$ ends at $k_\Theta(e)$.
		
		\item
		{\bf The relative gyrograph property:}
		Let $s$ be an interior point of some edge in $S$,
		and consider a safe walk starting at $s$, stopping as soon as we
		reach $\Rep$, having final edge $f$.
		Then, the relative gyration $\delta_f^\vee$ ends at some edge of $S$,
		and a safe walk which starts following this edge, finds
		$k_P(s)$ before returning to $\Rep$.
	\end{itemize}
\end{definition}

\begin{definition} \label{block:gyrograph_abs_rel}
	Given a gyrograph $S$, with the notation introduced in 
	\cref{block:gyrograph},
	we can define a continuous map
\begin{equation}
k_S:S\to S
\end{equation}
up to homotopy, as follows.
	For any $r \in \Rep$, we set $k_S(r) = k_\Rep(r)$,
	and if $v \in \V\setminus \Rep$, then $k_S(v) = k_P(v)$.
	For any $e$ in $S$, the path $k_S(e)$ in $S$ is the concatenation
	of $k_P(e)$ and the path induced by the gyration $\delta_e$. The absolute 
	condition
	in \Cref{def:gyrograph} guarantees that this way, $k_S$ is a
	continuous map.
We call this element of $[S,S]$ the \emph{monodromy spoor} of $S$.
	
	Now, let $e$ be any edge in $S$, with some orientation. Denote by $e^\vee$
	an oriented dual segment to the edge $e$ in a thickening $F$ of $S$.
	That is, the segment $e^\vee$ intersects $S$ in a single point $s\in e$,
	an interior point of $e$, this intersection is transverse, and the
	intersection number of $e$ and $e^\vee$ at $s$ is $+1$.
	Let $w_+$ be the concatenation of the safe walk from $s$ along the
	chosen orientation, the relative gyration defined by the last
	edge of this walk, and then a safe walk which finds $k_P(s)$,
	and define $w_-$ similarly, starting the safe walk against the chosen
	orientation.
	We then have an oriented cycle $w_+ - w_-$, which we denote by $v_e$,
and call the \emph{variation spoor} of $e$, as in \cite[Section 6]{AC_lag}.
\end{definition}

\subsection{The invariant spine as a gyrograph}
\label{ss:invariant_gyrograph}

\begin{block}
	In this subsection, we put a gyrograph structure on the invariant spine
	$S = \Sinv_\theta$, for a generic angle $\theta$. Recall 
	\Cref{def:hironaka_number} for the definition of the Hironaka number.
\end{block}

\begin{definition} \label{def:gyrodata}
	Let $\theta$ be a generic angle for the $f$ as in 
	\Cref{def:non_generic_angle}.
	With $S = \Sinv_\theta$ the invariant spine of $f$, denote by
	$\V = \V_\theta$ the set of singularities of $\xiro_\theta$,
	and $\E$ the set of trajectories, forming a graph.
	Let $\Rep = \Rep_\theta \subset \V_\theta$ be the set of vertices
	corresponding to the repellers
$q_{0,-1,\theta}[1],\ldots,q_{0,-1,\theta}[m_0]$.
	As this graph is embedded in $\Ainv_\theta$, which
	has a conformal structure near the vertices, it
	inherits a conformal structure.
	In particular, the set $\Theta_\Rep$ is naturally seen as a subset
	of the coordinate space $\{(\tilde \alpha_0, \tilde \beta_0)\}$, with
	$\R/2\pi\Z$ acting on the $\tilde \beta_0$ coordinate,
	\[
	\Theta_\Rep
	=
	\set{(\tilde\alpha_0, \tilde\beta_0) \in \left(\R/2\pi\Z\right)^2}
	{m_0 \alpha_0 = \theta}.
	\]
	For any vertex $v = q_{i,a,\theta}[b] \in \V_\theta \setminus \Rep_\theta$, 
	set
	\[
	\varsigma_v = h_0^{-1} - h_i^{-1},\qquad
	\varsigma^\vee_v = h_i^{-1}.
	\]
	We define a map $k_\Theta:\Theta_\Rep \to \Theta_\Rep$ by
	\[
	\left(\tilde \alpha_0,\tilde \beta_0\right)
	\mapsto
	\left(\tilde \alpha_0+\frac{1}{m_0},\tilde \beta_0\right).
	\]
	Using the notation in \Cref{s:mon_var},
	we define a graph automorphism $k_P$ by
	\[
	q_{i,0,\theta}[c]
	\mapsto
	q_{i,a,\theta}[c+1],\qquad 
	c \in \Z / m_{ik}\Z
	\]
	and
	\[
	q_{i,a,\theta}[b]
	\mapsto
	q_{i,a,\theta}[b+1],\qquad 
	b \in \Z / m_i\Z
	\]
	for $a>0$. For prongs, and any $a$, we map
	\[
	P_{i,0,\theta}[b]
	\mapsto
	P_{i,0,\theta}[b+1],\qquad
	P^\pm_{i,a,\theta}[b]
	\mapsto
	P^\pm_{i,a,\theta}[b+1],\qquad
	b \in \Z / m_i\Z.
	\]
\end{definition}
\begin{thm}\label{thm:gyro}
The invariant spine $\Sinv_\theta$, with the data in \Cref{def:gyrodata}
is a gyrograph. Furthermore, the monodromy spoor of $\Sinv_\theta$
coincides with the homotopy class of the geometric monodromy, and
if $e\in \E_\theta$ is an edge, and $e^\vee$ is a dual edge to it, then
the variation spoor $v_e$ of $e$ coincides with the variation map
of $e^\vee$.
\end{thm}

\begin{proof}
	Let $i,k \in \V$ be vertices in $\Gamma$ joined by an edge $i\to k$,
	and assume that there is an edge $j\to i$. If $i = 0$, then we use the
	virtual vertex $j=-1$.
	The map $\Delta_{ik}:\Dro_{ik} \dashrightarrow \Dro_{ji}$
has constant Jacobian
	matrix with respect to the coordinates
	$(\tilde\alpha_k, \tilde\beta_k)$ and $(\tilde\alpha_i, \tilde\beta_i)$.
	Denote this matrix by $K_k$.
	Using \Cref{lem:formula_Deltaiell,lem:formula_alpha_q,},
	and the changes of variables
	\[
	(\tilde\alpha_i, \tilde\beta_i)
	=
	(\alpha_i + b_i\beta_i, -\beta_i),\qquad
	(\alpha_i, \beta_{ik})
	=
	(\tilde\beta_k, \tilde\alpha_k),
	\]
	we find
	\[
	K_k
	=
	\left(
	\begin{matrix}
		1 & b_i \\
		0 & -1
	\end{matrix}
	\right)
	\left(
	\begin{matrix}
		1 & \frac{m_k}{m_i} + 
		\left(\frac{m_j}{m_i}-b_i\right)\frac{n_{ik}}{n_{ij}} \\
		0 & \frac{n_{ik}}{n_{ij}}
	\end{matrix}
	\right)
	\left(
	\begin{matrix}
		0 & 1 \\
		1 & 0
	\end{matrix}
	\right)
	=
	\left(
	\begin{matrix}
		\frac{m_k}{m_i} + \frac{m_j}{m_i}\frac{n_{ik}}{n_{ij}} & 1 \\
		-\frac{n_{ik}}{n_{ij}} & 0
	\end{matrix}
	\right).
	\]
	Let $P$ be a trajectory in $\Ainv$, and consider its image under $G$.
	If $P$ passes through $\Ainv_i$, then this part of the trajectory
	is mapped to another trajectory passing through $\Ainv_i$. If
	$P$ passes through $\Ainv_{ik}$, then its image in $\Ainv$
	makes a turn by the vector $-w_k / (m_im_k)$, where we set
	\[
	w_k
	=
	2\pi
	\left(
	\begin{matrix}
		m_i \\ -m_k
	\end{matrix}
	\right)
	\]
	in the coordinates $(\tilde\alpha_k, \tilde\beta_k)$ in $\Ainv_{ik}$,
	and then continues as a trajectory in $\Ainv_i$.
	Now, a direct computation gives
	\[
	K_k w_k = \frac{n_{ik}}{n_{ji}} w_i.
	\]
	Consider the unique geodesic
$0 \to i_1 \to \ldots \to i_{l-1} \to i_l$  with $i_0 = 0$, $i_{l-1}=i$
and $i_l = k$
	in the graph $\Gamma$. Then, the above formula gives a telescopic product
	\[
	K_{i_1}\cdots K_{i_{l-1}}K_{i_l} \frac{w_k}{m_im_k}
	= \frac{n_{ik}}{m_im_k n_{-1,0}} w_0
	= \frac{2\pi n_{ik}}{m_im_k n_{-1,0}}
	\left(
	\begin{matrix}
		0 \\ -m_0
	\end{matrix}
	\right)
	= \frac{n_{ik}}{m_im_k}
	\left(
	\begin{matrix}
		0 \\ -2\pi
	\end{matrix}
	\right).
	\]
Assume now that the turn that the image of $P$ makes in $\Ainv_{ik}$ does
not cross any unstable manifold. Then, by applying
the backwards flow of $\xiinv$ to this turn, we can see this as
a gyration by $-2\pi n_{ik}/(m_im_k)$ radians around a repeller. If the image
	does cross an unstable manifold, then, applying backwards flow by $\xiinv$
	to the turn in $\Ainv_{ik}$ ends up in a the same total gyration around
	repellers, with fast jumps along safe paths in between.
	
	If $P$ is a stable prong with an end point in a singularity of $\xiinv$ in
	$\Dro_{k,\theta}$, say $P = P^\pm_{k,a,\theta}[b]$ for some $a,b$, apply
	backwards flow to the whole of $G(P)$
(see \Cref{def:angled_ray} for notation). We end up with a total gyration
	around repellers by an angle
	\[
	-2\pi \sum_{s=1}^l \frac{n_{i_{s-1}i_s}}{m_{i_{s-1}} m_{i_s}}
	=
	-2\pi \sum_{s=1}^l
	\frac{c_{0,i_{s-1}}}{m_{i_{s-1}}}
	-
	\frac{c_{0,i_s}}{m_{i_s}}
	=
	h_k^{-1}
	-
	h_0^{-1}
	=
	-\varsigma_k.
	\]
	As a result, if $e$ is the edge in $\Sinv$ corresponding
	to the prong $P=P^\pm_{k,a,\theta}[b]$,
	then the image $G(P)$ is isotopic (fixing the end points)
	to the composition of the prong $P^\pm_{k,a,\theta}[b+1]$
	and a gyraton of length $\varsigma_k$.
	This is precisely the gyration $\delta_e$.
	Indeed, assume that we encounter another
	edge $e'$ in the spine $\Sinv_\theta$ along the gyration at time $t$,
	and that this edge $e'$ corresponds to some prong $P' = 
	P_{\ell,c,\theta}[d]$.
	There exists a unique $s$ such that
	\[
	\varsigma_k - \varsigma_{i_s} < t < \varsigma_k - \varsigma_{i_{s-1}}.
	\]
	If $\varsigma_{\ell} \geq \varsigma_{i_s}$, i.e. if
	$\varsigma_\ell > \varsigma_k - t$,
	then $G(P)$ already crosses
	the prong $P'$, and the gyration ignores the turn. Otherwise,
	the gyration takes the turn.
	
	As a result, we conclude that
	the data in \Cref{def:gyrodata} satisfies the absolute gyrograph property.
	Indeed, this follows from the fact that the composition of the prong
	$P_{k,a,\theta}[b+1]$ and the gyration obtained from
	$G(P_{k_a,\theta}[b])$ is well defined.
	
	Now, the relative gyrograph property is proved similarly.
	We start by considering a singularity $q_{i,a,\theta}[b]$.
	Let
$j_0 \to j_1 \to \ldots \to j_r$ with $j_0 = k$,
be the path induced by the prong
	$R^+_{i,a,\theta}[b]$, ending in some arrowhead vertex $a = j_r$.
	Using a similar argument as above, using backwards flow of $\xiinv$,
	shows that $G(R^+_{i,a,\theta})$ is isotopic to
	the concatenation of $R^\pm_{i,a,\theta}[b]$,
	then a safe walk from $q_{i,a,\theta}[b]$ until we reach $\Rep_\theta$,
	then a gyration,
	and then a safe walk which reaches $q_{i,a,\theta}[b+1]$, and does not pass
	through $\Rep_\theta$.
	This gyration has angle
	\[
	\sum_{s=0}^{r-1} h_{j_s}^{-1} - h_{j_{s+1}}^{-1}
	= h_i^{-1}
	= \varsigma^\vee_v
	\]
	since $h^{-1}_{j_r} = 0$, as $j_r$ is an arrowhead vertex.
	That this composition of paths is allowed proves the relative
	gyrograph property of this data.
	
It follows from this construction that the image by $G_\theta$ of
any edge is homotopic to the path which defines the monodromy spoor,
relative to its end-points.
As a result, the monodromy spoor is the homotopy type of the geometric
monodromy.

Furthermore, applying the variation map of $G_\theta$ to a dual edge $e^\vee$
gives precisely the variation spoor $v_e$.
\end{proof}

\subsection{An example with two Puiseux-pairs}
\label{ss:two_puiseux}

\begin{block}
Let us use the gyrograph structure on $\Sinv_\theta$ to calculate the
algebraic monodromy and variation map of the simplest branch
having two Puiseux pairs, given by the equation
\[
  f(x,y) = (y^2-x^3)^2 + x^5y.
\]
The minimal embedded resolution graph of $f$,
along with data, is as follows.

\begin{minipage}{.4\textwidth}
\begin{center}
\begin{picture}(0,0)%
\includegraphics{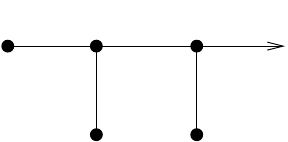}%
\end{picture}%
\setlength{\unitlength}{4144sp}%
\begin{picture}(2187,1079)(1291,-1239)
\put(1306,-331){\makebox(0,0)[lb]{\smash{\fontsize{12}{14.4}\normalfont $0$}}}
\put(1981,-331){\makebox(0,0)[lb]{\smash{\fontsize{12}{14.4}\normalfont $1$}}}
\put(2116,-1096){\makebox(0,0)[lb]{\smash{\fontsize{12}{14.4}\normalfont $2$}}}
\put(2881,-1096){\makebox(0,0)[lb]{\smash{\fontsize{12}{14.4}\normalfont $4$}}}
\put(2746,-331){\makebox(0,0)[lb]{\smash{\fontsize{12}{14.4}\normalfont $3$}}}
\end{picture}%

\label{fig:mnpolyg}
\end{center}
\end{minipage}
\begin{minipage}{.6\textwidth}
\begin{center}
\vspace{.4cm}
\begin{tabular}{c|c|c|c|c|c|c}
$i$ & $-b_i$ & $c_{0,i}$ & $m_i$ & $h_i$ &
$\varsigma_i$ & $\varsigma_i^\vee$ \\ \hline
0   & $-3$ & 1 & 4   &   4   &    0    &  1/4   \\ \hline
1   & $-2$ & 2 & 12  &   6   &  1/12   &  1/6   \\ \hline
2   & $-3$ & 1 & 6   &   6   &  1/12   &  1/6   \\ \hline
3   & $-1$ & 4 & 26  & 13/2  &  5/52   &  2/13  \\ \hline
4   & $-2$ & 2 & 13  & 13/2  &  5/52   &  2/13
\end{tabular}
\vspace{.4cm}
\end{center}
\end{minipage}
The gyrograph has four repellers
$\Rep = \{q_{0,-1,\theta}[1], q_{0,-1,\theta}[2],q_{0,-1,\theta}[3],
q_{0,-1,\theta}[4]\}$, as well as two groups of vertices
\[
  q_{1,0,\theta}[1],
  \ldots,
  q_{1,0,\theta}[6],\qquad
  q_{3,0,\theta}[1],
  \ldots,
  q_{3,0,\theta}[13],
\]
each permuted cyclically by the monodromy. 
In \cref{fig:gyro_ex} we see these vertices, along with
green and blue stable prongs, but with the repellers removed,
whereas in \cref{fig:gyrograph} we see a neighborhood
of $\Theta_\Rep$ in the blown up graph $\Blro_\R(\Sinv_\theta)$,
as well as the $12$ blue prongs coming from the first set of
vertices, and the $26$ green prongs coming from the second set of vertices.
The action on $\Theta_\Rep$ is induced by the explicit map
$G_0$ acting on a neighborhood of the repellers, in $\tilde U_0$.
In \cref{fig:gyrograph},
it sends each boundary component cyclically to the next one to the right.
By choosing coordinates, we can assume that any point is the end point
of a blue prong. We label that prong by $1$, and the rest of the prongs are
labelled cyclically, using the monodromy near the vertices $q_{1,0,\theta}[b]$.
The location of the rest of the blue prongs
is then determined, since the end-point of the next blue prong is
found by moving one step to the right, and going backwards by a gyration
of length $\varsigma_2 = 1/12$. This gyration takes no turns.
\begin{figure}[ht]
\begin{center}
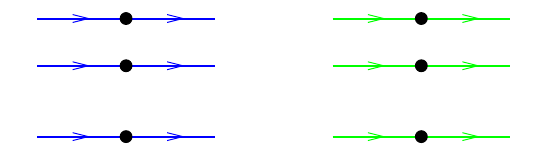
\caption{The spine $\Sinv_\theta$, with repellers removed, and stable
prongs labelled cyclically. The black dots
are the points $q_{i,0,\theta}[b]$.}
\label{fig:gyro_ex}
\end{center}
\end{figure}
Now, choose any point on $\Theta_\Rep$ which is not the endpoint of a
blue prong. Then there exists some $\theta$ such that this is the end point
of a green prong. We label this prong by $1$, and the rest cyclically as above.
The rest of the green prongs, and their cyclic order
is then determined, again by jumping to the right, from any green prong,
and following a gyration of length
\[
  \varsigma_4 = \frac{1}{4} - \frac{2}{13} = \frac{5}{52}
\]
backwards, with the rule that if a blue prong
is encountered at time
\[
t<\frac{5}{52} - \frac{1}{12} = \frac{1}{78}
\]
then we make a turn, otherwise the turn is ignored. In figure
\cref{fig:gyrograph}, we have chosen a green prong at a point near
the first blue prong, so that the gyration starting at the first
green prong of length $5/52$ takes a turn along blue prong number $1$,
comes back at the blue prong number $7$, and meets up with the image
of the green prong number $26$ under $k_\Theta$ in $\Theta_{r_3}$.

Clearly, we have
\[
  B_{0,\theta}
  =
  \left(
  \begin{matrix}
  0 & 0 & 0 & 1 \\
  1 & 0 & 0 & 0 \\
  0 & 1 & 0 & 0 \\
  0 & 0 & 1 & 0 \\
  \end{matrix}
  \right)
\]
Choosing a suitable orientation in \cref{fig:gyro_ex}, and using
\cref{fig:gyrograph}, we find that the differential $d_{1,\theta}$ is
given by the matrix
\[
\left(
\begin{array}{cccccc|ccccccccccccc}
-1& 0 & 1 & 0 &-1 & 0 &  -1 & 0 & 0 & 1 &-1 & 0 & 0 & 1 &-1 & 0 & 0 & 1 &-1 \\
0 &-1 & 0 & 1 & 0 &-1 &   1 &-1 & 0 & 0 & 1 &-1 & 0 & 0 & 1 &-1 & 0 & 0 & 1 \\
1 & 0 &-1 & 0 & 1 & 0 &   0 & 1 &-1 & 0 & 0 & 1 &-1 & 0 & 0 & 1 &-1 & 0 & 0 \\
0 & 1 & 0 &-1 & 0 & 1 &   0 & 0 & 1 &-1 & 0 & 0 & 1 &-1 & 0 & 0 & 1 &-1 & 0 \\
\end{array}
\right)
\]
From what we have seen, the map $k_S$, with $S = \Sinv_\theta$,
permutes the blue prongs cyclically. The green prongs are similarly
permuted cyclically, with the exception that the prong labelled by
$26$ is mapped to prong $1$, and then blue prong $1$ and $7$.
Taking the segment in \cref{fig:gyro_ex} as a basis for
$C_{1,\theta}$, we find that $B_{1,\theta}$ is a block matrix.
Along the diagonal, we have matrices of size $6\times 6$ and $13\times 13$
which are similar to permutation matrices, but have a $-1$ in the corner,
since applying the monodromy 6 times to a blue segment returns it to itself
with the opposite orientation, and similarly for green segments.
In the upper off-diagonal block we have one nonzero entry, the rest is zero.
Indeed, the only turn taken by an absolute gyrations is that of the
green prong $26$, which turns along prong $7$, and then follows prong
$1$, to come back to green prong $1$. Taking into account orientation,
this gives a $-1$ in the upper right hand corner.
\[
  B_{1,\theta} =
\left(
\begin{array}{cccccc|ccccccccccccc}
0 & 0 & 0 & 0 & 0 & -1 &   0 & 0 & 0 & 0 & 0 & 0 & 0 & 0 & 0 & 0 & 0 & 0 & -1 \\
1 & 0 & 0 & 0 & 0 & 0 &   0 & 0 & 0 & 0 & 0 & 0 & 0 & 0 & 0 & 0 & 0 & 0 & 0 \\
0 & 1 & 0 & 0 & 0 & 0 &   0 & 0 & 0 & 0 & 0 & 0 & 0 & 0 & 0 & 0 & 0 & 0 & 0 \\
0 & 0 & 1 & 0 & 0 & 0 &   0 & 0 & 0 & 0 & 0 & 0 & 0 & 0 & 0 & 0 & 0 & 0 & 0 \\
0 & 0 & 0 & 1 & 0 & 0 &   0 & 0 & 0 & 0 & 0 & 0 & 0 & 0 & 0 & 0 & 0 & 0 & 0 \\
0 & 0 & 0 & 0 & 1 & 0 &   0 & 0 & 0 & 0 & 0 & 0 & 0 & 0 & 0 & 0 & 0 & 0 & 0 \\
\hline                    
0 & 0 & 0 & 0 & 0 & 0 &   0 & 0 & 0 & 0 & 0 & 0 & 0 & 0 & 0 & 0 & 0 & 0 & -1 \\
0 & 0 & 0 & 0 & 0 & 0 &   1 & 0 & 0 & 0 & 0 & 0 & 0 & 0 & 0 & 0 & 0 & 0 & 0 \\
0 & 0 & 0 & 0 & 0 & 0 &   0 & 1 & 0 & 0 & 0 & 0 & 0 & 0 & 0 & 0 & 0 & 0 & 0 \\
0 & 0 & 0 & 0 & 0 & 0 &   0 & 0 & 1 & 0 & 0 & 0 & 0 & 0 & 0 & 0 & 0 & 0 & 0 \\
0 & 0 & 0 & 0 & 0 & 0 &   0 & 0 & 0 & 1 & 0 & 0 & 0 & 0 & 0 & 0 & 0 & 0 & 0 \\
0 & 0 & 0 & 0 & 0 & 0 &   0 & 0 & 0 & 0 & 1 & 0 & 0 & 0 & 0 & 0 & 0 & 0 & 0 \\
0 & 0 & 0 & 0 & 0 & 0 &   0 & 0 & 0 & 0 & 0 & 1 & 0 & 0 & 0 & 0 & 0 & 0 & 0 \\
0 & 0 & 0 & 0 & 0 & 0 &   0 & 0 & 0 & 0 & 0 & 0 & 1 & 0 & 0 & 0 & 0 & 0 & 0 \\
0 & 0 & 0 & 0 & 0 & 0 &   0 & 0 & 0 & 0 & 0 & 0 & 0 & 1 & 0 & 0 & 0 & 0 & 0 \\
0 & 0 & 0 & 0 & 0 & 0 &   0 & 0 & 0 & 0 & 0 & 0 & 0 & 0 & 1 & 0 & 0 & 0 & 0 \\
0 & 0 & 0 & 0 & 0 & 0 &   0 & 0 & 0 & 0 & 0 & 0 & 0 & 0 & 0 & 1 & 0 & 0 & 0 \\
0 & 0 & 0 & 0 & 0 & 0 &   0 & 0 & 0 & 0 & 0 & 0 & 0 & 0 & 0 & 0 & 1 & 0 & 0 \\
0 & 0 & 0 & 0 & 0 & 0 &   0 & 0 & 0 & 0 & 0 & 0 & 0 & 0 & 0 & 0 & 0 & 1 & 0 \\
\end{array}
\right)
\]
Now, let us calculate the variation map $V_{1,\theta}$ in this basis,
and its dual.
If $e$ is a blue prong, then the relative gyration
$\delta^\vee_e$ starts at the edge $e$ and takes a turn along any green
edge, if it is encountered at time
\[
  t < \varsigma^\vee_1 - \varsigma^\vee_3
  =  \frac{1}{6} - \frac{2}{13}
  =  \frac{1}{78}
\]
This is what happens if $e$ is the first blue prong. The relative gyration
finds the green prong $21$, makes a fast safe walk to green prong
$8$, and then stops at blue prong $8$. From there, there is a safe
walk to $k_P(e)$, i.e. blue prong $2$.
To construct $v_e$, we must also consider the \emph{other} gyration,
following a safe walk from blue prong $1$ to blue prong $7$, then
a gyration of length $1/6$ which makes a turn at green prong $15$, and
then finishes gyrating from green prong $2$ to blue prong $2$.
This shows that the cycle $v_e$ induces the homological cycle given
by the first column in the matrix $V_{1,\theta}$ below. This includes
the first entries of the column equal to $1$, since $v_e$ crosses
the blue prong $1$, and its image by $k_P$, i.e. blue prong $2$.
It also makes two green jumps, which, by the orientation indicated,

If $e$ is a green prong, then
the relative gyration $\delta^\vee_e$ always makes a turn if it
encounters a blue edge. In this example, we see that this gyration actually
never encounters another green edge.
If $e$ is the first green edge,
then $v_e$ is the cycle which starts at $e$, makes a turn at the blue edge
labelled $1$ and makes a quick safe walk to the blue edge labelled $7$,
and continues gyrating from there until it meets the green edge labelled $15$.
From there, it takes a safe walk to the green edge $2$ which equals $k_P(e)$.
Now, concatenate this with the opposite of the path going the other way, which
starts at the
green edge $1$, follows a safe walk to green edge $14$, then gyrates
of length $\varsigma_2$ directly to green edge $2$, closing the cycle.
Since $v_e$ crosses the edges $e$ and $k_P(e)$, we find the value $1$
at position $7,7$ and $8,7$ in the matrix for $V_{1,\theta}$.
Since the gyration did a fast jump along the first blue edge, there is one
more $1$ in the seventh column of this matrix at position $1,7$.
Now, one verifies that the jumps made by the cylcles $v_e$, for $e$
green prongs labelled $1,\ldots,13$ give similar cycles described by the
columns in $V_{1,\theta}$ seen below.

\[
  V_{1,\theta} = 
\left(
\begin{array}{cccccc|ccccccccccccc}
1 & 0 & 0 & 0 & 0 &-1   & 1 & 0 & 0 & 0 & 0 & 0 &-1 & 0 & 0 & 0 & 0 & 0 & 0 \\
1 & 1 & 0 & 0 & 0 & 0   & 0 & 1 & 0 & 0 & 0 & 0 & 0 &-1 & 0 & 0 & 0 & 0 & 0 \\
0 & 1 & 1 & 0 & 0 & 0   & 0 & 0 & 1 & 0 & 0 & 0 & 0 & 0 &-1 & 0 & 0 & 0 & 0 \\
0 & 0 & 1 & 1 & 0 & 0   & 0 & 0 & 0 & 1 & 0 & 0 & 0 & 0 & 0 &-1 & 0 & 0 & 0 \\
0 & 0 & 0 & 1 & 1 & 0   & 0 & 0 & 0 & 0 & 1 & 0 & 0 & 0 & 0 & 0 &-1 & 0 & 0 \\
0 & 0 & 0 & 0 & 1 & 1   & 0 & 0 & 0 & 0 & 0 & 1 & 0 & 0 & 0 & 0 & 0 &-1 & 0 \\ \hline
0 & 0 & 0 & 0 & 0 & 0   & 1 & 0 & 0 & 0 & 0 & 0 & 0 & 0 & 0 & 0 & 0 & 0 &-1 \\
1 & 0 & 0 & 0 & 0 & 0   & 1 & 1 & 0 & 0 & 0 & 0 & 0 & 0 & 0 & 0 & 0 & 0 & 0 \\
0 & 1 & 0 & 0 & 0 & 0   & 0 & 1 & 1 & 0 & 0 & 0 & 0 & 0 & 0 & 0 & 0 & 0 & 0 \\
0 & 0 & 1 & 0 & 0 & 0   & 0 & 0 & 1 & 1 & 0 & 0 & 0 & 0 & 0 & 0 & 0 & 0 & 0 \\
0 & 0 & 0 & 1 & 0 & 0   & 0 & 0 & 0 & 1 & 1 & 0 & 0 & 0 & 0 & 0 & 0 & 0 & 0 \\
0 & 0 & 0 & 0 & 1 & 0   & 0 & 0 & 0 & 0 & 1 & 1 & 0 & 0 & 0 & 0 & 0 & 0 & 0 \\
0 & 0 & 0 & 0 & 0 & 1   & 0 & 0 & 0 & 0 & 0 & 1 & 1 & 0 & 0 & 0 & 0 & 0 & 0 \\
-1& 0 & 0 & 0 & 0 & 0   & 0 & 0 & 0 & 0 & 0 & 0 & 1 & 1 & 0 & 0 & 0 & 0 & 0 \\
0 &-1 & 0 & 0 & 0 & 0   & 0 & 0 & 0 & 0 & 0 & 0 & 0 & 1 & 1 & 0 & 0 & 0 & 0 \\
0 & 0 &-1 & 0 & 0 & 0   & 0 & 0 & 0 & 0 & 0 & 0 & 0 & 0 & 1 & 1 & 0 & 0 & 0 \\
0 & 0 & 0 &-1 & 0 & 0   & 0 & 0 & 0 & 0 & 0 & 0 & 0 & 0 & 0 & 1 & 1 & 0 & 0 \\
0 & 0 & 0 & 0 &-1 & 0   & 0 & 0 & 0 & 0 & 0 & 0 & 0 & 0 & 0 & 0 & 1 & 1 & 0 \\
0 & 0 & 0 & 0 & 0 &-1   & 0 & 0 & 0 & 0 & 0 & 0 & 0 & 0 & 0 & 0 & 0 & 1 & 1 \\
\end{array}
\right)
\]
We now choose a basis for the homology of $\Sinv_\theta$, as follows.
Since the blue edges $1$ and $2$, and the green edge $1$ form together
a spanning tree, we can complete any other edge in the graph as a cycle in
$\Sinv_\theta$. This way, we get a basis, and one verifies that in this
basis, the monodromy is given by the matrix
\[
\left(
\begin{array}{cccccccccccccccc}
 0&1&0&-1&0&-1&1& 0&0&-1&1& 0&0&-1&1& 0\\
 1&0&0& 0&0& 0&0& 0&0& 0&0& 0&0& 0&0& 0\\
 0&1&0& 0&0& 0&0& 0&0& 0&0& 0&0& 0&0& 0\\
 0&0&1& 0&0& 0&0& 0&0& 0&0& 0&0& 0&0& 0\\
 0&0&0& 0&1&-1&1&-1&1&-1&1&-1&1&-1&1&-1\\
 0&0&0& 0&1& 0&0& 0&0& 0&0& 0&0& 0&0& 0\\
 0&0&0& 0&0& 1&0& 0&0& 0&0& 0&0& 0&0& 0\\
 0&0&0& 0&0& 0&1& 0&0& 0&0& 0&0& 0&0& 0\\
 0&0&0& 0&0& 0&0& 1&0& 0&0& 0&0& 0&0& 0\\
 0&0&0& 0&0& 0&0& 0&1& 0&0& 0&0& 0&0& 0\\
 0&0&0& 0&0& 0&0& 0&0& 1&0& 0&0& 0&0& 0\\
 0&0&0& 0&0& 0&0& 0&0& 0&1& 0&0& 0&0& 0\\
 0&0&0& 0&0& 0&0& 0&0& 0&0& 1&0& 0&0& 0\\
 0&0&0& 0&0& 0&0& 0&0& 0&0& 0&1& 0&0& 0\\
 0&0&0& 0&0& 0&0& 0&0& 0&0& 0&0& 1&0& 0\\
 0&0&0& 0&0& 0&0& 0&0& 0&0& 0&0& 0&1& 0\\
\end{array}
\right)
\]
Correspondingly, take dual segments to any edge in the complement of
the spanning tree, which form a basis.
The variation map is then given by the matrix
\[
\left(
\begin{array}{cccccccccccccccc}
   1 & 0 & 0 & 0 &0 &-1 & 0 & 0 & 0& 0& 0& 1& 0& 0& 0& 0\\
   1 & 1 & 0 & 0 &0 & 0 &-1 & 0 & 0& 0& 0& 0& 1& 0& 0& 0\\
   0 & 1 & 1 & 0 &0 & 0 & 0 &-1 & 0& 0& 0& 0& 0& 1& 0& 0\\
   0 & 0 & 1 & 1 &0 & 0 & 0 & 0 &-1& 0& 0& 0& 0& 0& 1& 0\\
   0 & 0 & 0 & 0 &1 & 0 & 0 & 0 & 0& 0& 0& 0& 0& 0& 0& 0\\
   0 & 0 & 0 & 0 &1 & 1 & 0 & 0 & 0& 0& 0& 0& 0& 0& 0& 0\\
  -1 & 0 & 0 & 0 &0 & 1 & 1 & 0 & 0& 0& 0& 0& 0& 0& 0& 0\\
   0 &-1 & 0 & 0 &0 & 0 & 1 & 1 & 0& 0& 0& 0& 0& 0& 0& 0\\
   0 & 0 &-1 & 0 &0 & 0 & 0 & 1 & 1& 0& 0& 0& 0& 0& 0& 0\\
   0 & 0 & 0 &-1 &0 & 0 & 0 & 0 & 1& 1& 0& 0& 0& 0& 0& 0\\
   0 & 0 & 0 & 0 &0 & 0 & 0 & 0 & 0& 1& 1& 0& 0& 0& 0& 0\\
   0 & 0 & 0 & 0 &0 & 0 & 0 & 0 & 0& 0& 1& 1& 0& 0& 0& 0\\
   1 & 0 & 0 & 0 &0 & 0 & 0 & 0 & 0& 0& 0& 1& 1& 0& 0& 0\\
   0 & 1 & 0 & 0 &0 & 0 & 0 & 0 & 0& 0& 0& 0& 1& 1& 0& 0\\
   0 & 0 & 1 & 0 &0 & 0 & 0 & 0 & 0& 0& 0& 0& 0& 1& 1& 0\\
   0 & 0 & 0 & 1 &0 & 0 & 0 & 0 & 0& 0& 0& 0& 0& 0& 1& 1\\
\end{array}
\right)
\]

\begin{figure}[p]
	\begin{center}
		\scalebox{0.5}{ 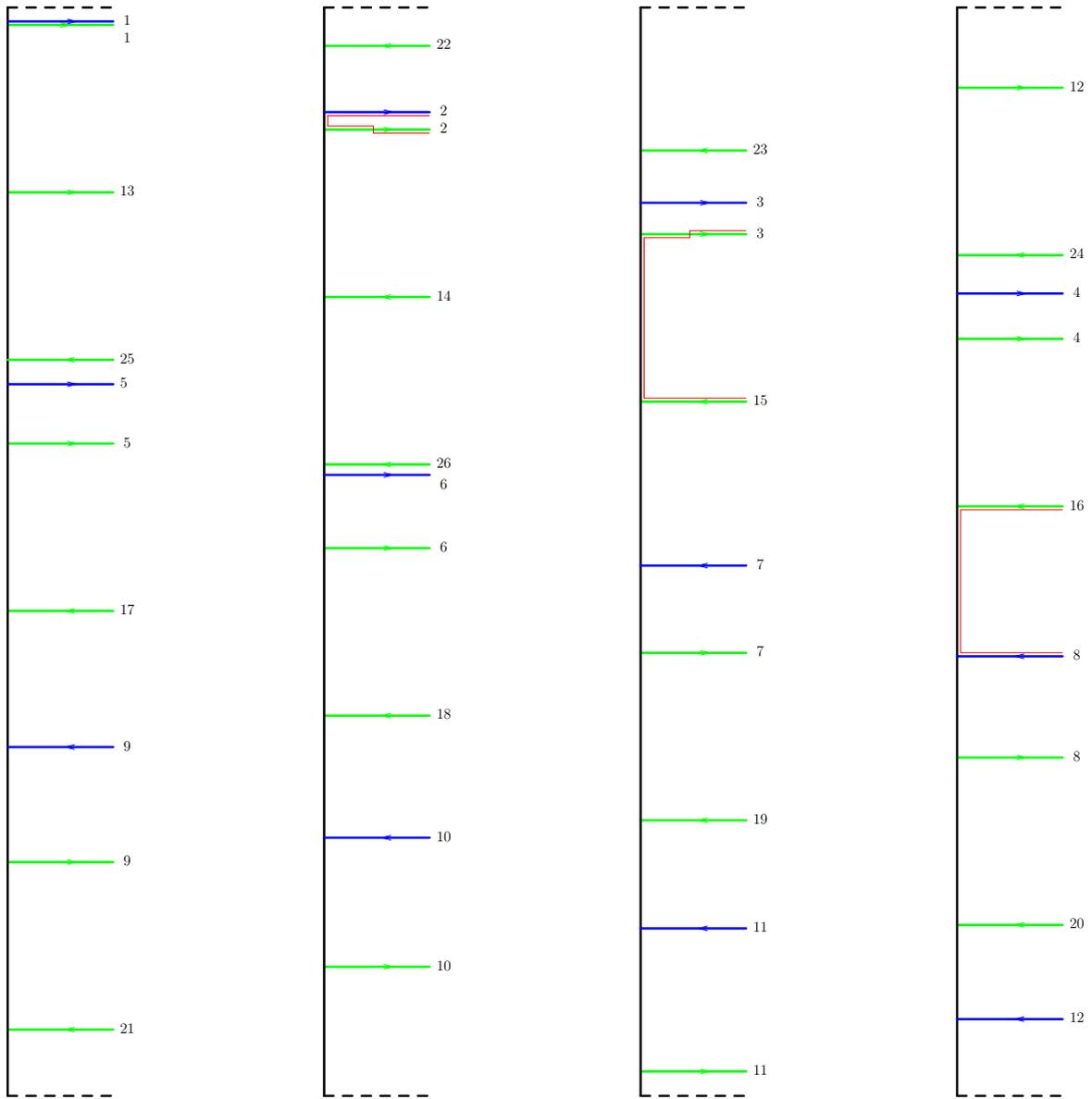}
		\caption{A blown up gyrograph for a plane curve with two Puiseux 
		pairs. The red path is the variation map applied to a relative cycle
intersecting the green segment $2$ precisely once.}
		\label{fig:gyrograph}
	\end{center}
\end{figure}
\end{block}

\section{Software implementation}\label{s:software}
We have implemented the algorithms described in this paper in a Python script 
\texttt{var\_class.py} located in a Github repository. This script relies 
on \texttt{Singular} for resolution of singularities and \texttt{NetworkX} for 
graph manipulations. 

\subsection{Usage}
To use the software, ensure you have \texttt{Singular} installed and accessible 
in your path. The Python dependencies are \texttt{networkx}, 
\texttt{matplotlib}, \texttt{numpy}, and \texttt{sympy}.

The main function is \texttt{run\_polynomial(polynomial)}, which takes a 
polynomial string as input and returns the resolution graph and other data 
structures.

\subsection{Example}
Consider the polynomial $f(x,y) = (y^2+x^3)(x^2+y^3)$.
We run the following code to compute the resolution graph and the algebraic 
invariants.
We define a parameter
{\tt T = 1/100} corresponding to the angle $\theta = 2\pi / 100$.
In the code we scale the angle $\theta$ by $1/2\pi$ for simplification.
We compute the monodromy matrix $M$, the variation matrix $Var$, and the 
differential map $DI$ on the chain complex.
Then, we compute the Smith Normal Form of $DI$ to find a basis where the 
differential is diagonal. This allows us to identify the homology of the 
complex (the kernel of $DI$).
We denote by $V$ the change of basis matrix. The matrices $M_H$ and $\Var_H$ 
represent the restrictions of the monodromy and variation operators to the 
homology.

\begin{verbatim}
	import var_class
	import sympy as sp
	from fractions import Fraction
	G, H, Up, EN, mult = var_class.run_polynomial("(y^2+x^3)*(x^2+y^3)")
	var_class.draw_resolution_graph(G, H, mult=False, order=False, euler=True)
\end{verbatim}
This produces the resolution graph shown in \Cref{fig:example_graph}.

\begin{figure}[!ht]
	\centering
	\includegraphics[width=0.5\textwidth]{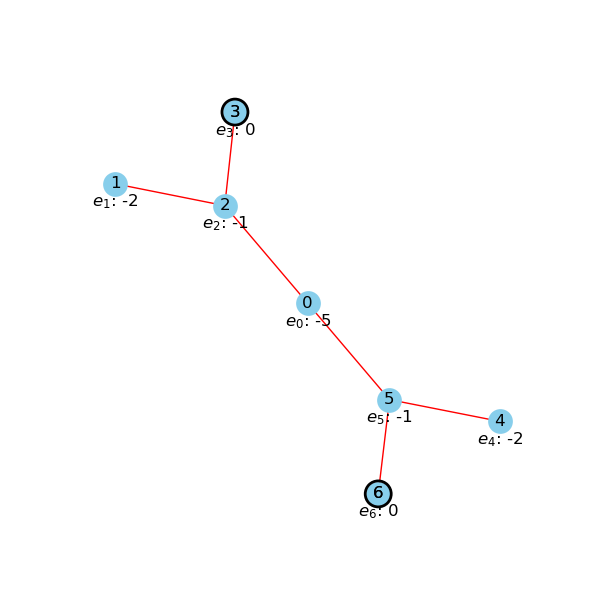}
	\caption{Resolution graph for $(y^2+x^3)(x^2+y^3)$ produced by the 
		software.}
	\label{fig:example_graph}
\end{figure}

\begin{verbatim}	
	# We define the parameters and compute the matrices
	T = Fraction(1, 100)
	M = var_class.monodromy_matrix(EN, G, Up, H, T)
	M = sp.Matrix(M)
	Var = var_class.variation_matrix(EN, G, Up, H, T)
	Var = sp.Matrix(Var)
	DI = var_class.differential_map(EN, G, Up, H, T)
	DI = sp.Matrix(DI)
	
	# We compute the Smith Normal Form to find the homology
	D, U, W = var_class.smith_normal_form_custom(DI)
	r = DI.rank()
	MH = (W**(-1)*M*W)[r:,r:]
	Var_H = ((W**-1)*Var*((V.transpose())**(-1)))[r:,r:]
\end{verbatim}

The script also computes the monodromy and variation matrices, as well as their 
homological counterparts.

\subsection*{Monodromy matrix on the complex}
The monodromy matrix $M$ acting on the chain complex.
\[ M = 
\left(\begin{array}{cccccccccccccc}0 & 0 & 0 & 1 & 0 & 0 & 0 & 0 & 0 & 0 & 0 & 
0 & 0 & 0\\1 & 0 & 0 & 0 & 0 & 0 & 0 & 0 & 0 & 0 & 0 & 0 & 0 & 0\\0 & 1 & 0 & 0 
& 0 & 0 & 0 & 0 & 1 & 0 & 0 & 0 & 0 & 1\\0 & 0 & 1 & 0 & 0 & 0 & 0 & 0 & 0 & 0 
& 0 & 0 & 0 & 0\\0 & 0 & 0 & 0 & 0 & 0 & 0 & 0 & -1 & 0 & 0 & 0 & 0 & 0\\0 & 0 
& 0 & 0 & 1 & 0 & 0 & 0 & 0 & 0 & 0 & 0 & 0 & 0\\0 & 0 & 0 & 0 & 0 & 1 & 0 & 0 
& 0 & 0 & 0 & 0 & 0 & 0\\0 & 0 & 0 & 0 & 0 & 0 & 1 & 0 & 0 & 0 & 0 & 0 & 0 & 
0\\0 & 0 & 0 & 0 & 0 & 0 & 0 & 1 & 0 & 0 & 0 & 0 & 0 & 0\\0 & 0 & 0 & 0 & 0 & 0 
& 0 & 0 & 0 & 0 & 0 & 0 & 0 & -1\\0 & 0 & 0 & 0 & 0 & 0 & 0 & 0 & 0 & 1 & 0 & 0 
& 0 & 0\\0 & 0 & 0 & 0 & 0 & 0 & 0 & 0 & 0 & 0 & 1 & 0 & 0 & 0\\0 & 0 & 0 & 0 & 
0 & 0 & 0 & 0 & 0 & 0 & 0 & 1 & 0 & 0\\0 & 0 & 0 & 0 & 0 & 0 & 0 & 0 & 0 & 0 & 
0 & 0 & 1 & 0\end{array}\right)
\]

\subsection*{Monodromy matrix after change of basis}
The monodromy matrix expressed in the basis given by the Smith Normal Form.
\[ W^{-1}MW = 
\left(\begin{array}{cccccccccccccc}0 & 1 & -1 & 0 & 0 & 0 & 0 & 0 & 0 & 0 & 0 & 
0 & 0 & 0\\-1 & 0 & 0 & 0 & 0 & 0 & 0 & 0 & 0 & 0 & 0 & 0 & 0 & 0\\0 & 0 & -1 & 
0 & 0 & 0 & 0 & 0 & 0 & 0 & 0 & 0 & 0 & 0\\0 & 0 & 0 & 0 & 1 & 0 & 0 & 0 & 0 & 
0 & 0 & 0 & 0 & 0\\1 & 0 & -1 & 1 & 0 & 0 & 1 & -1 & 1 & -1 & 1 & 0 & 0 & 0\\0 
& 0 & 1 & 0 & 0 & 1 & -1 & 1 & -1 & 1 & -1 & 1 & -1 & 1\\0 & 0 & 0 & 0 & 0 & 1 
& 0 & 0 & 0 & 0 & 0 & 0 & 0 & 0\\0 & 0 & 0 & 0 & 0 & 0 & 1 & 0 & 0 & 0 & 0 & 0 
& 0 & 0\\0 & 0 & 0 & 0 & 0 & 0 & 0 & 1 & 0 & 0 & 0 & 0 & 0 & 0\\0 & 0 & 0 & 0 & 
0 & 0 & 0 & 0 & 0 & 0 & 0 & 0 & 0 & -1\\0 & 0 & 0 & 0 & 0 & 0 & 0 & 0 & 0 & 1 & 
0 & 0 & 0 & 0\\0 & 0 & 0 & 0 & 0 & 0 & 0 & 0 & 0 & 0 & 1 & 0 & 0 & 0\\0 & 0 & 0 
& 0 & 0 & 0 & 0 & 0 & 0 & 0 & 0 & 1 & 0 & 0\\0 & 0 & 0 & 0 & 0 & 0 & 0 & 0 & 0 
& 0 & 0 & 0 & 1 & 0\end{array}\right)
\]

\subsection*{Monodromy matrix in homology}
The restriction of the monodromy matrix to the homology (kernel of the 
differential).
\[ M_H = 
\left(\begin{array}{ccccccccccc}0 & 1 & 0 & 0 & 0 & 0 & 0 & 0 & 0 & 0 & 0\\1 & 
0 & 0 & 1 & -1 & 1 & -1 & 1 & 0 & 0 & 0\\0 & 0 & 1 & -1 & 1 & -1 & 1 & -1 & 1 & 
-1 & 1\\0 & 0 & 1 & 0 & 0 & 0 & 0 & 0 & 0 & 0 & 0\\0 & 0 & 0 & 1 & 0 & 0 & 0 & 
0 & 0 & 0 & 0\\0 & 0 & 0 & 0 & 1 & 0 & 0 & 0 & 0 & 0 & 0\\0 & 0 & 0 & 0 & 0 & 0 
& 0 & 0 & 0 & 0 & -1\\0 & 0 & 0 & 0 & 0 & 0 & 1 & 0 & 0 & 0 & 0\\0 & 0 & 0 & 0 
& 0 & 0 & 0 & 1 & 0 & 0 & 0\\0 & 0 & 0 & 0 & 0 & 0 & 0 & 0 & 1 & 0 & 0\\0 & 0 & 
0 & 0 & 0 & 0 & 0 & 0 & 0 & 1 & 0\end{array}\right)
\]

\subsection*{Variation operator on the complex}
The variation operator $\textrm{Var}$ acting on the chain complex.
\[ \textrm{Var} = 
\left(\begin{array}{cccccccccccccc}1 & 0 & 0 & 1 & 0 & 0 & -1 & 0 & 0 & 0 & 0 & 
-1 & 0 & 0\\1 & 1 & 0 & 0 & 0 & 0 & 0 & -1 & 0 & 0 & 0 & 0 & -1 & 0\\0 & 1 & 1 
& 0 & -1 & 0 & 0 & 0 & 0 & -1 & 0 & 0 & 0 & 0\\0 & 0 & 1 & 1 & 0 & -1 & 0 & 0 & 
0 & 0 & -1 & 0 & 0 & 0\\0 & 0 & 0 & 0 & 1 & 0 & 0 & 0 & -1 & 0 & 0 & 0 & 0 & 
0\\0 & 0 & -1 & 0 & 1 & 1 & 0 & 0 & 0 & 0 & 0 & 0 & 0 & 0\\0 & 0 & 0 & -1 & 0 & 
1 & 1 & 0 & 0 & 0 & 0 & 0 & 0 & 0\\-1 & 0 & 0 & 0 & 0 & 0 & 1 & 1 & 0 & 0 & 0 & 
0 & 0 & 0\\0 & -1 & 0 & 0 & 0 & 0 & 0 & 1 & 1 & 0 & 0 & 0 & 0 & 0\\0 & 0 & 0 & 
0 & 0 & 0 & 0 & 0 & 0 & 1 & 0 & 0 & 0 & -1\\0 & 0 & -1 & 0 & 0 & 0 & 0 & 0 & 0 
& 1 & 1 & 0 & 0 & 0\\0 & 0 & 0 & -1 & 0 & 0 & 0 & 0 & 0 & 0 & 1 & 1 & 0 & 0\\-1 
& 0 & 0 & 0 & 0 & 0 & 0 & 0 & 0 & 0 & 0 & 1 & 1 & 0\\0 & -1 & 0 & 0 & 0 & 0 & 0 
& 0 & 0 & 0 & 0 & 0 & 1 & 1\end{array}\right)
\]

\subsection*{Variation after change of basis}
The variation operator expressed in the basis given by the Smith Normal Form.
\[ W^{-1}\textrm{Var}(W^{-1})^T = 
\left(\begin{array}{cccccccccccccc}0 & 0 & 0 & 0 & 0 & 0 & 0 & 0 & 0 & 0 & 0 & 
0 & 0 & 0\\0 & 0 & 0 & 0 & 0 & 0 & 0 & 0 & 0 & 0 & 0 & 0 & 0 & 0\\0 & 0 & 0 & 0 
& 0 & 0 & 0 & 0 & 0 & 0 & 0 & 0 & 0 & 0\\0 & 0 & 0 & 1 & 1 & -1 & 0 & 0 & 0 & 0 
& -1 & 0 & 0 & 0\\0 & 0 & 0 & 0 & 1 & 0 & 0 & 0 & 0 & -1 & 0 & 0 & 0 & 0\\0 & 0 
& 0 & 0 & -1 & 1 & 0 & 0 & 0 & 0 & 0 & 0 & 0 & 0\\0 & 0 & 0 & -1 & 0 & 1 & 1 & 
0 & 0 & 0 & 0 & 0 & 0 & 0\\0 & 0 & 0 & 0 & 0 & 0 & 1 & 1 & 0 & 0 & 0 & 0 & 0 & 
0\\0 & 0 & 0 & 0 & 0 & 0 & 0 & 1 & 1 & 0 & 0 & 0 & 0 & 0\\0 & 0 & 0 & 0 & 0 & 0 
& 0 & 0 & 0 & 1 & 0 & 0 & 0 & -1\\0 & 0 & 0 & 0 & -1 & 0 & 0 & 0 & 0 & 1 & 1 & 
0 & 0 & 0\\0 & 0 & 0 & -1 & 0 & 0 & 0 & 0 & 0 & 0 & 1 & 1 & 0 & 0\\0 & 0 & 0 & 
0 & 0 & 0 & 0 & 0 & 0 & 0 & 0 & 1 & 1 & 0\\0 & 0 & 0 & 0 & 0 & 0 & 0 & 0 & 0 & 
0 & 0 & 0 & 1 & 1\end{array}\right)
\]

\subsection*{Variation in homology}
The restriction of the variation operator to the homology.
\[ \textrm{Var}_H = 
\left(\begin{array}{ccccccccccc}1 & 1 & -1 & 0 & 0 & 0 & 0 & -1 & 0 & 0 & 0\\0 
& 1 & 0 & 0 & 0 & 0 & -1 & 0 & 0 & 0 & 0\\0 & -1 & 1 & 0 & 0 & 0 & 0 & 0 & 0 & 
0 & 0\\-1 & 0 & 1 & 1 & 0 & 0 & 0 & 0 & 0 & 0 & 0\\0 & 0 & 0 & 1 & 1 & 0 & 0 & 
0 & 0 & 0 & 0\\0 & 0 & 0 & 0 & 1 & 1 & 0 & 0 & 0 & 0 & 0\\0 & 0 & 0 & 0 & 0 & 0 
& 1 & 0 & 0 & 0 & -1\\0 & -1 & 0 & 0 & 0 & 0 & 1 & 1 & 0 & 0 & 0\\-1 & 0 & 0 & 
0 & 0 & 0 & 0 & 1 & 1 & 0 & 0\\0 & 0 & 0 & 0 & 0 & 0 & 0 & 0 & 1 & 1 & 0\\0 & 0 
& 0 & 0 & 0 & 0 & 0 & 0 & 0 & 1 & 1\end{array}\right)
\]

\subsection*{Verifications}
One verifies, using the above results, that the identity
\[M \textrm{Var}^T - \textrm{Var} = 0\]
holds.
Restricting to homology, we verify the well known identity
\[ 
M_H \textrm{Var}_H^T - \textrm{Var}_H = 0
\]
Now we consider the differential matrix $DI$:
\[ DI = 
\left(\begin{array}{cccccccccccccc}0 & 1 & 0 & -1 & 1 & -1 & 0 & 0 & 1 & 0 & 0 
& -1 & 1 & 0\\-1 & 0 & 1 & 0 & 0 & 1 & -1 & 0 & 0 & 1 & 0 & 0 & -1 & 1\\0 & -1 
& 0 & 1 & 0 & 0 & 1 & -1 & 0 & -1 & 1 & 0 & 0 & -1\\1 & 0 & -1 & 0 & -1 & 0 & 0 
& 1 & -1 & 0 & -1 & 1 & 0 & 0\end{array}\right)
\]
and verify that monodromy takes cycles to cycles ($DI \cdot M \cdot  V$). The 
columns 
corresponding to the kernel should be, and indeed are, zero:
\[ DI \cdot M \cdot V = 
\left(\begin{array}{cccccccccccccc}0 & 1 & -1 & 0 & 0 & 0 & 0 & 0 & 0 & 0 & 0 & 
0 & 0 & 0\\1 & 0 & 0 & 0 & 0 & 0 & 0 & 0 & 0 & 0 & 0 & 0 & 0 & 0\\0 & -1 & 0 & 
0 & 0 & 0 & 0 & 0 & 0 & 0 & 0 & 0 & 0 & 0\\-1 & 0 & 1 & 0 & 0 & 0 & 0 & 0 & 0 & 
0 & 0 & 0 & 0 & 0\end{array}\right)
\]

\bibliographystyle{alpha}
\bibliography{bibliography}	

\end{document}